\documentclass{article}
\usepackage{graphicx} 
\usepackage[top=1in,bottom=1in,left=1in,right=1in]{geometry}  
\usepackage[pagebackref]{hyperref}
\usepackage{mathtools} 
\usepackage{amsmath, amssymb, amsthm}
\usepackage{mathrsfs} 
\usepackage{stmaryrd} 
\usepackage[skins]{tcolorbox}
\usepackage{tikz}
\usepackage{tikz-cd}
\usepackage[nameinlink,capitalize]{cleveref}
\usepackage{autonum}
\usepackage{titlesec}
\usepackage{mathdots}
\titleformat{\subsubsection}[runin]
    {\normalfont\bfseries}
    {\thesubsubsection}
    {0.5em}
    {}
    []

\newcommand{\symup}{\mathrm}
\newcommand{\symbf}{\mathbf}
\newcommand{\symbb}{\mathbb}
\newcommand{\symsf}{\mathsf}
\DeclareMathAlphabet{\mathbfscr}{OMS}{mdugm}{b}{n}

\newcommand{\symcal}{\mathcal}
\newcommand{\symscr}{\mathscr}
\newcommand{\symfrak}{\mathfrak}

\newcommand{\bA}{{\symbf{A}}}

\newcommand{\bC}{{\symbf{C}}}
\newcommand{\bD}{{\symbf{D}}}

\newcommand{\bF}{{\symbf{F}}}
\newcommand{\bG}{{\symbf{G}}}
\newcommand{\bH}{{\symbf{H}}}

\newcommand{\bJ}{{\symbf{J}}}

\newcommand{\bM}{{\symbf{M}}}

\newcommand{\bT}{{\symbf{T}}}

\newcommand{\bW}{{\symbf{W}}}

\newcommand{\bZ}{{\symbf{Z}}}

\newcommand{\bFe}{{\symbf{e}}}
\newcommand{\bFf}{{\symbf{f}}}
\newcommand{\bFg}{{\symbf{g}}}

\newcommand{\bbA}{{\symbb{A}}}
\newcommand{\bbB}{{\symbb{B}}}
\newcommand{\bbC}{{\symbb{C}}}

\newcommand{\bbN}{{\symbb{N}}}

\newcommand{\bbQ}{{\symbb{Q}}}

\newcommand{\bbZ}{{\symbb{Z}}}

\newcommand{\cA}{{\symcal{A}}}
\newcommand{\cB}{{\symcal{B}}}
\newcommand{\cC}{{\symcal{C}}}
\newcommand{\cD}{{\symcal{D}}}
\newcommand{\cE}{{\symcal{E}}}
\newcommand{\cF}{{\symcal{F}}}

\newcommand{\cH}{{\symcal{H}}}

\newcommand{\cK}{{\symcal{K}}}
\newcommand{\cL}{{\symcal{L}}}
\newcommand{\cM}{{\symcal{M}}}

\newcommand{\cO}{{\symcal{O}}}
\newcommand{\cP}{{\symcal{P}}}

\newcommand{\cR}{{\symcal{R}}}

\newcommand{\cT}{{\symcal{T}}}

\newcommand{\sM}{{\symscr{M}}}

\newcommand{\bsM}{{\symbfscr{M}}}

\newcommand{\FRd}{{\symfrak{d}}}

\newcommand{\FRp}{{\symfrak{p}}}

\newcommand{\FRA}{{\symfrak{A}}}

\newcommand{\FRC}{{\symfrak{C}}}
\newcommand{\FRD}{{\symfrak{D}}}

\newcommand{\FRJ}{{\symfrak{J}}}

\newcommand{\FRM}{{\symfrak{M}}}

\newcommand{\FRS}{{\symfrak{S}}}
\newcommand{\FRT}{{\symfrak{T}}}

\newcommand{\SFQ}{{\symsf{Q}}}

\renewcommand{\emptyset}{\varnothing}
\newcommand{\widebar}[1]{\mkern 2.0mu\overline{\mkern-2.0mu#1\mkern-2.0mu}\mkern 2.0mu}
\renewcommand{\bar}{\widebar}

\newcommand{\x}{\times}

\newcommand{\longot}{\longleftarrow}
\newcommand{\longto}{\longrightarrow}

\newcommand{\ar}{\rightarrow}

\newcommand{\olto}{\overleftarrow}
\newcommand{\orto}{\overrightarrow}

\newcommand{\defeq}{\coloneqq}

\newcommand{\notion}{\emph}
\mathchardef\mathhyphen="2D 
\newcommand{\hy}{\mathhyphen} 

\providecommand{\given}{}
\newcommand\SetSymbol[1][]{%
    \nonscript\:#1\vert
    \allowbreak
    \nonscript\:
    \mathopen{}}
\DeclarePairedDelimiterX{\Set}[1]{\{}{\}}{%
    \renewcommand\given{\SetSymbol[\delimsize]}
    #1
}
\DeclarePairedDelimiterX{\restr}[1]{}{\vert}{#1}
\DeclarePairedDelimiterX{\Stack}[1]{[}{]}{%
    #1
}

\newcommand{\OneHalf}{{\frac{1}{2}}}

\newcommand{\dd}{\symup{d}} 

\DeclarePairedDelimiterX{\abs}[1]{\lvert}{\rvert}{%
    \ifblank{#1}{\:\cdot\:}{#1}
}
\DeclarePairedDelimiterX{\norm}[1]{\lVert}{\rVert}{%
    \ifblank{#1}{\:\cdot\:}{#1}
}
\DeclarePairedDelimiterX{\lrangle}[1]{\langle}{\rangle}{%
    \ifblank{#1}{\:\cdot\:}{#1}
}
\DeclarePairedDelimiterX{\powser}[1]{[\![}{]\!]}{%
    \ifblank{#1}{\:\cdot\:}{#1}
}
\DeclarePairedDelimiterX{\lauser}[1]{(\!(}{)\!)}{%
    \ifblank{#1}{\:\cdot\:}{#1}
}

\newcommand{\One}{\symbf{1}}

\DeclareMathOperator{\Cnt}{\#} 
\DeclareMathOperator{\Hom}{Hom} 
\DeclareMathOperator{\End}{End} 
\DeclareMathOperator{\Aut}{Aut} 
\DeclareMathOperator{\Out}{Out} 
\DeclareMathOperator{\Tr}{Tr} 
\DeclareMathOperator{\Nm}{Nm} 
\newcommand{\ev}{\symup{ev}} 
\DeclareMathOperator{\coker}{coker} 

\DeclareMathOperator{\Sym}{Sym}
\newcommand{\Disc}{\symup{Disc}} 

\DeclareMathOperator{\RH}{H} 
\DeclareMathOperator{\PH}{\prescript{\FRp}{}{H}} 

\DeclareMathOperator{\Spec}{Spec}

\DeclareMathOperator{\codim}{codim} 
\newcommand{\CoTB}{\Omega} 
\DeclareMathOperator{\PicS}{\cP ic} 
\DeclareMathOperator{\Div}{Div} 
\newcommand{\Gm}[1][]{%
    \ifblank{#1}{%
        {\symbb{G}}_{\symup{m}}
    }{
        {\symbb{G}}_{\symup{m},#1}
    }
}

\DeclareMathOperator{\BG}{\bbB\!} 
\DeclareMathOperator{\Bun}{Bun} 
\newcommand{\git}{\mathbin{\!\sslash\!}} 
\newcommand{\Gr}{\symsf{Gr}} 
\newcommand{\Cartan}{\symsf{C}} 
\DeclareMathOperator{\IHom}{\underline{Hom}} 
\DeclareMathOperator{\Irr}{Irr} 
\newcommand{\Loop}{\symbb{L}} 
\newcommand{\Arc}[1][]{
    \ifblank{#1}{%
        \Loop^+
    }{
        \Loop^{+ #1}
    }
}
\newcommand{\IC}{{\symup{IC}}} 
\DeclareMathOperator{\Ob}{Ob} 

\DeclareMathOperator{\val}{val} 
\DeclareMathOperator{\Char}{char} 
\DeclareMathOperator{\Gal}{Gal} 

\newcommand{\Qlb}{\bar{\bbQ}_\ell}

\DeclareMathOperator{\Lie}{Lie} 
\DeclareMathOperator{\Ad}{Ad} 
\DeclareMathOperator{\ad}{ad} 
\DeclareMathOperator{\Cent}{C} 
\DeclareMathOperator{\Norm}{N} 
\newcommand{\tp}[1]{\prescript{\symup{t}}{}{#1}} 
\newcommand{\dual}{\check} 
\DeclareMathOperator{\GL}{GL} 
\DeclareMathOperator{\PGL}{PGL} 
\DeclareMathOperator{\SL}{SL} 
\DeclareMathOperator{\Mat}{Mat} 
\newcommand{\La}[1]{\symfrak{#1}} 
\DeclarePairedDelimiterX\Pair[2]{\langle}{\rangle}{#1,#2}
\DeclareMathOperator{\Diag}{diag} 

\newcommand{\Der}{{\symup{der}}} 
\newcommand{\CoCharG}{\check{\symbb{X}}} 
\newcommand{\Roots}{\Phi} 
\newcommand{\SimRts}{\Delta} 
\newcommand{\Rt}{\alpha} 
\newcommand{\Wt}{\varpi} 
\newcommand{\CoWt}{\check{\varpi}} 
\newcommand{\SC}{{\symup{sc}}} 
\newcommand{\AD}{{\symup{ad}}} 
\newcommand{\reg}{{\symup{reg}}} 
\newcommand{\ssim}{{\symup{ss}}} 
\newcommand{\srs}{{\symup{srs}}} 

\newcommand{\Sat}{\symup{Sat}} 

\newtcbox{\TODO}{enhanced,nobeforeafter,tcbox raise
    base,boxrule=0.4pt,top=-1pt,bottom=-1pt,
    right=.5mm,left=12mm,arc=1pt,boxsep=2pt,before upper={\vphantom{dlg}},
    colframe=red!75!black,coltext=red!50!black,colback=red!5!white,
    overlay={\begin{tcbclipinterior}\fill[red!75!black] (frame.south west)
    rectangle node[text=white,font=\sffamily\bfseries\footnotesize
    ] {TODO} ([xshift=11.5mm]frame.north west);\end{tcbclipinterior}}}

\robustify{\TODO}

\newtcolorbox{TODObox}{colback=red!5!white,colframe=red!75!black,coltext=red!50!black,fonttitle=\sffamily\bfseries,title=TODO}

\newtcbox{\FIXME}{enhanced,nobeforeafter,tcbox raise
    base,boxrule=0.4pt,top=-1pt,bottom=-1pt,
    right=.5mm,left=12mm,arc=1pt,boxsep=2pt,before upper={\vphantom{dlg}},
    colframe=orange!75!black,coltext=orange!50!black,colback=orange!5!white,
    overlay={\begin{tcbclipinterior}\fill[orange!75!black] (frame.south west)
    rectangle node[text=white,font=\sffamily\bfseries\footnotesize
    ] {FIXME} ([xshift=11.5mm]frame.north west);\end{tcbclipinterior}}}

\robustify{\FIXME}

\newtcolorbox{FIXMEbox}{colback=orange!5!white,colframe=orange!75!black,coltext=orange!50!black,fonttitle=\sffamily\bfseries,title=FIXME}

\definecolor{poop-brown}{RGB}{122,89,1}
\newtcbox{\POOP}{enhanced,nobeforeafter,tcbox raise
    base,boxrule=0.4pt,top=-1pt,bottom=-1pt,
    right=.5mm,left=12mm,arc=1pt,boxsep=2pt,before upper={\vphantom{dlg}},
    colframe=poop-brown!75!black,coltext=poop-brown!50!black,colback=poop-brown!5!white,
    overlay={\begin{tcbclipinterior}\fill[poop-brown!75!black] (frame.south west)
    rectangle node[text=white,font=\sffamily\bfseries\footnotesize
    ] {POOP} ([xshift=11.5mm]frame.north west);\end{tcbclipinterior}}}

\robustify{\POOP}

\newtcolorbox{POOPbox}{breakable,pad at break*=1mm,colback=poop-brown!5!white,colframe=poop-brown!75!black,coltext=poop-brown,fonttitle=\sffamily\bfseries,title=POOP}

\newtcbox{\COMMENT}{enhanced,nobeforeafter,tcbox raise
    base,boxrule=0.4pt,top=-1pt,bottom=-1pt,
    right=.5mm,left=18mm,arc=1pt,boxsep=2pt,before upper={\vphantom{dlg}},
    colframe=teal!75!black,coltext=teal!50!black,colback=teal!5!white,
    overlay={\begin{tcbclipinterior}\fill[teal!75!black] (frame.south west)
    rectangle node[text=white,font=\sffamily\bfseries\footnotesize
    ] {COMMENT} ([xshift=17.5mm]frame.north west);\end{tcbclipinterior}}}

\robustify{\COMMENT}

\newtcolorbox{COMMENTbox}{colback=teal!5!white,colframe=teal!75!black,coltext=teal!50!black,fonttitle=\sffamily\bfseries,title=COMMENT}

\newtcbox{\QUESTION}{enhanced,nobeforeafter,tcbox raise
    base,boxrule=0.4pt,top=-1pt,bottom=-1pt,
    right=.5mm,left=18mm,arc=1pt,boxsep=2pt,before upper={\vphantom{dlg}},
    colframe=yellow!75!black,coltext=yellow!50!black,colback=yellow!5!white,
    overlay={\begin{tcbclipinterior}\fill[yellow!75!black] (frame.south west)
    rectangle node[text=white,font=\sffamily\bfseries\footnotesize
    ] {QUESTION} ([xshift=17.5mm]frame.north west);\end{tcbclipinterior}}}

\robustify{\QUESTION}

\newtcolorbox{QUESTIONbox}{colback=yellow!5!white,colframe=yellow!75!black,coltext=yellow!50!black,fonttitle=\sffamily\bfseries,title=QUESTION}

\newtcbox{\ANSWER}{enhanced,nobeforeafter,tcbox raise
    base,boxrule=0.4pt,top=-1pt,bottom=-1pt,
    right=.5mm,left=15mm,arc=1pt,boxsep=2pt,before upper={\vphantom{dlg}},
    colframe=green!75!black,coltext=green!50!black,colback=green!5!white,
    overlay={\begin{tcbclipinterior}\fill[green!75!black] (frame.south west)
    rectangle node[text=white,font=\sffamily\bfseries\footnotesize
    ] {ANSWER} ([xshift=14.5mm]frame.north west);\end{tcbclipinterior}}}

\robustify{\ANSWER}

\newtcolorbox{ANSWERbox}{colback=green!5!white,colframe=green!75!black,coltext=green!50!black,fonttitle=\sffamily\bfseries,title=ANSWER}

\newtheorem{theorem}[subsubsection]{Theorem}
\newtheorem{proposition}[subsubsection]{Proposition}
\newtheorem{lemma}[subsubsection]{Lemma}
\newtheorem{corollary}[subsubsection]{Corollary}

\theoremstyle{definition}
\newtheorem{definition}[subsubsection]{Definition}

\newtheorem{example}[subsubsection]{Example}

\theoremstyle{remark}
\newtheorem{remark}[subsubsection]{Remark}

\numberwithin{equation}{subsection}

\Crefname{equation}{}{}
\crefname{equation}{}{}
\Crefname{section}{\S}{\S\S}
\crefname{section}{\S}{\S\S}

\AtBeginDocument{
    \newcommand{\blank}{{-}} 
    \newcommand{\UNI}{\symup{U}} 
    \newcommand{\SYMS}{\symup{S}} 
    \newcommand{\Hit}{{\symup{Hit}}} 
    \newcommand{\JR}{{\symup{JR}}} 
    \newcommand{\OGT}{\vartheta} 
    \newcommand{\SIM}{{\symup{sim}}} 
    \renewcommand{\bsM}{{\underline{\sM}}} 
    \newcommand{\OI}{{\symbf{O}}} 

}

\title{Relative Fundamental Lemmas for Spherical Hecke Algebras and
Multiplicative Hitchin Fibrations: the Jacquet--Rallis Case}
\author{X. Griffin Wang, Zhiyu Zhang}
\date{\today}

\begin{document}

\maketitle 

\begin{abstract}
     We prove the Jacquet--Rallis fundamental lemma for spherical Hecke algebras
     over local function fields using multiplicative Hitchin fibrations.
     Our work is inspired by the proof of \cite{Yu11} in the Lie algebra case and
     builds upon the general framework of multiplicative Hitchin fibrations in \cite{Wa25}.
\end{abstract}

\setcounter{tocdepth}{2}
\tableofcontents

\section{Introduction} 
\label{sec:introduction}

\subsection{Jacquet--Rallis fundamental lemmas for spherical Hecke algebras}

Due to the work of many people including \cite{zhang2014fourier, xue2019global, GGP21,GGP22}, the
Gan--Gross--Prosad conjecture on Rankin--Selberg \(L\)-functions \cite{GGP12} for
unitary groups is now established in full generality via the relative trace
formula approach. A key local ingredient is the Jacquet--Rallis fundamental lemma
\cite{Yu11,BP21}, which is an equality between orbital integrals of
standard test functions on the unitary group \(\UNI_n(F)\) and the symmetric
space \(\SYMS_n=\GL_n(E)/\GL_n(F)\), where \(E/F\) is an unramified quadratic extension
of local fields with odd residue characteristics.

The first proof of such a
relative fundamental lemma is due to Yun in \cite{Yu11}, which mainly studies its
Lie algebra variants over function fields. Recently, a Jacquet--Rallis
fundamental lemma for spherical Hecke algebras over \(p\)-adic fields was proved
by Leslie in \cite{Les23} via local harmonic analysis and global comparison, which is used in the
formulation of arithmetic fundamental lemmas for Hecke algebras
\cite{AFL,LRZ24}. 

In this paper, we prove the Jacquet--Rallis fundamental lemma for spherical Hecke
algebras over local function fields, via the method of multiplicative Hitchin
fibrations. In particular, we avoid the reduction to Lie algebra variants. The
main theorem of this paper is the following.

\begin{theorem}[Jacquet--Rallis fundamental lemma for spherical Hecke algebras, \Cref{thm:main_theorem}]
    Let \(n\ge 1\) be an integer and \(F=k\lauser{\pi}\) be a local function
    field with \(\Char(k)>2n\). For any strongly regular semisimple \(A\in
    \SYMS_n(F)\) matching \(A'\in G'(F)\) and any pair of spherical functions
    \((f,f')\in \cH_0\x\cH_0'\) matched under the canonical spherical transfer
    \eqref{eqn:spherical_transfer}, we have equality in orbital integrals
    \begin{align}
        \Delta(A)\OI_{A,H}^{\eta}(f)=\OI_{A',H'}(f').
    \end{align}
\end{theorem}

Here we consider orbital integrals for the adjoint action of \(H=\GL_{n-1}\)
(resp. unitary group \( H'=\UNI_{n-1}  \)) on the symmetric space \(\SYMS_n\) (resp.
\(G'=\UNI_{n}\)) for strongly regular
semisimple orbits (\Cref{def:strongly_regular_semisimple}), and \(\Delta(A) \in
\{ \pm 1\}\) is the transfer factor. See \Cref{thm:main_theorem} for the meaning
of all notations. If there is no \(A' \in G'(F)\) matching \(A \in \SYMS_n(F)\),
then we show \(\OI_{A,H}^{\eta}(f)=0\) (\Cref{lem:vanishing_odd_val}) by
establishing functional equations for orbital integrals.

\subsubsection{Applications}
We expect that our theorem will imply the spherical fundamental lemma over
\(p\)-adic fields for sufficiently large prime \(p\) via model theory, similar to
\cite{Yu11}. Our proof is orthogonal to the local trace formulas and the
existence of smooth transfers used in \cite{Les23}, which are currently
unavailable in positive characteristics. Moreover, our method should also be able
to prove the homogeneous version of spherical Jacquet--Rallis fundamental lemmas
\cite{Les23}\cite[Theorem 3.6.1]{LRZ24}, generalizing the above inhomogeneous
version of fundamental lemmas.

Our method is also somewhat different from \cite{Yu11}. In \textit{loc.~cit.},
the spectral cover plays a heavy role in the study of both local and global
geometry. In multiplicative case, there is no spectral cover, and so we make use
of the cameral cover, which is much better generalized to other types of
relative fundamental lemmas in the framework of the relative Langlands program
\cite{BZSV} on \(L\)-functions and period integrals.  For instance, we have a
twisted Jacquet--Rallis lemma \cite{W23} for the twisted Gan--Gross--Prasad
conjecture \cite{TGGP} and also a Guo--Jacquet fundamental lemma for linear
periods \cite{guo1996generalization}, which an adaptation of
our method may prove. Similarly, our work may shed insights on spherical
relative fundamental lemmas for \(\UNI_n \times \UNI_n\), which in the unit
function case is deduced in the work of Liu \cite{Liu-GGP-Bessel} from Yun's
proof \cite{Yu11}.

Moreover, we expect our work in this paper will be helpful for studying analogs of
arithmetic fundamental lemmas for spherical Hecke algebras over moduli of local
Shtukas, which will be in a future work. Since our proof is geometric, we work
with Satake functions, namely spherical functions induced by geometric Satake
(also known as Kazhdan--Lusztig basis). It would be interesting to consider
arithmetic fundamental lemmas regarding the Taylor expansions of the orbital
integrals (with respect to a twisting parameter \(s\in\bbC\))
of those Satake functions.

\subsection{Sketch of the proof} 
\label{sub:sketch_of_the_proof}

We give an outline of our proof. The overall strategy is pretty close to
the global method in \cite{Yu11}, with the Lie algebras replaced by reductive
monoids (or the analog for the symmetric space). There have been many
geometric proofs of various fundamental lemmas at this point, such as
\cite{Ng10,Yu11,Wa25}. Thus, instead of giving a comprehensive description of
the common ingredients, such as Grothendieck--Lefschetz trace formula and
perverse sheaves (for which we recommend the reader to consult Ng\^o's
excellent exposition \cite{Ng17}), we will try to highlight the ones that are
more unique to this paper.

\subsubsection{}
The first difficulty in adapting Yun's method to our case is that their proof
heavily utilizes the spectral cover, which does not exist in the multiplicative
case. We first study the invariant theory in the monoid case, which enables
us to find the correct analogs of various key ingredients in \cite{Yu11}.
Among them, a set of inductively defined pure tensors \(\bFf_i\) and co-tensors
\(\bFf_i^\vee\) are especially important.

In fact, \cite{Yu11} does not directly use the quotient stack
\(\Stack*{\La{g}/H}\), but rather a deformation of it. More precisely, since
\(H=\GL_{n-1}\) may be regarded as the subgroup of \(G=\GL_n\) that fixes a
vector \(\bFe\) and a co-vector \(\bFe^\vee\) such that the pairing
\(\bFe^\vee\bFe=1\), the classifying stack \(\BG{H}\) of \(H\) classifies
rank-\(n\) vector bundles \(E\) together with a vector \(e\in E\) and a
co-vector \(e^\vee\in E^\vee\) such that \(e^\vee e=1\). We deform
\(\BG{H}\) by eliminating the condition \(e^\vee e=1\). This further induces a
deformation of \(\Stack*{\La{g}/H}\), and an analog is
also available in the monoidal setting. We also study the invariant theory
regarding this deformed version (see \Cref{sub:deformed_quotient_stack}),
because the global argument becomes significantly
easier compared to using the non-deformed version.

\subsubsection{}
In the local setting, most of the results are relatively straightforward. Our goal
is to encode the orbital integrals as a certain kind of point-count problem of
the relative version of the multiplicative affine Springer fibers (which we will
call the affine Jacquet--Rallis fibers). The transfer factor can be similarly
geometrized using the pure co-tensor \(\bFf_n^\vee\) (or equivalently, the pure
tensor \(\bFf_n\)).

The main difference from the Lie algebra case is that we need to relate the
formulation using groups to the formulation using monoids. An interesting result
is a new monoid-based proof of the relative Cartan decomposition originally due
to \cite{Of04}. In particular, it relies on the geometry of the monoids as
spherical varieties and avoids concrete matrix computations in \textit{loc.
cit.} As such, we expect it to be better generalized for other types of relative
trace formulae.

Since not every conjugacy class on the symmetric side will match a conjugacy
class on the unitary side, there is a Galois obstruction class that is
\(2\)-torsion. We study this obstruction and show it to be the same as the
parity of the discriminant valuation (see \Cref{def:extended_disc_divisor}).

\subsubsection{}
In the global setting, the key object is what we
call the Jacquet--Rallis mH-fibration. It is a relative version of the
multiplicative Hitchin fibrations (mH-fibrations) developed in the work of the
first-named author in \cite{Wa25}. It is also the multiplicative analog of the
global object studied in \cite{Yu11}.

Similar to \cite{Yu11}, the study of the Jacquet--Rallis mH-fibration mainly
focuses on its deformation theory and topological properties, and the main result
is a support theorem enabled by a stronger version of smallness. Unlike the Lie
algebra case, the deformation theory is a lot more complicated due to the
presence of singularities. Fortunately, most of the difficulty lies in
understanding the singularities of the usual mH-fibrations, which is already
completed in \cite{Wa25}.

The proof of smallness (more precisely, stratified-smallness) follows the same
strategy as in \cite{Yu11} and is rather straightforward. In particular, the
main ingredient is the so-called \(\delta\)-regularity of the usual
mH-fibrations, which is studied in \cite{Wa25}.

A crucial recurring input in the arguments in \cite{Yu11} is the integrality of
the spectral curve. Translating to the multiplicative case, it is essentially a
global ellipticity condition (see \Cref{sub:the_simple_locus}). The pure tensors
\(\bFf_i\) and \(\bFf_i^\vee\) from invariant theory play key roles here.

\subsubsection{}
Finally, to deduce the fundamental lemma from the global support theorem, we
need a product formula, which is not difficult to prove. After that, there
are two main steps.

For one, we need to prove the fundamental lemma in some especially simple cases
by direct computations. This step is not very difficult in terms of
computations, but the classification of such cases is given by results in
\cite{Wa25} which we will not elaborate on in this introduction.

Another more important step is to find a global approximation of any given
local case at a given place such that in other local places, the relevant
point-count is non-zero. It involves resolving a global Galois obstruction
which turns out to be the product of the local matching obstructions
discussed above.


\subsection{Structure of the paper} 
\label{sub:structure_of_the_paper}

Finally, we give a summary of the paper.

In \Cref{sec:invariant_theory}, we
study the relative invariant theory of reductive monoids as well as their
symmetric and unitary versions. In particular, we also study the deformed
version.

In \Cref{sec:local_formulations}, we formulate the affine
Jacquet--Rallis fibers, which are geometric incarnations of orbital integrals.
After some generalities, we study the local Galois obstructions to the matching
of conjugacy classes here. We also compute several simple cases of orbital
integrals and verify the fundamental lemmas in such cases.

In \Cref{sec:global_formulations}, we study the Jacquet--Rallis mH-fibrations and
establish the local model of singularity via deformation theory, the
local-global product formula, and the stratified smallness under mild technical
conditions. As a consequence, we deduce the support theorem for relevant
perverse sheaves.

In \Cref{sec:the_proof_of_the_fundamental_lemma}, we finish
the proof of the fundamental lemma using global matching. Along the way, we have
an approximation result and resolve a crucial global Galois obstruction.

In Appendix~\ref{sec:supplement_on_the_usual_mh_fibrations}, we slightly generalize many
necessary results regarding the usual mH-fibrations in \cite{Wa25} to our setup
of symmetric spaces (which are not groups) and unitary groups (which are
quasi-split only locally but not globally). 


\subsection{Acknowledgement}
We thank Zhiwei Yun and Wei Zhang for helpful discussions and comments. We also
thank the anonymous referee for carefully reading the manuscript, catching many
errors and typos, and providing helpful suggestions.

\subsection{Notations}

Here, we give a list of the most frequently used notations and conventions. It is for
reference purposes and readers' convenience only. The reader should refer to the
relevant part of the paper for detailed definitions.

\begin{itemize}
    \item \(k\): a finite field with \(\Char(k)>2n\);
    \item \(X\): a smooth, projective, geometrically connected curve over \(k\);
    \item \(\OGT\colon X'\to X\): an \'etale double cover of \(X\) with \(X'\)
        also geometrically connected;
    \item \(F\) (resp.~\(F_v\), resp.~\(\cO_v\)): the function field
        (resp.~local field, resp.~valuation ring at a place \(v\)) of \(X\);
    \item \(\breve{F}_{\bar{v}}\) (resp.~\(\breve{\cO}_{\bar{v}}\)): the
        unramified closure of \(F_v\) (resp.~\(\cO_v\)) at an algebraic
        geometric point \(\bar{v}\) over \(v\);
    \item Boldface letters: groups, monoids, etc. that are split. E.g., \(\bG=\GL_n\),
        \(\bH=\GL_{n-1}\), and so on;
    \item Boldface letters with a prime (\(\prime\)): Weil restrictions of
        split groups, monoids, etc. through double cover \(\OGT\). E.g.,
        \(\bG'=\OGT_*\bG\), \(\bM'=\OGT_*\bM\), and so on;
    \item Upright, italic or fraktur letters: usually related to
        constructions on the symmetric space side. E.g., \(\SYMS_n\),
        \(G=\GL_n\), \(\FRM\), and so on;
    \item Upright, italic, or fraktur letters with a prime (\(\prime\)): usually
        related to constructions on the unitary side. E.g., \(G'=\UNI_n\),
        \(\FRM'\), and so on;
    \item Overhead bar (\(\bar{\phantom{x}}\)): Frobenius conjugate induced by a
        quadratic extension when the object is a number or a matrix;
    \item \(\sigma\) (resp.~\(\sigma_v\), \(\sigma_{\Out}\)): Frobenius
        conjugate (resp.~at a place \(v\), resp.~related to outer automorphisms)
        in more general contexts, whose meaning may vary slightly;
    \item Super/subscript \(\SC\) (resp.~\(\AD\)): constructions related to
        simply-connected cover (resp.~adjoint quotient);
    \item \(\Stack{S/G}\) (resp.~\(S\git G\)): the stacky (resp.~GIT) quotient
        of a space (or stack) by a group \(G\);
    \item \(\bM\) (resp.~\(\FRM\), resp.~\(\FRM'\)): the universal monoid (or
        monoidal symmetric space) of
        \(\bG^\SC=\SL_n\) (resp.~\(\SYMS_n\), resp.~\(G'=\UNI_n\));
    \item \(\bA_\bM\) (resp.~\(\FRA_\FRM\)): the abelianization of \(\bM\)
        (resp.~both \(\FRM\) and \(\FRM'\));
    \item \(\bC_\bM\) (resp.~\(\FRC_\FRM\)): the GIT quotient \(\bM\git\bG\)
        (resp.~\(\FRM\git G\simeq \FRM'\git G'\));
    \item \(\bC_{\bM,\bH}\) (resp.~\(\FRC_{\FRM,H}\)): the GIT quotient
        \(\bM\git \bH\) (resp.~\(\FRM\git H\simeq \FRM'\git H'\));
    \item \(\bsM\) (resp.~\(\sM\), resp.~\(\sM'\)): the deformed version of
        quotient stack \(\Stack{\bM/\bH}\) (resp.~\(\Stack{\FRM/H}\),
        resp.~\(\Stack{\FRM'/H'}\));
    \item \(\bC_{\bsM}\) (resp.~\(\FRC_\sM\)): the affinization of \(\bsM\)
        (resp.~both \(\sM\) and \(\sM'\));
    \item \(\bFf_i\) (resp.~\(\bFf_i^\vee\)): the \(i\)-th pure tensor (resp.~co-tensor)
    \item \(\Disc\): the \(\bH\)-discriminant function \(\bFf_n^\vee\bFf_n\);
    \item \(\cM_{v}(a)\) (resp. \(\cM_{v}'(a)\)): the affine
        Jacquet--Rallis fiber on the symmetric (unitary) side associated with a
        point \(a\in\FRC_\sM(\cO_v)\) that is generically strongly regular
        semisimple;
    \item \(\cM_v^\Hit(a)\) (resp. \(\cM_v^{\prime\Hit}(a)\)) the
        multiplicative affine Springer fiber (MASF) on the symmetric
        (resp.~unitary) side;
    \item \(\cH\) (resp.~\(\cH'\)): smooth \(\Qlb\)-valued
        functions with compact support on \(\SYMS_n(F_v)\) (resp.~\(G'(F_v)\));
    \item \(\cH_0 \subset \cH\) (resp. \(\cH_0'\subset\cH'\)): the subspace of
        spherical functions;
    \item \(h\colon \cM \to \cA\) (resp. \(h'\colon \cM' \to \cA\)): the
        Jacquet--Rallis mH-fibration on the symmetric (resp.~unitary) side;
    \item \(h^\Hit\colon \cM^\Hit \to \cA^\Hit \) (resp. \(h^{\prime \Hit}\colon
        \cM^{\prime \Hit} \to \cA^\Hit \)): the usual mH-fibration on the
        symmetric (resp.~unitary) side.
    \item \(\cB\): the moduli stack of boundary divisors;
    \item \(\cA^\heartsuit\) (resp.~\(\cA^\dagger\), \(\cA^\SIM\),
        \(\cA^\ddagger\), \(\cA_{(\delta)}\)): the open subset of \(\cA\) where:
        \begin{itemize}
            \item the generic point of \(X\) is mapped to the strongly regular
                semisimple locus,
            \item resp.~the local model of singularity in
                \Cref{thm:topological_local_model_of_singularity} holds for the
                \emph{usual} mHiggs bundle,
            \item resp.~the \emph{usual} mHiggs bundle is simple in the sense of
                \Cref{sub:the_simple_locus}
            \item resp.~both conditions for \(\cA^\SIM\) and
                \(\cA^\dagger\) are satisfied,
            \item resp.~all necessary technical conditions (some dependent on a
                fixed integer \(\delta\)) are satisfied, including ones for
                \(\cA^\dagger\), \(\cA^\SIM\), and notably also the
                \(\delta\)-regularity condition;
        \end{itemize}
    \item \(\cM^\heartsuit\), \(\cM^\dagger\), etc.
        (resp.~\(\cM^{\prime\heartsuit}\), \(\cM^{\prime\dagger}\), etc.): the preimage of
        \(\cA^{\heartsuit}\), \(\cA^\dagger\), etc. under \(h\) (resp.~\(h'\));
    \item \(\cF^{\lambda_v}\) (resp.~\(\cF^{\prime\lambda_v}\)): the
        local Satake sheaf of \(\SYMS_n^\AD\) (resp.~\((G')^\AD\)) with
        highest coweight \( \lambda_v\);
    \item \(L_\eta^X\) (resp.~\(L_{v,\eta}\)): the local system that is the
        global (resp.~local) geometric incarnation of the quadratic character
        \(\eta\).
\end{itemize}


\section{Monoids and Relative Invariant Theory} 
\label{sec:invariant_theory}

\subsection{The setup} 
\label{sub:the_setup}

Let \(k\) be a finite field and \(\bG=\GL_n\) be the general linear group over
\(k\). Let \(\bT \subset \bG\) be the maximal torus of \(\bG\) given by diagonal
matrices. Suppose \(\Char(k)>2n\). Then we have the universal 
monoid \(\bM\) of \(\bG^\SC=\SL_n\). The unit group of \(\bM\) is the
reductive group
\begin{align}
    \bM^\x=\bG^{\SC}_+=\bigl(\bT^\SC\x \bG^\SC\bigr)/\bZ^\SC,
\end{align}
where the center \(\bZ^\SC\subset \bG^\SC\) acts anti-diagonally. The group
\(\bG^\SC\x\bG^\SC\) acts on \(\bM\) by left and right translations, which
induces the \notion{adjoint action} of \(\bG^\SC\) by restricting to the
diagonal. Note that the adjoint action factors through \(\bG^\AD\) and so lifts
to an adjoint action of \(\bG\). We fix a maximal torus \(\bT_\bM\) of
\(\bM^\x\) to be \(\bigl(\bT^\SC\x\bT^\SC\bigr)/\bZ^\SC\), whose closure
\(\bar{\bT}_\bM\) in \(\bM\) is a normal toric variety. Let \(\bZ_\bM\simeq
\bT^\SC\) be the center of \(\bM^\x\).

Let \(\bH=\GL_{n-1}\subset\bG\) be the subgroup embedded into the upper-left
block. We may view \(\bH\) as the stabilizer of the vector
\begin{align}
    \bFe=\tp{(0,\ldots,0,1)}\in k^n
\end{align}
and the co-vector
\begin{align}
    \bFe^\vee=(0,\ldots,0,1)\in (k^n)^\vee.
\end{align}
The adjoint action of \(\bG\) on \(\bM\) induces an action of \(\bH\) on
\(\bM\), still called the adjoint action.
We are interested in the invariant theory of this \(\bH\)-action.


\subsection{Review of the usual adjoint action} 
\label{sub:review_of_the_usual_adjoint_action}

It will be helpful to first review the well-understood \(\bG\)-action on \(\bM\)
before studying the \(\bH\)-action; see \cite[Chapter~2]{Wa25} for more details.

\subsubsection{}
We can express \(\bM\) in a more
explicit way using fundamental representations of \(\bG^\SC\). Let \(\Rt_1,
\hdots, \Rt_{n-1}\) denote the simple roots and \(\Wt_1, \hdots, \Wt_{n-1}\) the fundamental weights of \(\bG\). Since \(\bG=\GL_n\), the
fundamental representations are simply \(V_i=\wedge^ik^n\)
for \(1\le i\le n-1\). We have the following representation of \(\bM^\x\)
\begin{align}
    \rho_\star\colon\bM^\x=\bigl(\bT^\SC\x \bG^\SC\bigr)/\bZ^\SC &\longto
    \bbA^{n-1}\oplus\bigoplus_{i=1}^{n-1}\End(V_i)\\
    (z,g)&\longmapsto
    (\Rt_i(z),\Wt_i(z)\wedge^i g),
\end{align}
and \(\bM\) is the normalization of the closure of \(\rho_\star(\bM^\x)\).

\subsubsection{}
We have the abelianization map of the reductive monoid
\begin{align}
    \alpha_\bM\colon \bM\to \bA_\bM\defeq \bM\git (\bG^\SC\x
    \bG^\SC)\simeq\bbA^{n-1},
\end{align}
which factors through the projection
\begin{align}
    \bbA^{n-1}\oplus\bigoplus_{i=1}^{n-1}\End(V_i)\longto \bbA^{n-1}.
\end{align}
Note that \(\bA_\bM\) is a commutative monoid and \(\alpha_\bM\) is a monoid
homomorphism. We know that \(\alpha_\bM^{-1}(\bA_\bM^\x)=\bM^\x\) and \(\alpha_\bM^{-1}(1)=\bG^\SC\).

For any \(x\in\bM\), we also define \(x_i\in\End(V_i)\) to be its
image under the obvious projection maps. By the general theory of reductive monoids, 
\(\bC_\bM\defeq\bM\git\bG\) is isomorphic to affine space \(\bbA^{2n-2}\) with
coordinates given by
\begin{align}
    \Rt_1,\ldots,\Rt_{n-1},\chi_1,\ldots,\chi_{n-1},
\end{align}
where \(\Rt_i\) are treated as the coordinates on \(\bA_\bM\) and \(\chi_i\)
sends \(x\) to \(\Tr(x_i)\). We denote by
\(\chi_\bM\) the natural map \(\bM\to\bC_\bM\). Let \(\bC\cong\bbA^{n-1}\) be
the affine space with coordinates \(\chi_i\), then \(\bC_\bM\simeq
\bA_\bM\x\bC\).
We also have a Chevalley--Steinberg type isomorphism (\(\bW\) is the Weyl group)
\begin{align}
    \bT_\bM\git\bW\stackrel{\sim}{\longto}\bC_\bM.
\end{align}

\subsubsection{}
\label{ssub:usual_disc}
There is an \notion{extended discriminant function} \(\Disc_+\) on \(\bC_\bM\) defined as
follows: on the maximal torus \(\bT_\bM\), it has the formula
\begin{align}
    \Disc_+=e^{(2\rho,0)}\prod_{\Rt\in\Roots}(1-e^{(0,\Rt)}),
\end{align}
where \(\Rt\) ranges over the set of roots \(\Roots\) of \(\bG\).
It extends to a regular function on \(\bar{\bT}_\bM\), and by its
\(\bW\)-invariance, it descends to a function on \(\bC_\bM\) and cuts out a
reduced divisor on \(\bC_\bM.\)

\subsubsection{}
There is a natural section from \(\bA_\bM\) to \(\bM\) denoted as \(\delta_\bM\). On
\(\bA_\bM^\x\simeq\bT^\AD\), this section is given by be \(z\mapsto (z,z^{-1})\).
Similarly, the invariant map \(\chi_\bM\) also admits a section
\(\epsilon_\bM\), whose restriction to \(\bC_\bM^\x\simeq\bA_\bM^\x\x\bC\) is
given by the formula
\begin{align}
    (z,a)\mapsto \delta_\bM(z)\epsilon(a),
\end{align}
where \(\epsilon(a)\in\bG^\SC\) is the companion matrix associated with
\(a=(a_1,\ldots,a_{n-1})\):
\begin{align}
    \epsilon(a)=\begin{pmatrix}
        a_1 & -a_2 & \cdots & (-1)^{n-2}a_{n-1} & (-1)^{n-1}\\
        1 &  &  &  & \\
          & 1 &  &  & \\
          &  & \ddots &  & \\
          & & & 1 & 
    \end{pmatrix}.
\end{align}
See \cite[\S~2.3.16 and Proposition~2.4.4]{Wa25} for details.


\subsection{The \texorpdfstring{\(\bH\)}{H}-action and an explicit section from the GIT quotient} 
\label{sub:the_bh_action}

Let \(\chi_{\bM,\bH}\colon \bM\to\bM\git\bH\) be the invariant quotient, then it
maps further down to \(\bC_\bM\) and \(\bA_\bM\) by its universal property.
In this subsection, we show that \(\bC_{\bM,\bH}:=\bM \git \bH \) is isomorphic to an affine space
of dimension \(3n-3\), and that \(\chi_{\bM,\bH}\) admits a section.

\subsubsection{}
The vector \(\bFe\) and the co-vector \(\bFe^\vee\) decomposes \(k^n\)
into a direct sum \((k\bFe^\vee)^\perp\oplus k\bFe\), where \((k\bFe^\vee)^\perp\simeq
k^{n-1}\). As a result, we can decompose \(V_i=\wedge^i k^n\) (\(1\le i\le
n-1\)) as \( \bH \)-representations
\begin{align}
    \wedge^i k^n=\wedge^i k^{n-1}\oplus(\wedge^{i-1}k^{n-1}\otimes k\bFe),
\end{align}
and we denote
\begin{align}
    V_i'&=\wedge^i k^{n-1},\\
    V_i''&=\wedge^{i-1}k^{n-1}\otimes k\bFe.
\end{align}
In turn, for any \(x_i\in\End(V_i)\), its trace is the sum
\begin{align}
    \Tr(x_i)=\Tr'(x_i)+\Tr''(x_i),
\end{align}
where \(\Tr'\) (resp.~\(\Tr''\)) is the induced trace on the direct summand
\(V_i'\) (resp.~\(V_i''\)). Clearly, the functions
\begin{align}
    \chi_i'\colon x &\longmapsto\Tr'(x_i),\\
    \chi_i''\colon x&\longmapsto\Tr''(x_i)
\end{align}
on \(\bM\) are \(\bH\)-invariant, and we have \(\chi_i=\chi_i'+\chi_i''\).

\subsubsection{}
For a tuple \(b=(b_1,\ldots,b_{n-1})\in\bbA^{n-1}\), we define
\begin{align}
    \beta(b)\defeq\begin{pmatrix}
        1 & & & & \\
          &1& & & \\
          & & \ddots & & \\
          & & & 1 & \\
        (-1)^{n-1}b_1 & (-1)^{n-2}b_2 &\cdots & -b_{n-1} & 1
    \end{pmatrix}.
\end{align}
We can now state one of the main results of this section:

\begin{theorem}
    The map
    \begin{align}
        \tilde{\chi}_{\bM,\bH}\colon\bM&\longto\bA_\bM\x\bbA^{2n-2}\\
        x&\longmapsto \bigl(\alpha_\bM(x),\chi_i'(x),\chi_i''(x)\bigr)
    \end{align}
    induces an isomorphism \(\bC_{\bM,\bH}\simeq
    \bA_\bM\x\bbA^{2n-2}\simeq\bbA^{3n-3}\). Moreover, we have a section
    \begin{align}
        \epsilon_{\bM,\bH}\colon \bC_{\bM,\bH}&\longto \bM\\
        (z,a',a'')&\longmapsto
        \beta(a'')\epsilon_\bM(z,a)\beta(a'')^{-1},
    \end{align}
    where \(a_i'\) (resp.~\(a_i''\)) is the coordinate corresponding to
    \(\chi_i'\) (resp.~\(\chi_i''\)), \(a'=(a_i')\) (resp.~\(a''=(a_i'')\)), and
    \(a=(a_i)\) is such that \(a_i=a_i'+a_i''\).
\end{theorem}
\begin{proof}
    It is not hard to see (and well-known) that the generic stabilizer of
    \(\bH\) acting on \(\bM^\x\) is trivial. Indeed, for
    a sufficiently general regular semisimple (in the usual sense) element
    \(x\in\bG^\SC\), the vectors
    \begin{align}
        \bFe,x\bFe,\ldots,x^{n-1}\bFe
    \end{align}
    are linearly independent. If \(h\in\bH\) stabilizes \(x\), then
    it fixes the above vectors as well, and so must be trivial. The same is
    true if \(x\in\bM^\x\) because the \(\bH\)-action commutes with central
    translations.
    This implies that the map \(\bC_{\bM,\bH}\to \bbA^{3n-3}\) is dominant and
    generically quasi-finite after combining with the dimension counting:
    \begin{align}
        \dim\bM-\dim\bbA^{3n-3}=(n^2+n-2)-(3n-3)=n^2-2n+1=\dim{\bH}.
    \end{align}
    Therefore, if we can construct a section from \(\bbA^{3n-3}\)
    to \(\bM\), then we have an induced section \(\bbA^{3n-3}\to\bC_{\bM,\bH}\),
    which will prove both claims since \(\bC_{\bM,\bH}\) is irreducible.

    Note that the map \(\epsilon_{\bM,\bH}\) is well-defined on \(\bbA^{3n-3}\),
    and we show that
    \(\tilde{\chi}_{\bM,\bH}\circ\epsilon_{\bM,\bH}\) is the identity map.
    To simplify notations, we let \(A=\epsilon(a)\) and \(B=\beta(a'')\). We
    already know that the trace of \(\epsilon_\bM(z,a)=\delta_\bM(z)A\) on
    \(V_j\) is \(a_j\), and conjugation by \(B\) does not change it. So we only
    need to compute \(\chi_j'\) (or \(\chi_j''\) if we so choose). It is also
    easy to compute
    \begin{align}
        AB^{-1}=\begin{pmatrix}
            a_1' & -a_2' & \cdots &
            (-1)^{n-2}a_{n-1}' & (-1)^{n-1}\\
            1 &  &  &  & \\
              & 1 &  &  & \\
              &  & \ddots &  & \\
              & & & 1 & 
        \end{pmatrix}.
    \end{align}
    The matrix coefficients of \(AB^{-1}\) on \(V_j\), up to signs, are just
    its \(j\x j\) minors, and the diagonal matrix
    coefficients correspond to minors such that if the \(i\)-th row is chosen,
    then so is the \(i\)-th column.

    The action of \(\delta_\bM(z)\) on \(V_j\) is described as follows:
    the highest weight of \(V_j\) is \(\Wt_j\), a basis vector of which is
    \(\bFe_{\Wt_j}=\bFe_1\wedge\cdots\wedge\bFe_j\), where \(\bFe_i\) is the \(i\)-th
    standard vector
    \begin{align}
        \bFe_i=\tp{(0,\ldots,1,\ldots,0)}
    \end{align}
    with \(1\) in the \(i\)-th row. Then
    \(\delta_\bM(z)\bFe_{\Wt_j}=\bFe_{\Wt_j}\), and for any vector \(e_\mu\) of weight
    \(\mu=\Wt_j-\sum_ic_i\Rt_i\) (\(c_i\in\bbN\)), we have
    \begin{align}
        \delta_\bM(z)e_\mu=z_1^{c_1}\cdots z_{n-1}^{c_{n-1}}e_\mu,
    \end{align}
    where \(z_i=\Rt_i(z)\in k\).
    In addition, the action of (unipotent matrix) \(B\) on \(V_j\) sends
    any pure wedge
    \begin{align}
        \bFe_{i_1}\wedge\cdots\wedge\bFe_{i_j}, \quad 1\le i_1<\cdots<i_j\le n
    \end{align}
    to a sum of itself
    plus some wedges involving \(\bFe_n=\bFe\), the latter of which is contained
    in the complement space \(V_j''\). This means that
    \begin{align}
        \chi_j'\bigl(B\delta_\bM(z)AB^{-1}\bigr)=\chi_j'\bigl(\delta_\bM(z)AB^{-1}\bigr).
    \end{align}
    Since the action of \(\delta_\bM(z)\) is diagonal,
    it suffices to compute the diagonal matrix coefficients of \(AB^{-1}\) and
    then scale them using \(\delta_\bM(z)\). To this end, we need only consider the
    upper-left \((n-1)\x(n-1)\) block of \(AB^{-1}\).

    For any \(2\le i\le n-1\), if the \(i\)-th row of \(AB^{-1}\) (without the
    last row and column) is chosen as part of
    a \(j\x j\) minor, then for the minor to be non-vanishing (and contributing
    to \(\chi_j'\)), the \((i-1)\)-th
    column must also be chosen, and then so is the \((i-1)\)-th row. By
    induction, the only \(j\x j\) minor contributing to \(\chi_j'\) is the
    upper-left \(j\x j\) block,
    corresponding to highest-weight \(\Wt_j\). This means that the scaling by
    \(\delta_\bM(z)\) is \(1\). As a result,
    \begin{align}
        \chi_j'\circ\epsilon_{\bM,\bH}(z,a',a'')=\chi_j'(AB^{-1})=a_j'.
    \end{align}
    This finishes the proof.
\end{proof}


\subsection{Some examples} 
\label{sub:some_examples}

Here we provide some examples to help the reader digest the general result. We
will compare them with the Lie algebra case (in other words, the case where
\(\bH\) acting on \(\bFg=\Lie(\bG)\)).

\begin{example}
    When \(n=2\), the monoid \(\bM\) is just \(\Mat_2\). In this case the
    abelianization is \(\bA_\bM=\bbA^1\) given by the usual determinant map, and
    the other two invariants are given by
    \begin{align}
        \chi'\begin{pmatrix}
            a & b\\
            c & d
        \end{pmatrix}&=a,\\
        \chi''\begin{pmatrix}
            a & b\\
            c & d
        \end{pmatrix}&= d.
    \end{align}
    Here \(\bFg=\bM\) by coincidence, and the ``traditional'' invariants of
    \(x\in\bFg\) are usually given by
    \begin{align}
        \det(x)&=ad-bc,\\
        \Tr(x)&=a+d,\\
        \bFe^\vee x\bFe &=d,
    \end{align}
    which are isomorphic, albeit not identical, to the invariants given by the
    multiplicative formulation.
\end{example}

\begin{example}
    When \(n=3\), an element \((z,g)\) in \(\bM^\x\) can be represented by
    matrices
    \begin{align}
        z &= \Diag\bigl(t_1,t_2,(t_1t_2)^{-1}\bigr),\\
        g
        & = \begin{pmatrix}
            A & B & C\\
            D & E & F\\
            G & H & I
        \end{pmatrix},\quad \det(g)=1.
    \end{align}
    The abelianization map sends this element to
    \begin{align}
        (z_1,z_2)\defeq (t_1/t_2,t_1t_2^2),
    \end{align}
    and the rest of the invariants are:
    \begin{align}
        a_1'&=\chi_1'(z,g)=t_1(A+E),\\
        a_1''&=\chi_1''(z,g)=t_1I,\\
        a_2'&=\chi_2'(z,g)=t_1t_2(AE-BD),\\
        a_2''&=\chi_2''(z,g)=t_1t_2(AI-CG+EI-FH).
    \end{align}
    In this case, \(\bM\) is not \(\bFg=\Mat_3\), but we still have a
    well-defined embedding \(\Mat_3\subset\bM\) identifying \(\Mat_3\) with the
    preimage under \(\alpha_\bM\) of the submonoid
    \begin{align}
        \Set{1}\x\bbA^1\subset\bA_\bM.
    \end{align}
    When restricted to \(\bM^\x\), it is equivalent to setting \(t_1=t_2=t\) for
    \((z,g)\), and it corresponds to the matrix
    \begin{align}
        x=\begin{pmatrix}
            tA & tB & tC\\
            tD & tE & tF\\
            tG & tH & tI
        \end{pmatrix}\in \Mat_3.
    \end{align}
    The ``traditional'' set of invariants are:
    \begin{align}
    \det(x) &=t^3,\\
    \Tr(x) &=t(A+E+I),\\
    \bFe^\vee x\bFe &= tI,\\
    \Tr(\wedge^2 x) &=t^2(AE-BD+AI-CG+EI-FH),\\
    \bFe^\vee x^2\bFe &= t^2(CG+FH+I^2).
    \end{align}
    The ones given by the multiplicative formulation above (by setting
    \(t_1=t_2=t\)) are:
    \begin{align}
    z_1&=1,\\
    z_2&=t^3,\\
    a_1'&=t(A+E),\\
    a_1''&=tI,\\
    a_2'&=t^2(AI-CG+EI-FH),\\
    a_2''&=t^2(AE-BD).
    \end{align}
    The reader can verify that these two sets of invariants are isomorphic.
\end{example}

\begin{remark}
    For general \(n>3\), we still have \(\Mat_n\subset\bM\) as a
    closed submonoid, whose abelianization is the submonoid
    \begin{align}
        \Set{(1,\ldots,1)}\x\bbA^1\subset\bA_\bM,
    \end{align}
    where the non-trivial coordinate is given by the usual determinant.
    One can then compare the respective invariants on \(\Mat_n\) given by the
    Lie algebra formulation and the multiplicative one.
\end{remark}


\subsection{The discriminant divisor} 
\label{sub:the_discriminant_divisor}

We are going to define a principal divisor \(\bD_{\bM,\bH}^+\) on
\(\bC_{\bM,\bH}\), whose complement captures those \(\bH\)-orbits that are
especially nice.

\subsubsection{}
First, we have the following function on \(\bM^\x\):
\begin{align}
    \Disc_{\bM,\bH}\colon (z,g)\longmapsto z^{2\rho}\det (\bFe^\vee
    g^{i+j-2}\bFe)_{1\le i,j\le n},
\end{align}
where \(\rho\) is the half-sum of all positive roots, viewed as a function on
\(\bT^\SC\) written exponentially. One may check that this function is indeed
well-defined: indeed, if
\(\zeta\in\bZ^\SC\), then \((z\zeta,\zeta^{-1}g)=(z,g)\), but we also have
\(z^{2\rho}=(z\zeta)^{2\rho}\) and (we omitted the ranges \(1\le i,j\le n\)
below to avoid notation overload)
\begin{align}
    \det(\bFe^\vee(\zeta^{-1}g)^{i+j-2}\bFe)
    &=(\bFe^\vee\zeta^{-1}\bFe)^{n(n-1)}\det(\bFe^\vee g^{i+j-2}\bFe)\\
    &=(\bFe^\vee\zeta^{-n(n-1)}\bFe)\det(\bFe^\vee g^{i+j-2}\bFe)\\
    &=\det(\bFe^\vee g^{i+j-2}\bFe),
\end{align}
because \(\zeta^n=1\).

\subsubsection{}
To extend this function to all of \(\bM\), we make the following construction:
recall that for any \(x\in \bM\) and any \(1\le i\le n-1\), we have the
element \(x_i\in\End(V_i)\). Starting from \(\bFf_1(x)=x_1\bFe\), we define
for \(2\le i<n\)
\begin{align}
    \bFf_i(x)=x_i(\bFe\wedge \bFf_{i-1}(x))\in V_i,
\end{align}
and let \(\bFf_n(x)=\bFe\wedge\bFf_{n-1}(x)\).
Similarly, we have \(\bFf_1^\vee(x)=(\bFe^\vee)x_1\),
\begin{align}
    \bFf_i^\vee(x)=(\bFf_{i-1}^\vee(x)\wedge\bFe^\vee)x_i\in V_i^\vee,
\end{align}
and \(\bFf_n^\vee(x)=\bFf_{n-1}^\vee(x)\wedge\bFe^\vee\).

\begin{lemma}
    For any \(x\in\bM\) and any \(i\), both \(\bFf_i(x)\) and \(\bFf_i^\vee(x)\)
    are pure tensors.
\end{lemma}
\begin{proof}
    The claim is clear if \(x\in\bM^\x\), and the general case can be
    proved by taking the limit, because for any \(i\), pure tensors form a
    closed cone in \(V_i\).
\end{proof}

\begin{corollary}
    \label[corollary]{cor:f_i_are_pure_tensors}
    For any \(x\in\bM\) and any \(1< i\le n\), we have an inclusion
    \begin{align}
        \bFf_i(x) \in \bFf_{i-1}(x)\wedge V_1\subset V_i,
    \end{align}
    where we use the convention \(V_n=\wedge^nk^n\). A similar result holds for
    \(\bFf_i^\vee(x)\).
\end{corollary}
\begin{proof}
    The statement is obvious if \(x\in\bM^\x\) by the definition of
    \(\bFf_i(x)\). When \(x\in\bM\), first we note that the claim is trivial if
    \(\bFf_i(x)=0\), and if \(\bFf_i(x)\neq 0\), then
    \(\bFf_j(x)\neq 0\) for any \(j<i\). Choose a one-parameter
    family \(S\to\bM\) where \(S\) is a local \(\bar{k}\)-scheme of dimension
    \(1\) with special point \(s\) and generic point \(\eta\), such that the
    image of \(s\) is \(x\) and that of \(\eta\), denoted by \(x_\eta\), is
    contained in \(\bM^\x\). Since \(\bFf_i(x)\neq 0\), we must have
    \(\bFf_i(x_\eta)\neq 0\) either because the condition is open.

    For any \(i\), since \(\bFf_{i-1}\neq 0\) on \(S\) and is a pure tensor, the
    wedging map
    \begin{align}
        \bFf_{i-1}\wedge\colon V_1\longto V_i
    \end{align}
    is locally constant with \((i-1)\)-dimensional kernel. Since
    \(\bFf_i(x_\eta)\in \bFf_{i-1}(x_\eta)\wedge V_1\), we have
    \begin{align}
        \bFf_i(x)\in\bar{\bFf_{i-1}(x_\eta)\wedge V_1},
    \end{align}
    and the fiber of the right-hand side over \(x\) is exactly
    \(\bFf_{i-1}(x)\wedge V_1\) by local constancy. This proves the claim.
\end{proof}

\subsubsection{}
The pairing \(x\mapsto
\bFf_n^\vee(x)\bFf_n(x)\) defines a function on \(\bM\), and one can check that
if \(x\in\bM^\x\), we have
\begin{align}
    \Disc_{\bM,\bH}(x)= \bFf_n^\vee(x)\bFf_n(x).
\end{align}
Since \(k[\bM]\) is a subring of \(k[\bM^\x]\), we have the following lemma:

\begin{lemma}
    The function \(\Disc_{\bM,\bH}\) extends to a function on \(\bM\), which
    also descends to \(\bC_{\bM,\bH}\).
\end{lemma}
\begin{proof}
    The first claim is due to the discussion above, and the second claim is
    immediate since the function is clearly invariant under \(\bH\)-action.
\end{proof}

\begin{definition}
    The \notion{extended \(\bH\)-discriminant function} is defined as the product
    \begin{align}
        \Disc_{\bM,\bH}^+\defeq\Disc_+\cdot\Disc_{\bM,\bH}\in k[\bC_{\bM,\bH}],
    \end{align}
    where \(\Disc_+\) is as in \Cref{ssub:usual_disc}.
    The \notion{extended \(\bH\)-discriminant divisor}
    \(\bD_{\bM,\bH}^+\) is the principal divisor of \(\bC_{\bM,\bH}\) defined by
    \(\Disc_{\bM,\bH}^+\).
\end{definition}

\begin{definition}
    The complement \(\bC_{\bM,\bH}^\srs\defeq\bC_{\bM,\bH}-\bD_{\bM,\bH}^+\) is
    called the \notion{strongly \(\bH\)-regular
    semisimple} locus.
    An element \(x\in\bM\) is called \notion{strongly \(\bH\)-regular
    semisimple}, or \(\bH\hy\srs\) for short, if
    \(\chi_{\bM,\bH}(x)\in\bC_{\bM,\bH}^\srs\).
\end{definition}

\begin{lemma}
    \label[lemma]{lem:H_srs_implies_trivial_stab}
    If \(x\in\bM^{\bH\hy\srs}\) is \(\bH\hy\srs\), then its stabilizer \(\bH_x\)
    in \(\bH\) is trivial.
\end{lemma}
\begin{proof}
    By definition, \(x\) is also regular semisimple in the usual sense, so
    its stabilizer in \(\bG\) is a torus. Taking the intersection with \(\bH\),
    we see that the \(\bH_x\) is diagonalizable.
    We now show that \(\bH_x\) is also unipotent, which will imply the lemma.

    Since \(\bFf_n^\vee(x)\bFf_n(x)\neq 0\), the vector \(\bFf_n(x)\neq 0\) either.
    Therefore, \(\bFf_i(x)\neq 0\) for any \(i\). Suppose \(h\in \bH_x\), then
    \(\wedge^i h\) fixes \(\bFf_i(x)\) for any \(i\). When \(i=1\), this means
    that \(\bFe_1'\defeq x_1(\bFe)\) is an eigenvector of \(h\) with eigenvalue
    \(1\).
    By \Cref{cor:f_i_are_pure_tensors} and induction, we can find linearly
    independent vectors
    \begin{align}
        \bFe_0'=\bFe,\bFe_1',\ldots,\bFe_{n-1}',
    \end{align}
    such that \(h(\bFe_i')\in\bFe_i'+ \sum_{j=0}^{i-1}k\bFe_j'\). This shows
    that \(\bH_x\) is unipotent, as claimed.
\end{proof}

\begin{corollary}
    For any \(a\in\bC_{\bM,\bH}^\srs\), the fiber \(\chi_{\bM,\bH}^{-1}(a)\) is
    a \(\bH\)-torsor.
\end{corollary}
\begin{proof}
    Since \(\bH\) is reductive, there is a unique closed \(\bH\)-orbit in
    \(\chi_{\bM,\bH}^{-1}(a)\). By \Cref{lem:H_srs_implies_trivial_stab},
    every orbit in \(\chi_{\bM,\bH}^{-1}(a)\) is an \(\bH\)-torsor hence has the
    same dimension \(\dim{\bH}=(n-1)^2=\dim{\bM}-\dim{\bC_{\bM,\bH}}\). This
    implies that every orbit is both open and closed hence must be the unique
    closed orbit. This finishes the proof. 
\end{proof}


\subsection{The involutions and twisted forms} 
\label{sub:the_involutions_and_twisted_forms}

In order to study the Jacquet--Rallis fundamental lemma, we will need some Galois
twists of \(\bM\) and \(\bH\). This can be achieved using certain involution on
the monoid which we now formulate.

\subsubsection{}
The commutative monoid \(\bA_\bM\) admits an involution \(\iota_\bA\) sending a
simple root \(\Rt\) to \(-w_0(\Rt)\), where \(w_0\) is the longest element in
the Weyl group with respect to the Borel (this is valid for the universal
monoid associated with a general split semisimple and simply-connected group, but we
will not need it). Explicitly for \(\bG\) (in other words, \(\GL_n\)),
this is induced by flipping the Dynkin diagram of \(\bG\).

The involution \(\iota_\bA\) can be upgraded to an anti-involution \(\iota\) on
\(\bM\) as follows: given a fundamental representation
\(V_i\) of \(\bG^\SC\), its contragredient representation is again a fundamental
representation, denoted by \(V_i^\vee\). Any \(x_i\in\End(V_i)\) then has the
adjoint element \(x_i^\vee\in\End(V_i^\vee)\). There is a unique \(j=n-i\) such that
\(\Wt_j=-w_0(\Wt_i)\) and \(V_j\cong V_i^\vee\) as \(\bG^\SC\)-representations, and
we fix one such isomorphism \(\phi_{i}\). Then \(x_i^\vee\) may be viewed as an
element \(y_j\in
\End(V_j)\), and \(y_j\) is independent of the choice of \(\phi_i\) by Schur's
lemma. This defines an anti-involution
\begin{align}
    \iota_\star\colon \bA_\bM\oplus\bigoplus_{i=1}^{n-1}\End(V_i)&\longto
    \bA_\bM\oplus\bigoplus_{i=1}^{n-1}\End(V_i)\\
    (z, x_i)&\longmapsto (\iota_\bA(z), y_i).
\end{align}
It is easy to see that \(\iota_\star\) preserves \(\bM^\x\): for
\((z,g)\in\bM^\x\), \(\iota_\star\) is given by the formula
\begin{align}
    (z,g)&\longmapsto (\iota_\bA(z),g^{-1}).
\end{align}
Therefore, \(\iota_\star\) induces an anti-involution \(\iota\) on \(\bM\) as
desired.

\begin{remark}
    \begin{enumerate}
        \item When \(x\in\bM\) is contained in \(\bG^\SC\),
            then \(\iota(x)\) is just \(x^{-1}\). However, \(\iota\) has
            nothing to do with ``inversion'' on \(\bM\), which does not exist.
        \item Conceptually, \(\iota\) is the canonical anti-isomorphism from
            \(\bM\) to its opposite monoid, and it is defined for the universal
            monoid of any \(\bG^\SC\), not just \(\SL_n\).
    \end{enumerate}
\end{remark}

\begin{example}
    When \(n=2\), we have \(\bM=\Mat_2\), and \(\iota\) is the map
    \begin{align}
        \begin{pmatrix}
            a & b\\
            c & d
        \end{pmatrix}
        \longmapsto
        \begin{pmatrix}
            d & -b\\
            -c & a
        \end{pmatrix}.
    \end{align}
    When \(n\ge 3\), however, it is impossible to lift \(\iota_\bA\) to the
    monoid \(\Mat_n\), unlike the case of \(\bM\).
\end{example}

\subsubsection{}
There is also an \emph{involution} \(\tau\) on \(\bM\) specific to the
\(\bG=\GL_n\) case. Indeed, the group \(\Out(\bG^\SC)\) of outer automorphisms
has order \(2\), and it induces an involution \(\tilde{\tau}\) on \(\bM\) using
the standard pinning of \(\bG^\SC\). Let \(\dot{w}_0\) be the following element
of \(\SL_n\) lifting \(w_0\):
\begin{align}
    \label{eqn:w_0_matrix}
    \dot{w}_0=\begin{pmatrix}
         & & & & (-1)^{n-1}\\
         & & & (-1)^{n-2} & \\
         & & \iddots & & \\
         & -1 & & & \\
         1 & & & & \\
    \end{pmatrix},
\end{align}
and define
\begin{align}
    \tau\defeq \Ad_{\dot{w}_0}\circ\tilde{\tau}.
\end{align}
One can verify that on \(\bM^\x\), \(\tau\) is given by the formula
\begin{align}
    (z,g)\longmapsto (\iota_\bA(z),\tp{g}^{-1}),
\end{align}
so it is indeed an involution. Note that \(\tau\) is compatible with the
involution on \(\bG\) (resp.~\(\bH\)) given by transpose-inverse.

\subsubsection{}
For any \(k\)-scheme \(X\) and an \'etale double cover \(\OGT\colon X'\to
X\), we consider the Weil restriction
\begin{align}
    \bM'\defeq\OGT_*(\bM\x X'),
\end{align}
viewed as a reductive monoid scheme with \(\bH'\defeq\OGT_*(\bH\x X')\)-action
over \(X\). For convenience, we let \(\bG'\defeq \OGT_*(\bG\x X')\). The
anti-involution \(\iota\) on \(\bM\) induces an anti-involution
on \(\bM'\), still denoted by \(\iota\), and the latter clearly commutes with
the unique non-trivial element \(\sigma\in\Gal(X'/X)\). Let \(\FRM\) to be the
monoidal symmetric space
\begin{align}
    \FRM\defeq (\bM')^{\iota\sigma}.
\end{align}
The group \(\bH\) can be identified with the fixed-point subgroup
\(H\defeq (\bH')^\sigma\), and \(H\) acts on \(\FRM\).

\subsubsection{}
On the other hand, we have the unitary form of \(\bM\) defined as
\begin{align}
    \FRM'\defeq (\bM')^{\tau\sigma},
\end{align}
as well as unitary groups
\begin{align}
    G'&\defeq\UNI_n=(\bG')^{\tau\sigma},\\
    H'&\defeq\UNI_{n-1}=(\bH')^{\tau\sigma}.
\end{align}
Clearly, \(H'\) acts on \(\FRM'\) by the adjoint action.

\subsubsection{}
The abelianization map \(\alpha_{\bM'}\) is easily seen to be equivariant under
both \(\iota\sigma\) and \(\tau\sigma\), and these two Galois actions coincide
when decended to \(\bA_{\bM'}\). Therefore we may define the common
abelianization of \(\FRM\) and \(\FRM'\) to be the invariant space
\(\bA_{\bM'}^{\iota\sigma}=\bA_{\bM'}^{\tau\sigma}\). We denote it simply by
\(\FRA_\FRM\) without any prime (\('\)) symbol.

\subsubsection{}
Now we study the invariant theory of the \(H\)-action on \(\FRM\) and that of
the \(H'\)-action on \(\FRM'\) using the results from the \(\bH\)-action on
\(\bM\).

First, we note that both \(\iota\sigma\) and \(\tau\sigma\) maps
\(\bH'\)-orbits to \(\bH'\)-orbits, and so they descend to involutions on
\(\bC_{\bM,\bH}'\defeq\OGT_*(\bC_{\bM,\bH}\x X')\). Moreover, note
that the action of \(\iota\sigma\) coincides with that of \(\tau\sigma\),
because on \(\bM^\x\), \(\iota\) and \(\tau\) only differ by a transposition
on the \(\bG^\SC\)-factor, which does not affect the traces on \(V_i'\) and
\(V_i''\). This allows us to define
\begin{align}
    \FRC_{\FRM,H}\defeq
    (\bC_{\bM,\bH}')^{\iota\sigma}=(\bC_{\bM,\bH}')^{\tau\sigma}.
\end{align}

As a result, we have the following commutative diagram:
\begin{equation}
    \begin{tikzcd}
        \FRM \ar[r]\ar[d] & \bM' \ar[d] & \FRM'\ar[l]\ar[d] \\
        \FRM\git H \ar[r]\ar[d] & \bM'\git \bH'\ar[d] & \FRM'\git H'\ar[l]\ar[d]\\
        \FRC_{\FRM,H}\ar[r] &
        \bC_{\bM,\bH}' & \FRC_{\FRM,H}\ar[l]
    \end{tikzcd}.
\end{equation}

\begin{lemma}
    We have natural isomorphisms
    \begin{align}
        \FRM\git H\stackrel{\sim}{\longto}\FRC_{\FRM,H}\stackrel{\sim}{\longot}
        \FRM'\git H'.
    \end{align}
\end{lemma}
\begin{proof}
    The question is local in \(X\), so by passing to a finite \'etale
    cover, we may assume that \(\OGT\) is split. In this case, the claim is
    clear since all three schemes are isomorphic to \(\bC_{\bM,\bH}\x X\).
\end{proof}

\subsubsection{}
Define the twisted forms of \(\chi_{\bM,\bH}\):
\begin{align}
    \chi_{\FRM,H}\colon \FRM&\longto\FRC_{\FRM,H},\\
    \chi_{\FRM,H}'\colon \FRM'&\longto\FRC_{\FRM,H},
\end{align}
then their fibers can be similarly studied using fibers of \(\chi_{\bM,\bH}\).
The extended \(\bH\)-discriminant divisor \(\bD_{\bM,\bH}^+\) also similarly induces a
principal \notion{extended \(H\)-discriminant} divisor \(\FRD_{\FRM,H}^+\) on
\(\FRC_{\FRM,H}\), whose complement is the \(H\)-srs locus
\(\FRC_{\FRM,H}^\srs\). These facts are straightforward to check, and we leave
the details to the reader.

\subsubsection{}
For future convenience, we also describe the scheme \(\FRC_{\FRM,H}\) in a more
explicit manner. We will only consider the case when \(X=\Spec R\) and
\(X'=\Spec R'\) are affine, and the general case is similar.

The scheme \(\bC_{\bM,\bH}'\) in this case is just an affine space
\(\bbA_{R'}^{3n-3}\), so it suffices to describe the set \(\FRC_{\FRM,H}(R)\) as
a subset of \((R')^{3n-3}\). On the direct factor \((R')^{n-1}\) corresponding
to the abelianization, the involution \(\iota\sigma=\tau\sigma\) takes
\((z_i)\) to the point \((\sigma z_{n-i})\).
On the remaining direct factor, if we write \((a_i',a_i'')\) corresponding
to coordinates \((\chi_i',\chi_i'')\), then the
involution is the map
\begin{align}
    (a_i',a_i'')\longmapsto (\sigma a_{n-i}'',\sigma a_{n-i}').
\end{align}
To see this, one only needs to notice that if we identify \(V_i^\vee\) with
\(V_{n-i}\) using the standard volume form, then its direct summand
\((V_i')^\vee\) is
identified with \(V_{n-i-1}'\otimes k\bFe\), which is exactly \(V_{n-i}''\).
As a result, \(\FRC_{\FRM,H}(R)\) is the set
\begin{align}
    \Set*{(z_i,a_i',a_i'')\given z_i=\sigma z_{n-i}, a_i'=\sigma a_{n-i}'',
    a_i''=\sigma a_{n-i}'},
\end{align}
where the \(z_i\) part belongs to \(\FRA_\FRM\).


\subsection{Deformed quotient stack} 
\label{sub:deformed_quotient_stack}

Although the quotient stacks \(\sM_1\defeq\Stack*{\FRM/H}\) and
\(\sM_1'\defeq\Stack*{\FRM'/H'}\) already captures the
information on the geometric side of the trace formula that we are interested
in, experience shows that by considering certain respective deformations of
these stacks, many global results become significantly easier to prove. Such
deformations are analogues to the ones for Lie algebras introduced by \cite{Yu11}.

\subsubsection{}
We first consider the split case and begin by reformulating
\(\Stack*{\bM/\bH}\) in a coordinate-free way. We consider the
following classifying stack (\(S\) is a \(k\)-scheme)
\begin{align}
    \bsM(S)=\Set*{(\cE,x,e,e^\vee)\given \begin{array}{l}
            \cE\in\Bun_{\bG}(S),\\
            x\in \bM\x^{\bG}\cE,\\
            e\colon\cO_S\to V_\cE, e^\vee\colon V_\cE\to \cO_S
    \end{array}},
\end{align}
where \(V_\cE\) means the vector bundle associated with the \(\bG\)-torsor
\(\cE\) and \(V_\cE^\vee\) its vector dual. It maps to \(\Stack*{\bM/\bG}\) by
forgetting the vector \(e\) and co-vector \(e^\vee\), and this forgetful map is
representable (by section spaces of vector bundles). As a result,
\(\bsM\) is an algebraic stack of finite type.

Inside \(\bsM\) we have the closed substack \(\bsM_1\) consisting of tuples such
that \(e^\vee e=1\).
\begin{lemma}
    We have a natural isomorphism
    \begin{align}
        \Stack*{\bM/\bH}\stackrel{\sim}{\longto}\bsM_1.
    \end{align}
\end{lemma}
\begin{proof}
    The natural map is defined as follows: given \((\cE_\bH,x_\bH)\) (where \(\cE_\bH\) is an
    \(\bH\)-torsor) on the left-hand side, we have induced \(\bG\)-torsor
    \(\cE\) and the point \(x\) is obvious. Let \(V_{\cE_\bH}\) be the
    rank-\((n-1)\) vector bundle associated with \(\cE_\bH\), then \(V_\cE\) is
    just the direct sum of \(V_{\cE_\bH}\) with the trivial line
    bundle, and \(e\) (resp.~\(e^\vee\)) is the canonical inclusion
    (resp.~projection) of the trivial line.

    The inverse functor is similarly
    straightforward: if \(e^\vee e=1\), then together \(e\) and \(e^\vee\) define a
    direct summand of \(V_\cE\) that is a trivial line bundle, which is
    equivalent to an \(\bH\)-reduction \(\cE_\bH\) of \(\cE\).
    The rest of the verifications are left to the reader.
\end{proof}

\subsubsection{}
There is an obvious fibration \(b_0\colon \bsM\to \bbA^1\) sending
\((\cE,x,e,e^\vee)\) to \(e^\vee e\), and for any \(s\in\bbA^1\), we let
\(\bsM_s=b_0^{-1}(s)\) to be the preimage of \(s\). The following result is
obvious:
\begin{lemma}
    We have an isomorphism
    \begin{align}
        \bsM_1\x\Gm &\longto b_0^{-1}(\Gm)\\
        \bigl((\cE,x,e,e^\vee),s\bigr) &\longmapsto (\cE,x,e,se^\vee).
    \end{align}
\end{lemma}

\subsubsection{}
We review a little bit of linear algebra before we proceed. Given a vector space
\(V\) and a vector \(e\in V\), we may define for any \(i\in\bbN\) a
natural linear map
\begin{align}
    e\wedge\colon \wedge^i V\longto \wedge^{i+1}V
\end{align}
by wedging \(e\) on the left. On the other hand, given a co-vector \(e^\vee \in V^\vee \), we may also
define a \notion{contraction} (linear) map
\begin{align}
    e^\vee \colon\wedge^{i+1}V&\longto \wedge^iV\\
    e_0\wedge\cdots\wedge e_i&\longmapsto \sum_{j=0}^i(-1)^j(e^\vee e_j)
    \bigl(e_0\wedge\cdots\wedge\hat{e}_j\wedge\cdots \wedge e_i\bigr),
\end{align}
where \(\hat{e}_j\) means the \(e_j\)-term is deleted.

\subsubsection{}
Let \(\bC_\bsM\defeq \bC_{\bM,\bH}\x\bbA^1\), and we now define a map
\begin{align}
    \bsM\longto\bC_\bsM
\end{align}
extending the map \(\Stack*{\bM/\bH}\to \bC_{\bM,\bH}\). Given
\((\cE,x,e,e^\vee )\in\bsM\), we have invariants \(a_i=\chi_i(x)\) as before.
For \(1\le i\le n-1\), recall the element \(y_i\in \End(V_i)\) induced by
any element \(y\in\bM\). By twisting by \(\cE\), we have
\(x_i\in\End(V_{\cE,i})\) associated with \(x\) where \(V_{\cE,i}=\wedge^i
V_\cE\). Define \(x_i^{e,e^\vee }\) to be the composition map
\begin{align}
    V_{\cE,i-1}\stackrel{e\wedge}{\longto}V_{\cE,i}\stackrel{x_i}{\longto}
    V_{\cE,i}\stackrel{e^\vee }{\longto}V_{\cE,i-1},
\end{align}
and let \(\chi_i^{e,e^\vee }(x)\) be the trace of \(x_i^{e,e^\vee }\).
Let
\begin{align}
    b_i= \begin{cases}
        e^\vee e & i=0,\\
        \chi_i^{e,e^\vee }(x) & 1\le i\le n-1.
    \end{cases}
\end{align}

When \(V=k^n\), \(e=\bFe\), \(e^\vee =\bFe^\vee \), it is easy to see that \(b_0=1\) and
for \(i\ge 1\), \(b_i=a_i''\). We already have \(a_i=a_i'+a_i''\), so \(\Rt_i\)
(viewed as functions on \(\bA_\bM\)), \(\chi_i\), and \(\chi_i^{e,e^\vee }\) give an
alternative set of coordinates of \(\bC_{\bM,\bH}\). Define
\begin{align}
    \chi_\bsM\colon \bsM&\longto \bC_\bsM\\
    (\cE,x,e,e^\vee )&\longmapsto \bigl(\alpha_\bM(x),\chi_i(x),\chi_i^{e,e^\vee }(x),b_0\bigr).
\end{align}
The restriction of \(\chi_\bsM\) to \(\bsM_1\) is the same as the one induced by
\(\chi_{\bM,\bH}\).

\subsubsection{}
There is a section from \(\bC_\bsM\) to \(\bsM\) extending the one on
\(\bC_{\bM,\bH}\), defined as follows: given \(z\in\bA_\bM\),
\(a=(a_i)\), \(b=(b_i)\) (including \(b_0\)), we let \(\cE=\cE_0\) be the
trivial \(\bG\)-torsor (so that \(V_\cE=k^n\)), \(x_{z,a}=\epsilon_\bM(z,a)\),
\(e=\bFe\), and
\begin{align}
    e^\vee =\bFe_b^\vee \defeq
    (-1)^{n-1}b_1\bFe_1^\vee+(-1)^{n-2}b_2\bFe_2^\vee+\cdots-b_{n-1}\bFe_{n-1}^\vee+b_0\bFe^\vee.
\end{align}
We define map
\begin{align}
    \epsilon_\bsM\colon\bC_\bsM&\longto \bsM\\
    (z,a,b)&\longmapsto (\cE_0,x_{z,a},\bFe,\bFe_b^\vee ).
\end{align}

\begin{remark}
    If \(b_0=1\), then \(\epsilon_\bsM\) is the same as the map to
    \(\bsM_1\) induced by \(\epsilon_{\bM,\bH}\), because the role of conjugation by
    \(\beta(a_i'')\) in \(\epsilon_{\bM,\bH}\) is the same as changing the basis of
    \(k^n\) so that \(\bFe\) stays the same and \(\bFe^\vee \) becomes
    \(\bFe_b^\vee \). For \(\epsilon_{\bM,\bH}\), there is additional information being constructed, namely
    how we change the basis elements \(\bFe_i\) and \(\bFe_i^\vee \). For
    \(\epsilon_\bsM\), we no longer retain such additional information.
\end{remark}

\begin{proposition}
    The map \(\epsilon_\bsM\) is a section of \(\chi_\bsM\) and its
    restriction to \(\bC_{\bM,\bH}\) coincides with \(\epsilon_{\bM,\bH}\).
\end{proposition}
\begin{proof}
    The proof is, in essence, similar to that for \(\epsilon_{\bM,\bH}\), and
    clearly we only
    need to consider the \(b_j\)-coordinates for \(j\ge 1\). We choose the
    basis of \(V_j\) induced by the standard basis of \(k^n\).
    Since wedging with \(\bFe\) kills any pure tensor with factor
    \(\bFe\), we only need to consider the effect of \(x_j^{e,e^\vee }\) on basis
    elements of the form
    \begin{align}
        \bFe_{i_1\cdots i_{j-1}}=\bFe_{i_1}\wedge\cdots\wedge \bFe_{i_{j-1}},\quad 1\le
        i_1<\cdots<i_{j-1}\le n-1.
    \end{align}
    As before, let \(A=\epsilon(a)\in\SL_n\) be the companion matrix associated
    with \(a\), then \(A\) sends \(\bFe_i\) to
    \((-1)^{i-1}a_i\bFe_1+\bFe_{i+1}\) (recall \(\bFe_n=\bFe\)) and \(\bFe\) to
    \((-1)^{n-1}\bFe_1\). This means that \(\bFe\wedge\bFe_{i_1\cdots i_{j-1}}\) is sent to
    \begin{align}
        (-1)^{n-1}\bFe_1\wedge \bigl((-1)^{i_1-1}a_{i_1}\bFe_1+\bFe_{i_1+1}\bigr)\wedge
        \cdots\wedge \bigl((-1)^{i_{j-1}-1}a_{i_{j-1}}\bFe_1+\bFe_{i_{j-1}+1}\bigr)\\
        =
        (-1)^{n-1}\bFe_1\wedge \bFe_{i_1+1}\wedge \cdots\wedge \bFe_{i_{j-1}+1}.
    \end{align}
    Scaling the result by \(\delta_\bM(z)\), we see that
    \(x_{z,a}\) sends \(\bFe\wedge\bFe_{i_1\cdots i_{j-1}}\) to
    \begin{align}
        \Bigl((z_2\cdots z_{i_1})(z_3\cdots z_{i_2})\cdots(z_{j}\cdots
        z_{i_{j-1}})(-1)^{n-1}\Bigr)\bFe_1\wedge \bFe_{i_1+1}\wedge \cdots\wedge
        \bFe_{i_{j-1}+1},
    \end{align}
    where the product \(z_{l+1}\cdots z_{i_l}\) is \(1\) if \(i_l< l+1\), or
    equivalently, \(i_l=l\). To compute the contribution of \(\bFe_{i_1\cdots
    i_{j-1}}\) to the trace \(\chi_j^{e,e^\vee }\), we need to compute
    the contraction with \(\bFe_b^\vee \) and then project to the basis
    \(\bFe_{i_1\cdots i_{j-1}}\). Therefore, for \(\bFe_{i_1\cdots i_{j-1}}\) to
    have non-zero contribution to \(b_j\), we must have
    \begin{align}
        (1,i_1+1,\ldots,i_{j-2}+1)=(i_1,i_2,\ldots,i_{j-1})=(1,2,\ldots,j-1).
    \end{align}
    Moreover, only the term involving \(\bFe_{j}^\vee \) in \(\bFe_b^\vee \) can
    contribute to \(b_j\).
    The trace \(\chi_j^{e,e^\vee }\) can then be computed to be the coefficient in front of the
    pure tensor
    \begin{align}
        (-1)^{n-j}b_j\bFe_j^\vee \bigl((-1)^{n-1}\bFe_1\wedge \bFe_{2}\wedge \cdots\wedge
        \bFe_{j}\bigr)
        =
        b_j\bFe_1\wedge \bFe_{2}\wedge \cdots\wedge \bFe_{j-1},
    \end{align}
    which is \(b_j\). This finishes the proof.
\end{proof}

\subsubsection{}
By replacing \(\bFe\) with \(e\) (resp.~\(\bFe^\vee \) with \(e^\vee \)), we can define the
analogue of \(\bFf_i(x)\) (resp.~\(\bFf_i^\vee (x)\)) for any \(m=(\cE,x,e,e^\vee )\):
\begin{align}
    \bFf_i(m)&= \begin{cases}
        x_1 e & i=1,\\
        x_i\bigl(e\wedge \bFf_{i-1}(m)\bigr) & 1< i<n,\\
        e\wedge \bFf_{n-1}(m) & i=n;
    \end{cases}\\
    \bFf_i^\vee (m)&= \begin{cases}
        e^\vee x_1 & i=1,\\
        \bigl(\bFf_{i-1}^\vee (m)\wedge e^\vee \bigr)x_i & 1< i<n,\\
        \bFf_{n-1}^\vee (m)\wedge e^\vee  & i=n.
    \end{cases}
\end{align}
Let \(\bsM^\x\subset\bsM\) be the preimage of \(\Stack*{\bM^\x/\bG}\), then
clearly it is an open dense substack. Similar to
\Cref{cor:f_i_are_pure_tensors}, we have:
\begin{lemma}
    \label[lemma]{lem:f_i_are_pure_tensors_extended}
    Both \(\bFf_i(m)\) and \(\bFf_i^\vee (m)\) are pure tensors for all \(m\in\bsM\).
\end{lemma}

\begin{definition}
    We call \(m\in\bsM\) \notion{strongly regular semisimple} if its image \(x\)
    in \(\Stack*{\bM/\bG}\) is regular semisimple and \(\bFf_n^\vee (m)\bFf_n(m)\neq
    0\). We denote the open substack of strongly regular semisimple elements by
    \(\bsM^\srs\).
\end{definition}

\begin{proposition}
    The map \(\chi_\bsM\) is initial in the category of maps \(\bsM\to S\) where
    \(S\) is an affine scheme.
\end{proposition}
\begin{proof}
    It is easy to see that any map \(b_0^{-1}(\Gm)\to S\) to a scheme factors
    through \(\Gm\x\bC_{\bM,\bH}\) using the result for \(\bsM_1\).

    On the other hand, we claim that \(\bsM^{\x,\srs}\) is equal to
    \(\chi_{\bsM}^{-1}(\chi_{\bsM}(\bsM^{\x,\srs}))\) and also isomorphic to
    \(\chi_{\bsM}(\bsM^{\x,\srs})\). By
    \(\bZ_\bM\)-equivariance, we only need to prove this for the closed substack
    of \(\bsM^{\x,\srs}\) lying over \(\Stack*{\bG^\SC/\bG}\),
    or equivalently, the abelianization is \(1\).
    By embedding \(\bG^\SC=\SL_n\) into \(\Mat_n=\La{gl}_n\), the claim then
    follows from the well-known result from the Lie algebra case \cite{Yu11}.

    Finally, given a map \(\bsM\to S\) where \(S\) is affine, its restriction
    to \(\bsM^{\x,\srs}\cup b_0^{-1}(\Gm)\) uniquely factors through the image of the
    latter in \(\bC_{\bsM}\), whose complement has codimension \(2\). Since
    \(S\) is affine and \(\bC_\bsM\) is normal, the factorization then extends
    to the whole \(\bC_\bsM\).
\end{proof}

\begin{corollary}
    The pairing \(\bFf_n^\vee \bFf_n\) descends to a function on \(\bC_\bsM\).
\end{corollary}

\begin{definition}
    \label[definition]{def:extended_disc_divisor}
    \begin{enumerate}
        \item We define the \notion{(resp.~extended) discriminant function}
            \begin{align}
               \Disc_\bsM\defeq\bFf_n^\vee \bFf_n \quad \text{(resp.~}\Disc_\bsM^+\defeq
                \Disc_+\cdot\bFf_n^\vee \bFf_n\text{)}
                 \in k[\bC_\bsM],
            \end{align}
            where \(\Disc_+\) is the usual extended discriminant
            function on \(\bC_\bM\).
        \item The divisor cut out by \(\Disc_\bsM\) (resp.~\(\Disc_\bsM^+\)) is
            called the \notion{(resp.~extended) discriminant divisor}, denoted
            by \(\bD_\bsM\) (resp.~\(\bD_\bsM^+\)).
        \item The complement \(\bC_{\bsM}^\srs=\bC_\bsM-\bD_\bsM^+\) is called
            the \notion{strongly regular semisimple locus}.
    \end{enumerate}
\end{definition}

\begin{proposition}
    \label[proposition]{prop:srs_implies_iso_to_GIT}
    The restriction of \(\chi_\bsM\) to \(\bsM^\srs\) is an isomorphism.
\end{proposition}
\begin{proof}
    We first prove that \(\bsM^\srs\) is a sheaf. Since \(\bsM\) is clearly an
    algebraic stack, it suffices to show that the automorphism group of any
    point \(m=(\cE,x,e,e^\vee)\in\bsM^\srs(R)\) is trivial, where \(R\) is an
    Artinian algebra over \(\bar{k}\). We
    fix a trivialization of \(\cE\), then the automorphism group of \(m\) is the
    same as the stabilizer in \(\bG\) of the tuple
    \((x,e,e^\vee)\in\bM\x\bbA^n\x\bbA^n\).

    By definition, \(x\) is semisimple, so we may choose the trivialization of
    \(\cE\) in such a way that \(x\) is identified with a point in
    \(\bar{\bT}_\bM(R)\). Since \(x\) is also \(\bG\)-regular, it is contained in
    the big-cell locus (see \cite[\S~2.4.15]{Wa25} for example for a quick
    summary). By possibly conjugating by
    the Weyl group, we may assume that \(x\) can be written as
    \begin{align}
        x=e_{I,\SimRts}zt,
    \end{align}
    where \(e_{I,\SimRts}\) is an idempotent in \(\bar{\bT}_\bM\),
    \(t\in\bT^\SC(R)\), and \(z\in\bZ_\bM\simeq\bT^\SC(R)\). Since the stabilizer
    stays constant along \(\bZ_\bM\)-orbits in \(\bM\x\bbA^n\x\bbA^n\) (where
    \(\bZ_\bM\) acts on \(\bbA^n\x\bbA^n\) trivially), we may
    assume that \(z=1\). Some basic properties of
    \(e_{I,\SimRts}\) can be found in \cite[\S~2.4.17]{Wa25}, and since
    \(\bG=\GL_n\), an explicit description is as follows: the standard
    basis of the standard representation \(V_1\) has weights
    \begin{align}
        \Wt_1,\Wt_1-\Rt_1,\ldots,\Wt_1-\Rt_1-\cdots -\Rt_{n-1}.
    \end{align}
    For each \(1<i\le n-1\), the highest weight of \(V_i\) is
    \begin{align}
        \Wt_i&=\sum_{j=1}^{i}\Bigl(\Wt_1-\sum_{l=1}^{j-1}\Rt_l\Bigr),\\
        &=i\Wt_1-\sum_{j=1}^{i}(i-j)\Rt_j,
    \end{align}
    and so lower weights in \(V_i\) can be written as
    \begin{align}
        \mu\defeq\Wt_i-\sum_{j=i}^{n-1}c_j\Rt_j,
    \end{align}
    such that if \(c_j\neq 0\), then \(c_l\neq 0\) for all \(i\le l<j\).
    The eigenvalue \(\mu(e_{I,\SimRts})\) is \(1\) if \(c_j=0\) outside
    \(I\), and \(0\) otherwise. Therefore, we may assume without loss of
    generality that
    \begin{align}
        I=\Set{1,\ldots,s},\quad 0\le s\le n-1.
    \end{align}
    Let \(t=\Diag(t_1,\ldots,t_n)\), and suppose
    \(e=\tp(e_1,\ldots,e_n)\in R^n\), then \(\bFf_n(m)\) is the determinant of the
    matrix
    \begin{align}
        \begin{pmatrix}
            e_1 & \cdots & e_1t_1^{n-s-2} & e_1t_1^{n-s-1} & \cdots & e_1t_1^{n-1}\\
            \vdots &  & \vdots & \vdots & & \vdots \\
            e_{s+1} & \cdots & e_{s+1}t_{s+1}^{n-s-2} & e_{s+1}t_{s+1}^{n-s-1} & \cdots & e_{s+1}t_{s+1}^{n-1}\\
            e_{s+2} & \cdots & e_{s+2}t_{s+2}^{n-s-2} & & \\
            \vdots & \iddots &  &  & &  \\
            e_{n} &  &  & & 
        \end{pmatrix}.
    \end{align}
    Since \(\bFf_n(m)\) is invertible,
    we must have \(e_i\in R^\x\) for all \(i\). The
    stabilizer of \(x\) in \(\bG\) is diagonal, so \(e_i\in R^\x\) implies that if
    \(\Ad_g(x)=x\) and \(ge=e\), then \(g=1\). This proves the claim that
    \(\bsM^\srs\) is a sheaf.

    Still assuming \(z=1\), the invariants \(b_i\) can be similarly computed as
    entries in the row vector
    \begin{align}
        (e_1e_1^\vee,\ldots,e_ne_n^\vee)\bF,
    \end{align}
    where \(\bF\) is the matrix
    \begin{align}
        \begin{pmatrix}
            t_1\Sigma_{1,n-1} & \cdots & t_1\Sigma_{1,s+1} & t_1\Sigma_{1,s} & \cdots & t_1\\
            \vdots &  & \vdots & \vdots & & \vdots \\
            t_{s+1}\Sigma_{s+1,n-1} & \cdots & t_{s+1}\Sigma_{s+1,s+1} & t_{s+1}\Sigma_{s+1,s} & \cdots & t_{s+1}\\
            t_{s+2}\Sigma_{s+2,n-1} & \cdots & t_{s+2}\Sigma_{s+2,s+1} & & \\
            \vdots & \iddots &  &  & &  \\
            t_n\Sigma_{n,n-1} &  &  & & 
        \end{pmatrix},
    \end{align}
    in which \(\Sigma_{i,l}\) is the summation
    \begin{align}
        \sum_{\substack{1\le j_1<\ldots<j_l\le\max\Set{l,s+1}\\i\neq j_1,\ldots,
        j_l}}t_{j_1}\cdots t_{j_l}.
    \end{align}
    We claim that \(\bF\) is invertible: indeed, if \(s<l<n\), then
    \(\Sigma_{l+1,l}\) is simply the product
    \begin{align}
        t_1\cdots t_{l}\in R^\x,
    \end{align}
    so we only need to show that the determinant of the upper-right block of
    \(\bF\) is invertible. Using elementary column operations and induction, we
    see that the said determinant is, up to a sign, equal to that of the
    Vandermonde matrix
    \begin{align}
        \begin{pmatrix}
            t_1^{s+1} & \cdots & t_1\\
            \vdots & & \vdots \\
            t_{s+1}^{s+1} & \cdots & t_{s+1}
        \end{pmatrix},
    \end{align}
    which is invertible because the upper-right corner of \(\bFf_n(m)\) is
    invertible.

    If \(m\) and \(m'\) are two points in \(\bsM^\srs\) with the same
    invariants, then by \(\bG\)-conjugation, we may assume that \(\cE=\cE'\) and
    \(x=x'\). Since \(\bF\) is invertible, we have
    \(e_ie_i^\vee=e_i'e_i^{\vee\prime}\) for all \(i\). Therefore,
    using that \(e_i\in R^\x\), we may find a scalar in \(R^\x\)
    that simultaneously sends \(e\) to \(e'\) and \(e^\vee\) to
    \(e^{\vee\prime}\). This shows that \(\chi_{\bsM}\) is a
    bijection on \(R\)-points for any Artinian algebra \(R\), hence must be an
    isomorphism.
\end{proof}

\subsubsection{}
The Galois twisted forms of \(\bsM\) is similarly formulated. We already have
the monoidal symmetric space \(\FRM\) and groups \(G\simeq\bG\) and
\(H\simeq\bH\) over \(X\), induced by a given \'etale double cover
\(\OGT\colon X'\to X\).

\begin{definition}
We define \(\sM\) to be the
stack whose \(S\)-points (\(S\) is an \(X\)-scheme) are tuples \((\cE,x,e,e^\vee )\)
where \(\cE\) is a \(G\)-torsor over \(S\), \(x\in
\FRM\x^G\cE(S)\), \(e\) (resp.~\(e^\vee \)) is a vector (resp.~co-vector) of
\(V_\cE\), or in other words, \(e\) is a morphism  \(\cO_S\to V_\cE\) and
\(e^\vee \) is a morphism \(V_\cE\to\cO_S\).

We define the unitary form \(\sM'\) to be the stack classifying tuples \((\cE',x',e')\) where
\(\cE'\) is a \(G'=\UNI_n\)-torsor over \(X\), \(x'\in\FRM'\x^{G'}\cE'(S)\), and
\(e'\colon \cO_{S'}\to V'_{\cE'}\) is a vector. Here \(S'=X'\x_X S\) and
\(V'_{\cE'}\) is the Hermitian vector bundle on \(S'\)
associated with \(\cE'\). Through the Hermitian form we may
identify \(V'_{\cE'}\) with the dual of \(\sigma^*V'_{\cE'}\), and so \(e'\)
induces a co-vector
\begin{align}
    e^{\prime\vee}\colon
    V'_{\cE'}\simeq(\sigma^*V'_{\cE'})^\vee \to\sigma^*\cO_{S'}\simeq \cO_{S'}.
\end{align}

\end{definition}

\subsubsection{}
Both \(\sM\) and \(\sM'\) map to the same space \(\FRC_\sM\) of invariants
\begin{align}
    \sM\longto \FRC_\sM\longot \sM'.
\end{align}
Here \(\FRC_\sM\) is isomorphic to \(\FRC_{\FRM,H}\x\bbA^1\), where \(\bbA^1\) corresponds
to the invariant \(b_0\). More concretely, if
\(X=\Spec{R}\) and \(X'=\Spec{R'}\), then \(\FRC_\sM(R)\) is the subset of
\((R')^{3n-2}\)
\begin{align}
    \Set*{(z_1,\ldots,z_{n-1},a_1,\ldots,a_{n-1},b_1,\ldots,b_{n-1},b_0)
    \given z_i=\sigma z_{n-i},a_i=\sigma a_{n-i}, b_i=\sigma (b_0a_{n-i}-b_{n-i})},
\end{align}
where we use the convention \(a_n\defeq 1\) and \(b_n\defeq 0\), so that
\(b_0=\sigma (b_0a_n-b_n)\) is equivalent to saying \(b_0\in R\).

\subsubsection{}
The pure tensors \(\bFf_i\) and \(\bFf_i^\vee \) are not stable under the Galois action
if \(0<i<n\), but both \(\bFf_n\) and \(\bFf_n^\vee \) have twisted forms, which we
denote by \(f_n(x)\) and \(f_n^\vee (x)\) for \((\cE,x,e,e^\vee )\in \sM\), and
\(f_n'(x')\) for \((\cE',x',e')\) for \((\cE',x',e')\in\sM'\). In the unitary
case, the co-tensor \(f_n^{\prime\vee}(x)\) is the dual of \(\sigma^*f_n'(x')\).
Additionally, \(\bFf_i\) and \(\bFf_i^\vee \) are still well-defined as tensors
over \(S'\), and if \(i=n\), they coincide with the pullbacks of \(f_n\) and
\(f_n^\vee \), respectively.



\section{Local Formulations} 
\label{sec:local_formulations}

In this section, we study the geometry in the local setting. For this
purpose, we fix a closed point \(v\in\abs{X}\). Let \(X_v=\Spec{\cO_v}\) be the
formal disc at \(v\), \(X_v^\bullet=\Spec{F_v}\) the punctured disc, and \(k_v\)
the residue field. We also choose a uniformizer \(\pi_v\in\cO_v\).

Similarly, we
let \(v'=\Spec{k_v'}\) be the preimage of \(v\) in \(X'\): it is either a
single point with \(k_{v}'\) being a quadratic field extension of \(k_v\), or
two points with \(k_{v}'=k_v\x k_v\). Let \(X_{v}'=\Spec{\cO_{v}'}\) be the formal
disc around \(v'\), where \(\cO_v'=\cO_v\otimes_{k_v}k_{v}'\), and
\(X_{v}^{\prime\bullet}=\Spec{F_v'}\) the punctured disc. For any geometric
point \(\bar{v}\) over \(v\), that is, a \(k\)-field embedding \(\bar{v}\colon
k_v\to \bar{k}\), we let \(\cO_{\bar{v}}=\cO_v\otimes_{k_v}\bar{k}\), and so on.

\subsection{Affine Jacquet--Rallis fibers} 
\label{sub:affine_jacquet_rallis_fibers}

Similar to the Lie algebra case (\cite{Yu11}) or the usual adjoint
action (\cite{Wa25}), we have an analog of affine Springer fiber which are
geometric incarnations of orbital integrals. We will dub them \notion{affine
Jacquet--Rallis fibers} in this paper.

\subsubsection{}
Recall we have surjective morphisms of stacks
\begin{align}
    \sM\longto \FRC_\sM\longot \sM'.
\end{align}
We define \(\FRC_\sM^\heartsuit(\cO_v)\) to be the subset of \(\FRC_\sM(\cO_v)\)
consisting of points whose restriction to \(F_v\) are invertible and strongly
regular semisimple.

Similarly as in \cite[\S~4]{Wa25}, the image of
\(a\in\FRC_\sM(\cO_v)\) that is invertible over \(F_v\) induces a point in
\(\FRA_\FRM(\cO_v)\cap\FRA_\FRM^\x(F_v)\), and so we have a canonically attached
dominant cocharacter \(\lambda_v\in\CoCharG(T^\AD)\), called the
\notion{boundary divisor} of \(a\).

\begin{definition}
    \label[definition]{def:affine_jacquet_rallis_fiber}
    Let \(a\in\FRC_\sM^\heartsuit(\cO_v)\). The \notion{affine Jacquet--Rallis
    fiber} \(\cM_{v}(a)\) is the \(k\)-stack (for the \'etale topology)
    associated with the
    prestack sending a \(k\)-algebra \(R\) to the groupoid
    \begin{align}
        \Set*{m\in \sM(\cO_v\hat{\otimes}_k R)\given
        \chi_\sM(m)=a\in\FRC_\sM(\cO_v\hat{\otimes}_k R)}.
    \end{align}
    Similarly, we define \(\cM_v'(a)\) by replacing \(\sM\) with \(\sM'\).
    For any \(k\)-field embedding \(\bar{v}\colon k_v\to \bar{k}\), we also
    define \(\cM_{\bar{v}}(a)\) (resp.~\(\cM_{\bar{v}}'(a)\)) to be the analogue
    of \(\cM_v(a)\) (resp.~\(\cM_v'(a)\)) by replacing \(\cO_v\hat{\otimes}_k
    R\) by \(\cO_{\bar{v}}\hat{\otimes}_{\bar{k}}R\) for any \(\bar{k}\)-algebra
    \(R\).
\end{definition}

Clearly, we have the canonical isomorphism
\begin{align}
    \cM_v(a)\otimes_k\bar{k}\simeq\prod_{\bar{v}\colon k_v\to
    \bar{k}}\cM_{\bar{v}}(a),
\end{align}
where \(\Gal(k_v/k)\) acts on the direct factors of the right-hand side
simply-transitively by permutations, and similarly for \(\cM_v'(a)\).

\subsubsection{}
Before we study the properties of \(\cM_v(a)\) and \(\cM_v'(a)\), we need
to also recall the usual notion of multiplicative affine Springer fibers (MASFs),
following \cite[\S~4]{Wa25}.

Note that over \(\cO_v\), the unitary group \(G'\)
becomes quasi-split, and so the formulations in \cite[\S~4]{Wa25} directly
apply to the unitary side. For the symmetric space side, it is not directly
addressed in \textit{loc. cit.}, but the basic definitions still make sense
using the monoidal symmetric space \(\FRM\), and the geometric results still
hold because \(\FRM\) becomes split after base change to \(\bar{k}\), and so we
only have to worry about rationality questions.

Let \(\bar{a}\in\FRC_\FRM(\cO_v)\) be the image of \(a\), which is generically
invertible and regular semisimple, and we restrict ourselves to only the
MASF associated with \(\bar{a}\). The definition of MASFs in \cite[\S~4]{Wa25}
relies on an \(F_v\)-point \(\gamma\in\FRM(F_v)\)
(resp.~\(\gamma'\in\FRM'(F_v)\)) lying over \(\bar{a}\). In unitary case,
\(\gamma'\) always exists due to a well-known result by Steinberg (one can also
see \cite[Theorem~2.4.24]{Wa25} for a slightly stronger result based on
Steinberg's argument).

In the symmetric case,
since \(a\) is generically strongly regular semisimple, it lifts to a unique
\(F_v\)-point of \(\sM\), which in particular gives a regular semisimple element
\([\gamma]\in\Stack*{\FRM/G}(F_v)\). Since \(G=\GL_n\), any \(G\)-bundle over \(F_v\)
is trivial, and so if we choose a trivialization of the \(G\)-bundle associated
with \([\gamma]\), then it induces the desired \(\gamma\).

\subsubsection{}
A subtlety here is that \(a\) also induces a point
\([\gamma'']\in\Stack*{\FRM'/G'}(F_v)\), whose associated \(G'\)-bundle may or
may not be trivial. By local Tate-Nakayama duality, it has a canonical
obstruction element
\begin{align}
    \Ob_v(a)=\Ob{[\gamma'']}\in\RH^1(F_v,G')\simeq\pi_0(Z_{\dual{\bG}}^{\Gamma_v})^*,
\end{align}
where superscript \(*\) means Pontryagin dual, and \(\Gamma_v\) is
the Galois group of \(F_v\), which acts on \(Z_{\dual{\bG}}=\Gm\)
through \(\Gal(k_v'/k_v)\). This obstruction is non-trivial only if \(v'\) is
non-split over \(v\), in which case the Frobenius element acts on \(\Gm\) by
inversion, and we have
\begin{align}
    \pi_0(Z_{\dual{\bG}}^{\Gamma_v})=\pi_0(\Gm^{\Gal(k_v'/k_v)})=\Set{\pm 1}\simeq
    \bbZ/2\bbZ.
\end{align}
We will study the obstruction \(\Ob_v(a)\) in more details in
\Cref{sub:frobenius_structures_and_obstructions}.

\subsubsection{}
Using \(\gamma\) (resp.~\(\gamma'\)), we can define MASF \(\cM_v^\Hit(a)\)
(resp.~\(\cM_v^{\prime\Hit}(a)\)) the same way as in \cite[\S~4]{Wa25}.
Concretely, \(\cM_v^\Hit(a)(R)\) is the groupoid classifying tuples \((E,\phi,
\iota)\) where \(E\) is a \(G\)-bundle over
\(X_{v,R}=\Spec{\cO_v\hat{\otimes}_k R}\), \(\phi\) is a section of \(\FRM\x^G
E\) over \(X_{v,R}\) such that \((E,\phi)\) lies over \(a\), and \(\iota\) is an
isomorphism over \(X_{v,R}^\bullet\) between \((E,\phi)\) and \((E_0,\gamma)\)
(\(E_0\) means the trivial \(G\)-torsor). \(\cM_v^{\prime\Hit}(a)\) is described
similarly. An explicit description of the \(k\)-points of MASF (in the
unitary case)
involving Cartan double cosets and the boundary divisor of \(\gamma'\) may be found in
\cite[\S\S~4.1.2 and 4.1.12]{Wa25}.

In the symmetric case, we have a well-defined forgetful map
\begin{align}
    \label{eqn:local_forgetful_map_symmetric}
    \cM_v(a)\longto \cM_v^\Hit(a)
\end{align}
by forgetting the vector and the co-vector. Note that for any \(m\in\cM_v(a)\),
the rational isomorphism \(\iota\) of its image under this map is induced by the
fact that \(\sM^\srs\to\FRC_\sM^\srs\) is an isomorphism.

In unitary case, due to the obstruction \(\Ob_v(a)\), we may not have a
forgetful map
\begin{align}
    \label{eqn:local_forgetful_map_unitary_desired}
    \cM_v'(a)\longto \cM_v^{\prime\Hit}(a).
\end{align}
Nevertheless, we may slightly modify the
definition of \(\cM_v^{\prime\Hit}(a)\) by replacing the distinguished point
\((E_0,\gamma')\) by \([\gamma'']\). The resulting functor
\(\cM_v^{\prime\prime\Hit}(a)\) is isomorphic to \(\cM_v^{\prime\Hit}(a)\) after
base changing to \(\bar{k}\), so its geometric properties are still the same,
and only the Galois structure changes. We then have a forgetful map
\begin{align}
    \label{eqn:local_forgetful_map_unitary}
    \cM_v'(a)\longto \cM_v^{\prime\prime\Hit}(a).
\end{align}

\begin{lemma}
    The forgetful maps \eqref{eqn:local_forgetful_map_symmetric} and
    \eqref{eqn:local_forgetful_map_unitary} are monomorphisms of sheaves.
\end{lemma}
\begin{proof}
    We prove the claim for \(\cM_v(a)\) and the proof for \(\cM_v'(a)\) is
    similar.

    We first prove that both functors in
    \eqref{eqn:local_forgetful_map_symmetric} are sheaves.
    Since for any \(k\)-algebra \(R\), we have the inclusion map
    \(\cO_v\hat{\otimes}_k R\subset F_v\hat{\otimes}_k R\), automorphisms of
    any point in \(\cM_v(a)(R)\) (a tuple defined over \(\cO_v\hat{\otimes}_k
    R\)) are determined after restricting to \(F_v\hat{\otimes}_k R\), which
    must be trivial because \(a\) is strongly regular semisimple. Therefore,
    \(\cM_v(a)\) is a sheaf. It is also known (and easily seen) that
    \(\cM_v^\Hit(a)\) is a sheaf. So \eqref{eqn:local_forgetful_map_symmetric}
    is a morphism of sheaves.

    Next, we show injectivity. Suppose \(m\) and \(m'\) are two points in
    \(\cM_v(a)(R)\) with isomorphic image in \(\cM_v^\Hit(a)\). Without loss
    of generality, we may even assume that their images are identical, denoted
    by \((E,\phi,\iota)\). Since \(a\) is strongly regular semisimple over
    \(F_v\), \(m\) and \(m'\), viewed as the \(G\)-bundle \(E\) together with some
    additional data over \(\cO_v\hat{\otimes}_k R\), are isomorphic when
    restricted to \(F_v\hat{\otimes}_k R\). In particular,
    the vector (resp.~co-vector) parts of both \(m\) and
    \(m'\) are isomorphic, via \(\iota\), to the same vector (resp.~co-vector) in
    \((F_v\hat{\otimes}_kR)^n\), the latter being the (trivial)
    vector bundle induced by \([\gamma]\) over \(F_v\hat{\otimes}_k R\).
    Since the map \(\cO_v\hat{\otimes}_kR\to F_v\hat{\otimes}_kR\) is injective,
    this means the vector and co-vector parts of \(m\) and \(m'\) remain
    respectively isomorphic over \(\cO_v\hat{\otimes}_kR\).
    This proves injectivity.
\end{proof}

\subsubsection{}
In order to show the representability of \(\cM_v(a)\) and \(\cM_v'(a)\), we
need an alternative description of these functors using lattices. In fact, we
are going to prove that both functors are representable by a projective scheme
over \(k\), and by Galois descent, it suffices to prove the same claim over
\(\bar{k}\). Therefore, we may assume that \(v'\) is split over \(v\) and
\(\cM_v(a)\cong \cM_v'(a)\). We may also assume \(k_v=k\) for simplicity.

By splitness, we now have a section from \(\FRC_\sM\) to \(\sM\), sending \(a\)
to a distinguished point \(m=(V,\gamma,e,e^\vee)\in\sM(\cO_v)\) where \(V\) is
the trivial vector bundle of rank \(n\), \(\gamma\in\FRM(\cO_v)\), and \(e\)
(resp.~\(e^\vee\)) is a vector (resp.~co-vector) in \(V(\cO_v)\). For any local
\(k\)-algebra \(R\), an \(R\)-point of \(\cM_v(a)\) can then be identified with
an element \(g\in
G(R\lauser{\pi_v})/G(R\powser{\pi_v})=\GL_n(R\lauser{\pi_v})/\GL_n(R\powser{\pi_v})\) such that
\(\Ad_g^{-1}(\gamma)\in \FRM(R\powser{\pi_v})\), and both \(g^{-1}e\) and \(e^\vee
g\) are integral (namely, have \(R\powser{\pi_v}\) entries).

For each \(i\ge 0\), let \(V_i=\wedge^i V\) as before, then \(g\) induces a
full-rank \(R\powser{\pi_v}\)-lattice \(\Lambda=gV(R\powser{\pi_v})\) in
\(V(R\lauser{\pi_v})\), hence also lattices
\begin{align}
    \Lambda_i\defeq\wedge^i\Lambda_i=(\wedge^i g)V_i(R\powser{\pi_v}).
\end{align}
Let \(\gamma_i\in\End(V_i)\) (\(1\le i\le n-1\)) be the elements induced by
\(\gamma\), then we must have
\begin{align}
    \gamma_i\Lambda_i\subset\Lambda_i,
\end{align}
because \(\Ad_g^{-1}(\gamma)\) is integral. An immediate consequence is the
following:
\begin{lemma}
    \label[lemma]{lem:JR_fiber_lattice_description}
    Suppose \(v'\) is split over \(v\), then the functor
    \(\cM_v(a)\cong\cM_v'(a)\) is isomorphic to the functor whose \(R\)-points
    are
    \begin{align}
        \label{eqn:JR_fiber_lattice_description}
        \Set*{R\powser{\pi_v}\text{-lattices }\Lambda\in V(R\lauser{\pi_v})
        \given \gamma_i\Lambda_i\subset\Lambda_i, e\in\Lambda, e^\vee\in\Lambda^\vee}.
    \end{align}
\end{lemma}

\subsubsection{}
Recall that to \(m=(V,\gamma,e,e^\vee)\) we can associate pure tensors
\(f_i(m)\) and \(f_i^\vee(m)\), defined inductively
(cf.~\Cref{sub:deformed_quotient_stack}). These tensors are integral because
\(m\) is. Similarly, the tuple \(m'=(V,\Ad_g^{-1}(\gamma), g^{-1}e,e^\vee g)\) is
also integral, so its associated tensors \(f_i(m')\) and \(f_i^\vee(m')\) are
again integral. It is not hard to see that
\begin{align}
    f_i(m)&=(\wedge^i g)f_i(m')\in \Lambda_i,\\
    f_i^\vee(m)&=f_i^\vee(m')(\wedge^i g^{-1})\in \Lambda_i^\vee.
\end{align}

\begin{lemma}
    \label[lemma]{lem:coarse_lattice_bounds}
    There exists some integer \(N\) depending only on \(a\), such that for
    any \(\Lambda\in\cM_v(a)(R)\), we have
    \begin{align}
        \pi_v^{N}V(R\powser{\pi_v})\subset\Lambda\subset\pi_v^{-N}V(R\powser{\pi_v}).
    \end{align}
\end{lemma}
\begin{proof}
    First, we have \(F_v\)-linearly independent vectors \(e\) and
    \(f_1(m)=\gamma_1e\) in \(\Lambda\). Since
    \(\gamma_1\Lambda\subset\Lambda\), the vector \(\gamma_1^2e\) is contained
    in \(\Lambda\) as well. Since \(a\) is generically invertible, \(f_2(m)\) is
    equal to \(c_2(\gamma_1e\wedge\gamma_1^2e)\) for some \(c_2\in F_v^\x\). Such
    \(c_2\) depends on the abelianization part of \(a\). By induction, and the
    fact that \(f_n(m)\in F_v^\x\) because \(a\) is generically strongly regular
    semisimple, we obtain linearly independent vectors in \(\Lambda\):
    \begin{align}
        e,\gamma_1e,\ldots,\gamma_1^{n-1}e.
    \end{align}
    They generate a sublattice of \(\Lambda\), and thus \(\Lambda\) must be
    bounded below. Similarly, \(\Lambda^\vee\) must also be bounded below,
    which implies that \(\Lambda\) is bounded above. Both bounds depend only on
    \(\gamma\) or \(a\) as claimed.
\end{proof}

\begin{corollary}
    Both \(\cM_v(a)\) and \(\cM_v'(a)\) are representable by projective
    \(k\)-schemes.
\end{corollary}
\begin{proof}
    By \Cref{lem:coarse_lattice_bounds}, one
    may embed \(\cM_{\bar{v}}(a)\) in a Grassmannian of a sufficiently large
    vector space, and \Cref{lem:JR_fiber_lattice_description} shows that it is
    cut out by closed conditions. By Galois descent, the same is true for
    both \(\cM_v(a)\) and \(\cM_v'(a)\).
\end{proof}


\subsection{Frobenius structures and obstructions} 
\label{sub:frobenius_structures_and_obstructions}

In this subsection, we describe the obstruction \(\Ob_v(a)\) on the unitary side
in more explicit ways. A closely related question is to describe the
Frobenius structure of the lattice description
\eqref{eqn:JR_fiber_lattice_description} induced by either \(\cM_v(a)\) or
\(\cM_v'(a)\). We will need both for local computations in some simple cases.

\subsubsection{}
We first describe \(\Ob_v(a)\). We assume that \(v'\) is non-split over \(v\),
and let \(F_v'=F_v\otimes_{k_v}k_v'\). Let \(\sigma\) be the non-trivial
element in \(\Gal(F_v'/F_v)\), which is identified with the Frobenius element of
\(k_v\). Let \(\sigma\) (resp.~\(g\mapsto \bar{g}\)) be the Frobenius action on
\(G'(F_v')=\GL_n(F_v')\) induced by \(G'\) (resp.~the split \(F_v\)-group
\(\GL_n\)), then
\begin{align}
    \sigma(g)=\tp{\bar{g}}^{-1}
\end{align}
for all \(g\in G'(F_v')\).
Since \(\GL_n\)-torsors are trivial over \(F_v'\), any
\(G'\)-torsor over \(F_v\) is induced by a cocycle
\begin{align}
    \sigma\longmapsto g\in G'(F_v'),
\end{align}
where \(g\) is such that \(g\sigma(g)=1\), or in other words,
\(g=\tp{\bar{g}}\) is Hermitian.

Let \(A'=G'/(G')^\Der\). It is a non-split \(1\)-dimensional torus whose
\(F_v\)-points are norm-\(1\) elements in \((F_v')^\x\). Then the cocycle given by
\(g\in G'(F_v')\) is a coboundary if and only if \(\det{g}\in A'(F_v')=(F_v')^\x\)
has even valuation.

\subsubsection{}
Consider \(a\in\FRC_\sM^\heartsuit(\cO_v)\) and its unique \(F_v\)-lift
\(m'\in\sM'(F_v)\). The image of \(m'\) in \(\Stack{\FRM'/G'}(F_v)\) we denote
by \([\gamma'']\) as in \Cref{sub:affine_jacquet_rallis_fibers}, and we have
seen that there exists some \(\gamma'\in\FRM'(F_v)\) lying over the image
\(\bar{a}\in\FRC_\FRM(\cO_v)\) of \(a\).

The canonical point \([\gamma'']\) is a \(G'\)-torsor \(\cE'\) over \(F_v\) together with a
\(G'\)-equivariant map
\begin{align}
    \phi'\colon \cE'\longto \FRM',
\end{align}
whose scheme-theoretic image is a closed \(F_v\)-subvariety
\(O_{[\gamma'']}\) in \(\FRM'_{F_v}\). Let \(\FRJ_{\bar{a}}'\) be the regular
centralizer of the \(G'\)-action at \(\bar{a}\), then \(\phi'\) induces a
\(\chi_{\FRM'}^*\FRJ_{\bar{a}}'\)-torsor
\begin{align}
    \phi'_{\gamma'}\colon \cE'\longto O_{[\gamma'']},
\end{align}
because \(\bar{a}\) is generically regular semisimple.

We know that the homogeneous \(G'\)-space \(O_{[\gamma'']}\) contains at least
one \(F_v\)-point, namely \(\gamma'\), and so the fiber of \(\phi'_{\gamma'}\) over
\(\gamma'\) is a \(\FRJ_{\bar{a}}'\)-torsor \(\mu_{\FRJ'}\) over \(F_v\). Using the
canonical isomorphism between the centralizer
\(G'_{\gamma'}=\Cent_{G'}(\gamma')\) and \(\FRJ_{\bar{a}}'\), we view
\(\mu_{\FRJ'}\) as a \(G'_{\gamma'}\)-torsor, which then induces a
\(G'\)-torsor \(\mu'\).
\begin{lemma}
    \label[lemma]{lem:J_prime_torsor_vs_G_prime_torsor}
    The isomorphism class of \(\mu'\) is independent of the choice of
    \(\gamma'\). In fact, we have \(\mu'\cong \cE'\).
\end{lemma}
\begin{proof}
    Suppose \(\gamma'''\) is another \(F_v\)-point lying over \(\bar{a}\), and
    let \(\mu'''\) be the corresponding \(G'\)-torsor. Then
    we may find \(h\in G'(\bar{F}_v)\) such that
    \(\Ad_h(\gamma')=\gamma'''\). Then it is easy to check that \(\mu'\) and
    \(\mu'''\) differ by the coboundary \(s\mapsto s(h)^{-1}h\) (where
    \(s\in\Gal(\bar{F}_v/F_v)\)), and so they are
    isomorphic. Showing that \(\mu'\cong \cE'\) is also straightforward: choose an
    arbitrary \(\bar{F}_v\)-point in \(\cE'\) that maps to \(\gamma'\) by
    \(\phi'_{\gamma'}\), one can easily show that \(\cE'\) and \(\mu'\) are induced by
    exactly the same cocycle. We leave the details to the reader.
\end{proof}

\subsubsection{}
The Langlands dual group \(\dual{\bJ}\) of \(\FRJ_{\bar{a}}'\) may be identified
with the maximal torus \(\dual{\bT}\subset\dual{\bG}\) and such identification
is well-defined up to \(\dual{\bG}\)-conjugacy. In particular, we have a
canonical embedding of tori
\begin{align}
    Z_{\dual{\bG}}\longto \dual{\bJ}.
\end{align}
The \(\Gamma_v\)-action on \(\dual{\bJ}\) may be identified with
some \(W\)-twisted \(\Gamma_v\)-action on \(\dual{\bT}\), and such
identification depends on \(\gamma'\). Since
\(W\) acts on \(Z_{\dual{\bG}}\) trivially, we further have a canonical
embedding
\begin{align}
    Z_{\dual{\bG}}^{\Gamma_v}\longto \dual{\bJ}^{\Gamma_v}.
\end{align}
This leads to a canonical commutative diagram:
\begin{equation}
    \begin{tikzcd}
        \RH^1(F_v,\FRJ_{\bar{a}}') \ar[r,"\sim"]\ar[d] & \pi_0(\dual{\bJ}^{\Gamma_v})^* \ar[d]\\
        \RH^1(F_v,G') \ar[r,"\sim"] & \pi_0(Z_{\dual{\bG}}^{\Gamma_v})^*
    \end{tikzcd}.
\end{equation}
\Cref{lem:J_prime_torsor_vs_G_prime_torsor} shows that \(\Ob_v(a)\) is equal to
\(\mu'\) for any choice of \(\gamma'\).

\subsubsection{}
To further explicate \(\mu'\), we need to use the information contained in
\(m'\). Over \(F_v'\), we can extend \(m'\) to an \(\cO_v'\)-point \(\tilde{m}\) in \(\sM'\).
Explicitly, \(\tilde{m}\) is a tuple \((V,\gamma,e,e^\vee)\) where \(V\) is
the free \(\cO_v'\)-module \((\cO_v')^n\), \(\gamma\in \FRM'(\cO_v')\),
\(e\in V(\cO_v')\) and \(e^\vee\in V^\vee(\cO_v')\). Since both \(\tilde{m}\)
and \(\sigma(\tilde{m})\) lies over \(a\), there exists a unique \(h\in
G'(F_v')\) such that
\begin{align}
    \sigma\bigl(\gamma,e,e^\vee\bigr)
    =\bigl(\Ad_h^{-1}(\gamma),h^{-1}e,e^\vee h\bigr).
\end{align}
This implies that \(\sigma(h)^{-1}h^{-1}\) lies in the centralizer of tuple
\((\gamma,e,e^\vee)\) (the vector bundle \(V\) is always trivial), hence must
be trivial because \(a\) is generically strongly regular semisimple. In other
words,
\begin{align}
    h=\sigma(h)^{-1}=\tp{\bar{h}}
\end{align}
is Hermitian, hence induces a \(G'\)-torsor over \(F_v\). It is a non-trivial
torsor if and only if \(\Ob_v(a)\) is, hence must be equal to the latter. Recall the
discriminant function (\ref{def:extended_disc_divisor}) \(\Disc=\Disc_\bsM=\bFf_n^\vee\bFf_n=f_n^\vee f_n\) on \(\FRC_\sM\).
We claim that \(\Ob_v(a)\) may be described by the valuation of \(\Disc(a)\):
\begin{proposition}
    \label[proposition]{prop:local_obstruction_and_valuation_parity}
    Suppose \(v'\) is non-split over \(v\), then the obstruction \(\Ob_v(a)\) is
    non-trivial if and only if \(\val_v(\Disc(a))\) is odd.
\end{proposition}
\begin{proof}
    Let \(\gamma_1\) be the image of \(\gamma\) in \(\End(V)(\cO_v')\), and for
    \(1\le i\le n-1\), let \(z_i\in \cO_v'-\Set{0}\) be the coordinate of the
    abelianization of \(\gamma\) corresponding to the \(i\)-th simple root.
    Then it is straightforward to compute:
    \begin{align}
        f_n(\tilde{m})
        &=\prod_{i=1}^{n-1}z_i^{\frac{(n-i)(n-i-1)}{2}}\bigl(e\wedge \gamma_1
        e\wedge\cdots\wedge \gamma_1^{n-1} e\bigr),\\
        f_n^\vee(\tilde{m})
        &=\prod_{i=1}^{n-1}z_i^{\frac{(n-i)(n-i-1)}{2}}\bigl(e^\vee\gamma_1^{n-1}\wedge\cdots\wedge
            e^\vee\gamma_1 \wedge e^\vee\bigr),
    \end{align}
    and so
    \begin{align}
        \val_v(\Disc(a))\equiv \val_v(\Pair{e\wedge \gamma_1 e\wedge\cdots\wedge
            \gamma_1^{n-1} e}{e^\vee\gamma_1^{n-1}\wedge\cdots\wedge
            e^\vee\gamma_1 \wedge e^\vee})\bmod 2.
    \end{align}
    If we regard \(e\) as a column vector and \(e^\vee\) as a row vector, the
    condition \(\sigma(e,e^\vee)=(h^{-1}e,e^\vee h)\) is equivalent to saying
    \begin{align}
        e^\vee =\tp{\bar{e}}h^{-1}=\tp{\bar{e}}\tp{\bar{h}}^{-1}.
    \end{align}
    Therefore, we have
    \begin{align}
        e^\vee\gamma_1^{n-1}\wedge\cdots\wedge e^\vee\gamma_1 \wedge e^\vee
        &=
        \tp{\bar{e}}h^{-1}\gamma_1^{n-1}\wedge\cdots\wedge \tp{\bar{e}}h^{-1}\gamma_1
        \wedge \tp{\bar{e}}h^{-1}\\
        &=
        \bigl(\tp{\bar{e}}h^{-1}\gamma_1^{n-1}h\wedge\cdots\wedge
            \tp{\bar{e}}h^{-1}\gamma_1h
        \wedge \tp{\bar{e}}\bigr)\det{h^{-1}}.
    \end{align}
    Since we also have \(\sigma(\gamma)=\Ad_h^{-1}(\gamma)\), we have
    \begin{align}
        h^{-1}\gamma_1h=\sigma(\gamma)_1=\left(\prod_{i=1}^{n-1}\bar{z}_i^{n-i}\right)^{\frac{2}{n}}\tp{\bar{\gamma_1}}^{-1}.
    \end{align}
    This  implies that
    \begin{align}
        \tp{\bar{e}}h^{-1}\gamma_1^{n-1}h\wedge\cdots\wedge \tp{\bar{e}}h^{-1}\gamma_1h
        \wedge \tp{\bar{e}}
        &=
        \bigl(\tp{\bar{e}}\tp{\bar{\gamma_1}}^{1-n}\wedge\cdots\wedge
            \tp{\bar{e}}\tp{\bar{\gamma_1}}^{-1}
        \wedge \tp{\bar{e}}\bigr)\left(\prod_{i=1}^{n-1}\bar{z}_i^{n-i}\right)^{n-1},
    \end{align}
    and so
    \begin{align}
        \val_v\bigl(e^\vee\gamma_1^{n-1}\wedge\cdots\wedge e^\vee\gamma_1 \wedge e^\vee\bigr)
        &=\val_v\bigl(\tp{\bar{e}}\tp{\bar{\gamma_1}}^{1-n}\wedge\cdots\wedge
            \tp{\bar{e}}\tp{\bar{\gamma_1}}^{-1}
        \wedge \tp{\bar{e}}\bigr)+(n-1)\sum_{i=1}^{n-1}(n-i)\val_v(z_i)\\
        &= \val_v\bigl(\gamma_1^{-(n-1)} e\wedge \cdots\wedge\gamma_1^{-1} e\wedge e\bigr)
        +(n-1)\sum_{i=1}^{n-1}(n-i)\val_v(z_i)\\
        &=\val_v\bigl(e\wedge \gamma_1 e\wedge\cdots\wedge
            \gamma_1^{n-1} e\bigr)
            +(n-1)\left[\sum_{i=1}^{n-1}(n-i)\val_v(z_i)-\val_v(\det(\gamma_1))\right].
    \end{align}
    As a result, it suffices to prove that
    \begin{align}
        (n-1)\left[\sum_{i=1}^{n-1}(n-i)\val_v(z_i)-\val_v(\det(\gamma_1))\right]
    \end{align}
    is even. Indeed, write \(\gamma\) as \((z,x)\) for some ramified element \(z\in
    T^\SC(k_v'\lauser{\pi_v^{1/n}})\) and \(x\in
    G^\SC(k_v'\lauser{\pi_v^{1/n}})\), then we have \(z_i=\Rt_i(z)\) and
    \begin{align}
        \det{\gamma_1}=\Wt_1(z)^n=\prod_{i=1}^{n-1}z_i^{n-i},
    \end{align}
    which clearly shows that
    \begin{align}
        \sum_{i=1}^{n-1}(n-i)\val_v(z_i)-\val_v(\det(\gamma_1))=0.
    \end{align}
    This finishes the proof.
\end{proof}

\begin{corollary}
    When either \(v'\) is split over \(v\) or \(\val_v(\Disc(a))\) is even, we
    have a \(k\)-isomorphism
    \begin{align}
        \cM_v^{\prime\Hit}(a)\cong\cM_v^{\prime\prime\Hit}(a).
    \end{align}
    Consequently, the forgetful map \eqref{eqn:local_forgetful_map_unitary}
    induces a forgetful map \eqref{eqn:local_forgetful_map_unitary_desired}.
\end{corollary}
\begin{proof}
    Immediate from \Cref{prop:local_obstruction_and_valuation_parity} and the
    discussions in \Cref{sub:affine_jacquet_rallis_fibers}.
\end{proof}

\subsubsection{}
With \(\Ob_v(a)\) completely described, we can now look at the
Frobenius action on \eqref{eqn:JR_fiber_lattice_description} induced by
\(\cM_v'(a)\). We will also describe the similar action induced by \(\cM_v(a)\).
For this purpose, it suffices to let \(R=k_v'\) in
\eqref{eqn:JR_fiber_lattice_description}. We let
\(\Lambda_0=V(\cO_v')=(\cO_v')^n\) be the standard lattice in \(V(F_v')\).
As before, we let \(\tilde{m}=(V,\gamma,e,e^\vee)\) be the image of \(a\) under
the section over \(\cO_v'\), and \(h\in G'(F_v')\) is the canonical Hermitian
element induced by the \(F_v\)-structure on \(\tilde{m}\).

The usual \(\sigma\)-action sending \((\gamma,e,e^\vee)\) to
\(\sigma(\gamma,e,e^\vee)\) corresponds to
sending \(g\Lambda_0\) to \(\tp{\bar{g}}^{-1}\Lambda_0\), while \(\Ad_h^{-1}\)
sends \(g\Lambda_0\) to \(hg\Lambda_0\). Since \((\gamma,e,e^\vee)\) is fixed by
the \(\sigma\)-action induced by \(\tilde{m}\), the \(\sigma\)-action on
\eqref{eqn:JR_fiber_lattice_description} induced by that of \(\cM_v'(a)\) is
\begin{align}
    g\Lambda_0\longmapsto h^{-1}\tp{\bar{g}}^{-1}\Lambda_0.
\end{align}
If \(g\Lambda_0\) is a fixed point of this action, then we must have
\begin{align}
    \tp{\bar{g}}hg\in G'(\cO_v'),
\end{align}
which is possible if and only if either \(v'\) is split over \(v\) or
\(\det{h}\) has even valuation, or equivalently, \(\Ob_v(a)\) is trivial.

\subsubsection{}
Suppose we define a bilinear form
\begin{align}
    (\blank,\blank)\colon V(F_v')\x V(F_v')&\longto F_v'\\
    (x,y)&\longmapsto \tp{\bar{x}}h^{-1}y,
\end{align}
then 
\begin{align}
    \label{eqn:JR_fiber_lattice_description_unitary}
    \cM_v'(a)(k_v)=
    \Set*{\cO_v'\text{-lattices }\Lambda\in V(F_v') \given
        \begin{array}{l}
            \Lambda\text{ is self-dual with respect to }(\blank,\blank),\\
            \gamma_i\Lambda_i\subset\Lambda_i, e\in\Lambda, e^\vee\in\Lambda^\vee
        \end{array}
    }.
\end{align}
It is non-empty if and only if \(\Ob_v(a)\) is trivial, or equivalently, either
\(v'\) is split over \(v\) or \(\Disc(a)\) has even valuation.

\subsubsection{}
Similarly, there is a canonical element \(s\in\GL_n(F_v')\) with \(s\bar{s}=1\) induced by
\(\tilde{m}\) and the Frobenius structure on \(\cM_v(a)\), and
\begin{align}
    \label{eqn:JR_fiber_lattice_description_symmetric}
    \cM_v(a)(k_v)=
    \Set*{\cO_v'\text{-lattices }\Lambda\in V(F_v') \given
        \begin{array}{l}
            s\Lambda=\bar{\Lambda},\\
             \gamma_i\Lambda_i\subset\Lambda_i, e\in\Lambda, e^\vee\in\Lambda^\vee
        \end{array}
    }.
\end{align}

\subsubsection{}
There is an involution on the set \(\cM_v(a)(k_v)\) defined as follows:
write \(\gamma=(z,x)\) where \(z\in T^\SC(k_v'\lauser{\pi_v^{1/n}})\) and \(x\in
G^\SC(k_v'\lauser{\pi_v^{1/n}})\). We claim \(\gamma^*=(z,\tp{x})\) has the
same image in \(\FRC_\FRM\) as \(\gamma\): indeed, its abelianization is the
same as that of \(\gamma\), and clearly \(\Wt_i(z)\wedge^i\tp{x}\) has the
same trace as \(\Wt_i(z)\wedge^i x\). Since \(\gamma^*_i=\tp{\gamma_i}\) for
each \(1\le i\le n-1\), \(\gamma^*\in\FRM(\cO_v')\) as well. Let
\begin{align}
    \tilde{m}^*=(V,\gamma^*,\tp{e}^\vee,\tp{e})\in \sM(\cO_v').
\end{align}
Note that for any \(e\in V\), the transpose of the wedging map \(e\wedge\) is
precisely the contraction by \(\tp{e}\), and as a result, \(\tilde{m}^*\) has
the same image of \(\tilde{m}\) in \(\FRC_\sM\). This shows that there exists
a unique element \(\FRd\in \GL_n(F_v')\) such that
\begin{align}
    (\gamma^*,\tp{e}^\vee,\tp{e})=(\Ad_{\FRd}(\gamma),\FRd e,e^\vee \FRd^{-1}).
\end{align}
If we let \(s^*\in\GL_n(F_v')\) be the canonical element induced by
\(\tilde{m}^*\), then we have
\begin{align}
    s^*=\tp{s}^{-1}=\bar{\FRd}s\FRd^{-1},
\end{align}
as well as an alternative description of \(\cM_v(a)(k_v)\):
\begin{align}
    \cM_v(a)(k_v)\cong \cM_v(a)(k_v)^*\defeq
    \Set*{\cO_v'\text{-lattices }\Lambda\in V(F_v') \given
        \begin{array}{l}
            s^*\Lambda=\bar{\Lambda},\\
            \gamma_i^*\Lambda_i\subset\Lambda_i, \tp{e}^\vee\in\Lambda, \tp{e}\in\Lambda^\vee
        \end{array}
    }.
\end{align}
The involution on \(\cM_v(a)(k_v)\) is then given by
\begin{align}
    \label{eqn:involution_on_symmetric_side}
    \Lambda=g\Lambda_0\longmapsto \Lambda^*\defeq \FRd^{-1}\tp{g}^{-1}\Lambda_0.
\end{align}

\subsubsection{}
We make two observations about the involution
\eqref{eqn:involution_on_symmetric_side}. First, recall we have the tensors
\(f_n(\tilde{m})\) and \(f_n^\vee(\tilde{m})\) which are defined over \(F_v\).
The valuations \(\val_v\bigl(\wedge^n \Lambda/f_n(\tilde{m})\bigr)\) and
\(\val_v\bigl(f_n^\vee(\tilde{m})/{\wedge^n \Lambda}\bigr)\) are
well-defined since \(a\) is generically strongly regular semisimple, and we have
\begin{align}
    \val_v\bigl(\wedge^n \Lambda/f_n(\tilde{m})\bigr)
    +\val_v\bigl(f_n^\vee(\tilde{m})/{\wedge^n \Lambda}\bigr)
    =\val_v(\Disc(a)).
\end{align}
Moreover, it is straightforward to see that
\begin{align}
    \val_v\bigl(\wedge^n \Lambda/f_n(\tilde{m})\bigr)
    &=\val_v\bigl(f_n^\vee(\tilde{m}^*)/{\wedge^n\Lambda^*}\bigr),\\
    \val_v\bigl(f_n^\vee(\tilde{m})/{\wedge^n\Lambda}\bigr)
    &=\val_v\bigl(\wedge^n \Lambda^*/f_n(\tilde{m}^*)\bigr),
\end{align}
which implies:
\begin{lemma}
    \label[lemma]{lem:valuation_under_involution}
    We have
    \begin{align}
        \val_v\bigl(\wedge^n \Lambda^*/f_n(\tilde{m}^*)\bigr)=\val_v(\Disc(a))-
        \val_v\bigl(\wedge^n \Lambda/f_n(\tilde{m})\bigr).
    \end{align}
\end{lemma}

A second observation is the following: suppose \(\Lambda=g\Lambda_0\) and
\(\Lambda^*=\FRd^{-1}\tp{g}^{-1}\Lambda_0\), and suppose the image of
\(\Ad_g^{-1}(\gamma)\) in \(G^\AD\) is contained in the Cartan double coset
\begin{align}
    G^\AD(\cO_v')\pi_v^{\lambda}G^\AD(\cO_v'),
\end{align}
then the image of \(\Ad_{\tp{g}\FRd}(\gamma)\) in \(G^\AD\) is contained is the
same double coset, because any Cartan double coset is stable under
transposition.


\subsection{Matching orbits} 
\label{sub:matching_orbits}

In this subsection,
we describe the process for matching local conjugacy classes. This uses some
basic invariant-theoretic results that closely mirror the
monoidal setting in \Cref{sec:invariant_theory} and the arguments will be
similar. For this reason, we will omit less important details.

\subsubsection{}
\label{ssub:review_of_SYMS}
Recall we have the group \(\bG'=\OGT_*\bG\) defined by Weil restriction, whose
\(\cO_v\)-points are \(\bG(\cO_v')=\GL_n(\cO_v')\). The group \(\bG\) naturally
embeds into \(\bG'\) as the fixed-point subgroup under the Frobenius action
\(\sigma\). For any \(g\in\bG'\), we also denote \(\sigma(g)\) by \(\bar{g}\). Let
\begin{align}
    \SYMS_n'\defeq \bG'\git \bG.
\end{align}
Then \(\SYMS_n'\) is a spherical \(\bG'\)-variety.

We have an anti-involution \(\iota\) on \(\bG'\) being the inversion, and let
\begin{align}
    \SYMS_n\defeq (\bG')^{\iota\sigma}.
\end{align}
There is a natural map
\begin{align}
    \bG'&\longto \SYMS_n\\
    g&\longmapsto g\bar{g}^{-1},
\end{align}
and this map clearly factors through \(\SYMS_n'\). By Hilbert 90, it
induces an isomorphism
\begin{align}
    \SYMS_n'\stackrel{\sim}{\longto}\SYMS_n.
\end{align}
Indeed, the question is \'etale-local, so for any
\(\cO_v\)-algebra \(R\), we may replace
\(R\) by an \'etale double cover so that \(\bG'\) become split over \(R\). In
other words, \(\bG'(R)\) is isomorphic to \(\bG(R)\x \bG(R)\) and \(\sigma\)
interchanges the two factors. Then \(\SYMS_n(R)\subset\bG(R)\x\bG(R)\) is the
anti-diagonal, and both injectivity and surjectivity become obvious. 

\subsubsection{}
Let \(G=\bG\). There is a natural \(G\)-action on \(\SYMS_n'\) by multiplication on
the left. This translates to the \(G\)-conjugation
action on \(\SYMS_n\). Similarly, the \(\bG'\)-action on \(\SYMS_n'\) corresponds to the
action on \(\SYMS_n\) induced by the \(\sigma\)-twisted adjoint action of \(\bG'\).
To better integrate this subsection with the rest of the paper, we will, from
now on, solely use the description given by \(\SYMS_n\).

\subsubsection{}
On the other hand, we have an involution \(\tau\) on \(\bG'\) given by
transpose-inverse. As before we have
\begin{align}
    G'\defeq (\bG')^{\tau\sigma}.
\end{align}
We will consider the adjoint \(G'\)-action on itself.

\subsubsection{}
The \notion{abelianization} of \(\SYMS_n\) is defined to be the GIT-quotient
\begin{align}
    \FRA_n\defeq \SYMS_n\git(\bG')^\SC.
\end{align}
It is a one-dimensional torus that is split over \(\cO_v'\). If we view
\(\SYMS_n\) as the subscheme of \(\bG'\), then the natural quotient map is simply
the determinant
\begin{align}
    \det\colon \SYMS_n\longto \FRA_n.
\end{align}

Similarly, the abelianization of \(G'\) is the GIT-quotient \(G'\git(G')^\SC\x
(G')^\SC\) by left and right translations. It is easy to see that this quotient
is isomorphic to \(\FRA_n\).
The natural quotient map is also the determinant
\begin{align}
    \det\colon G'\longto\FRA_n.
\end{align}

\subsubsection{}
If we replace \(\bG\) with \(\bG^\SC=\SL_n\) in the discussion above, we obtain
the analogues \(\SYMS_n^\SC\) and \((G')^\SC\). Notice that \(\SYMS_n^\SC\) embeds into
\(\FRM\) as the fiber over \(1\in\FRA_\FRM\), and similarly \((G')^\SC\) embeds
into \(\FRM'\). Let \(V\) be the standard representation of \(\bG\), then the
traces of fundamental representations \(V_i=\wedge^i V\) gives a complete set of
invariants \(a_1,\ldots,a_{n-1}\) of the adjoint \(\bG\)-action on \(\bG^\SC\).

Both \(\iota\sigma\) and \(\tau\sigma\) induce the same automorphism on the invariants:
\begin{align}
    \bigl(a_1,\ldots,a_{n-1}\bigr)\longmapsto \bigl(\bar{a}_{n-1},\ldots,\bar{a}_{1}\bigr).
\end{align}
It implies that we have
\begin{align}
    \FRC_n^\SC\defeq \SYMS_n^\SC\git G\simeq (G')^\SC\git G'\cong \bbA^{n-1}.
\end{align}
This result can be directly deduced from the ones in \Cref{sec:invariant_theory}.

The representations \(V_i\) naturally extends to \(\bG\), and their traces,
together with the determinant, give a complete set of invariants of the
\(\bG\)-action on \(\bG\). Both \(\iota\sigma\) and \(\tau\sigma\) induce the
same map
\begin{align}
    \bigl(\det,a_1,\ldots,a_{n-1}\bigr)\longmapsto
    \bigl(\bar{\det}^{-1},\bar{\det}^{-1}\bar{a}_{n-1},\ldots,\bar{\det}^{-1}\bar{a}_{1}\bigr).
\end{align}
As a result, we still have
\begin{align}
    \FRC_n\defeq \SYMS_n\git G\simeq G'\git G',
\end{align}
which is a \(\FRC_n^\SC\)-bundle over \(\FRA_n\).

Note that although \(\bG\) embeds into \(\bM\), this embedding is not
equivariant under \(\iota\sigma\) nor \(\tau\sigma\), so the last result above
cannot be deduced from the ones in \Cref{sec:invariant_theory} without
modifications. However, the core argument is still very similar and not
difficult using linear algebra anyway, so we leave details to the reader.

\subsubsection{}
Recall we have \(\bH=\GL_{n-1}\subset \bG\), as well as \(H\subset G\) and
\(H'\subset G'\). They can all be viewed as the respective stabilizer subgroups
of the standard pair \((\bFe,\bFe^\vee)\in V\x V^\vee\). For any \(A\in\bG\), we
have additional \(\bH\)-invariants
\begin{align}
    b_i'\defeq \bFe^\vee A^i\bFe,\quad 0\le i\le n-1.
\end{align}
It is known in the literature that
\begin{align}
    \det,a_1,\ldots,a_{n-1},b_1',\ldots,b_{n-1}'
\end{align}
give a complete set of invariants of the \(\bH\)-action on \(\bG\) by
conjugation. However, the Frobenius actions on \(b_i'\) are complicated to
express, and so we will use an alternative set similar to the ones in
\Cref{sec:invariant_theory}. Let \(b_i\) be the trace
of the composition map
\begin{align}
    V_{i-1}\stackrel{\bFe\wedge}{\longto}V_{i}\stackrel{\wedge^i A}{\longto}
    V_{i}\stackrel{\bFe^\vee }{\longto}V_{i-1}.
\end{align}
Then
\begin{align}
    \det,a_1,\ldots,a_{n-1},b_1,\ldots,b_{n-1}
\end{align}
is another complete set of \(\bH\)-invariants, and both \(\iota\sigma\) and
\(\tau\sigma\) induce the same map
\begin{multline}
    \bigl(\det,a_1,\ldots,a_{n-1},b_1,\ldots,b_{n-1}\bigr)\\\longmapsto
    \bigl(\bar{\det}^{-1},\bar{\det}^{-1}\bar{a}_{n-1},\ldots,\bar{\det}^{-1}\bar{a}_{1},
    \bar{\det}^{-1}(\bar{a}_{n-1}-\bar{b}_{n-1}),\ldots,\bar{\det}^{-1}(\bar{a}_1-\bar{b}_{1})\bigr).
\end{multline}
As a result, we have a matching of GIT quotients
\begin{align}
    \FRC_{n,\JR}\defeq \SYMS_n\git H\simeq G'\git H'.
\end{align}

\subsubsection{}
Using the natural embedding \(\bG\subset\bM\), we may restrict the discriminant
function \(\Disc=\bFf_n^\vee\bFf_n\) to \(\bG\). It is not hard to see to be the
same as the function
\begin{align}
    A\longmapsto \det{(\bFe^\vee A^{i+j}\bFe)_{0\le i,j\le n-1}}.
\end{align}
The Frobenius action \(\iota\sigma\) sends a matrix \(A\) to \(\bar{A}^{-1}\),
and we have
\begin{align}
    \Disc(\bar{A}^{-1})=(\det{\bar{A}})^{2-2n}\Disc(\bar{A})=\bar{(\det{A})^{2-2n}\Disc(A)}.
\end{align}
This implies that the function
\begin{align}
    \Disc_n\colon A\longmapsto(\det{A})^{1-n}\Disc(A)
\end{align}
descends to \(\SYMS_n\). Similarly, \(\tau\sigma\) sends \(A\) to
\(\tp{\bar{A}}^{-1}\), and so \(\Disc_n\) also descends to \(G'\). In fact,
\(\Disc_n\) is clearly invariant under \(\bH\), so it further descends to
\(\FRC_{n,\JR}\).

\begin{definition}\label[definition]{def:strongly_regular_semisimple}
    The \notion{strongly regular semisimple} locus \(\FRC_{n,\JR}^\srs\) in
    \(\FRC_{n,\JR}\) is where the matrix \(A\) is regular semisimple (in the
    usual sense) and \(\Disc_n\) is non-vanishing.
\end{definition}

\subsubsection{}
Similar to the discussions in \Cref{sec:invariant_theory}, the natural maps
\begin{align}
    \Stack*{\SYMS_n/H}\longto \FRC_{n,\JR}\longot \Stack*{G'/H'}
\end{align}
are isomorphisms over \(\FRC_{n,\JR}^\srs\), and so the fibers over the same
locus are \(H\)-torsors or \(H'\)-torsors respectively. Given a point
\(a\in\FRC_{n,\JR}^\srs(F_v)\), since \(H=\GL_{n-1}\) and
\(\RH^1(F_v,H)=0\), there exists a point \(A\in \SYMS_n(F_v)\) lying over
\(a\).

On the other hand, \(\RH^1(F_v,H')\) is trivial if \(v'\) is split over
\(v\), and isomorphic to \(\bbZ/2\bbZ\) otherwise. Similar to
\Cref{sub:frobenius_structures_and_obstructions}, there exists a point \(A'\in
G'(F_v)\) lying over \(a\) if and only if either \(v'\) is split over \(v\) or
\(\Disc_n(a)\) has even valuation. Note that when \(v'\) is non-split over
\(v\) and \(A\in \bG(F_v')\) lies over \(a\), then \(\det{A}\) always has
valuation \(0\) since \(\det{A}\) is a norm-\(1\) element in \(F_v'\), and as a
result, the valuation of \(\Disc_n(A)\) is the same as \(\Disc(A)\).

Returning to \(A\in \SYMS_n(F_v)\), and consider its image in \(\SYMS_n^\AD\) (the
quotient of \(\SYMS_n\) by the central torus). Note that \(\SYMS_n^\AD=\FRM^\x/Z_\FRM\).
Since \(Z_\FRM\) is an induced torus, its torsors over \(F_v\) are trivial,
and so there exists \(\gamma\in\FRM^\x(F_v)\) that has the same image as \(A\) in
\(\SYMS_n^\AD\). Together with \(\bFe\) and \(\bFe^\vee\), it determines a strongly
regular semisimple point \(a_\sM\in \FRC_\sM(F_v)\).

\begin{lemma}
    \label[lemma]{lem:disc_of_group_vs_monoid}
    Suppose \(v'\) is non-split over \(v\). Then for any such \(a_\sM\) induced
    by \(a\), we have
    \begin{align}
        \val_v(\Disc_n(a))\equiv\val_v(\Disc(a_\sM))\bmod 2.
    \end{align}
\end{lemma}
\begin{proof}
    First note that if \(\gamma'\in\FRM(F_v')\) is any point that has the same
    image as \(A\) in \(\SYMS_n^\AD\), we always have
    \begin{align}
        \val_v(\Disc(\gamma'))\equiv \val_v(\Disc(a_\sM))\bmod 2,
    \end{align}
    because the difference of these two valuations is equal to
    \(\Pair{2\rho}{\mu}\) where \(\mu\) lies in the coroot lattice. This claim
    is easily deduced from \cite[Lemma~2.5.1]{Ch19} (see also
    \cite[Lemma~4.1.5]{Wa25}). Then we may assume that \(\gamma'\in\bG(F_v')\)
    and the lemma is immediate because then the valuation of
    \(\det{\gamma'}/\det{A}\) must lie in \(n\bbZ\), whose \((n-1)\)-th
    multiple is even.
\end{proof}

\begin{corollary}
    There exists a point \(A'\in G'(F_v)\) lying over \(a\) if and only if
    \(\Ob_v(a_\sM)\) is trivial for any \(a_\sM\) induced by \(a\).
\end{corollary}

\begin{definition}
    \label[definition]{def:matching_orbits}
    A point \(a\in\FRC_{n,\JR}^\srs(F_v)\) is called a \notion{matching pair}
    if \(\Ob_v(a_\sM)\) is trivial for any (and every) \(a_\sM\) induced by \(a\).
\end{definition}


\subsection{Local transfer map} 
\label{sub:local_transfer_map}

We have established the matching between conjugacy classes, and we now describe
how to match functions. Consider the space \(\cH\) (resp.~\(\cH'\)) of smooth
(equivalently, locally constant) \(\Qlb\)-valued
functions with compact support on \(\SYMS_n(F_v)\) (resp.~\(G'(F_v)\)).

\subsubsection{}
The \notion{spherical Hecke module} of \(\SYMS_n(F_v)\) is the subspace
\(\cH_0\subset\cH\) consisting of functions that are invariant under
\(\bG'(\cO_v)\)-action. In other words, if \(f\in\cH_0\), then for any \(x\in
\SYMS_n(F_v)\) and any \(g\in\bG'(\cO_v)=\GL_n(\cO_v')\), we have
\begin{align}
    (gf)(x)=f(g^{-1}x\bar{g})=f(x).
\end{align}
Similarly, the \notion{spherical Hecke algebra} of \(G'(F_v)\) is the subspace
\(\cH_0'\subset\cH'\) consisting of functions invariant under \(G'(\cO_v)\x
G'(\cO_v)\): if \(f\in\cH_0'\), \(x\in G'(F_v)\) and \((g_1,g_2)\in G'(\cO_v)\x
G'(\cO_v)\), then
\begin{align}
    ((g_1,g_2)f)(x)=f(g_1^{-1}xg_2)=f(x).
\end{align}

\subsubsection{}
The spherical Hecke algebra \(\cH_0'\) can easily be described by the Cartan
decomposition of \(G'(F_v)\). For \(\cH_0\), there is a similar
decomposition of \(\SYMS_n(F_v)\) into \(\bG'(\cO_v)\)-cosets. We shall describe
both with the help of the Cartan decomposition of \(\bG'(F_v)=\GL_n(F_v')\).

Given a \(W\)-orbit of cocharacter \([\lambda]\) in \(\bG'\), we have the
corresponding Cartan double coset
\begin{align}
    \Cartan^{\lambda}\defeq \GL_n(\cO_v')\pi_v^\lambda\GL_n(\cO_v'),
\end{align}
which does not depend on the representative \(\lambda\). Note that if \(v'\) is
split over \(v\), then \(\lambda\) is represented by two cocharacters
\((\lambda_1,\lambda_2)\) in \(\GL_n\), and a single cocharacter in \(\GL_n\) if
\(v'\) is non-split.
If \(C\subset \SYMS_n(F_v)\) is a \(\bG'(\cO_v)\)-coset, then by definition it must
be contained in \(\Cartan^\lambda\) for a unique \([\lambda]\). Moreover, if we
choose \(\lambda\) to be dominant, then since
\(C=\bar{C}^{-1}\), we must have \(\lambda_2=-\lambda_1\) if \(v'\) is split, or
\(-w_0(\lambda)=\lambda\) if \(v'\) is non-split. For convenience, we will
denote this relation by
\begin{align}
    \lambda=\sigma_{\Out}(\lambda).
\end{align}
We also have the
analogous setup by passing to the adjoint group.
The following result is crucial in describing the Cartan decomposition of
\(\SYMS_n(F_v)\):

\begin{lemma}
    \label[lemma]{lem:Cartan_decomp_S_n_crucial}
    Suppose \(\lambda\) is a dominant cocharacter of \(\PGL_n\) and
    \begin{align}
        x_1,x_2\in\Cartan_\AD^\lambda\cap \SYMS_n^\AD(F_v),
    \end{align}
    then we can find \(g\in\PGL_n(\cO_v')\) such that
    \begin{align}
        \label{eqn:Cartan_decomp_S_n_desired_equality}
        g^{-1}x_1\bar{g}=x_2.
    \end{align}
\end{lemma}
\begin{proof}
    The \(\SYMS_n\)-version of this lemma was first proved in \cite{Of04}.
    Here we give a different proof using reductive monoids. It is a slightly
    stronger result because although we can always lift \(\lambda\) to a
    cocharacter of \(\GL_n\), we cannot always find one that is fixed by
    \(\sigma_{\Out}\).

    As discussed before, we have \(\SYMS_n^\AD(F_v)=\FRM^\x/Z_\FRM\), and \(x_1\)
    (resp.~\(x_2\)) may be lifted to a point \(\gamma_1\)
    (resp.~\(\gamma_2\)) in \(\FRM^\x(F_v)\). Moreover, we can always choose
    such lifts so that the abelianizations of \(\gamma_1\) and \(\gamma_2\) are
    both contained in \(\pi_v^\lambda\FRA_\FRM^\x(\cO_v)\). This way, both \(\gamma_1\)
    and \(\gamma_2\) may be extended to \(\cO_v\)-points of the big-cell locus
    \(\FRM^\circ\) (cf.~\cite[\S~2.4.15]{Wa25}).

    The group \(\bM^\x(\cO_v')\) acts on \(\FRM(\cO_v)\) by
    \begin{align}
        g\colon x\longmapsto gx\bar{g}^{-1}.
    \end{align}
    This action is induced by the
    \(\OGT_*\bM^\x\)-action on \(\FRM\), which preserves the
    big-cell locus.
    By Lang's theorem, we may replace \(\gamma_2\) by
    \(g\gamma_2\bar{g}^{-1}\) for some
    \(g\in\bM^\x(\cO_v')\) (correspondingly, \(x_2\) is replaced by \(g_\AD
    x_2\bar{g}_\AD^{-1}\)) so that \(\gamma_1\) and \(\gamma_2\)
    lies over the same point in \(\FRA_\FRM(\cO_v)\).

    Consider the action map
    \begin{align}
        \OGT_*\bG^\SC \x\FRM^\circ\longto \FRM^\circ.
    \end{align}
    After base change to \(\bar{k}\), it becomes isomorphic to the
    \(\bG^\SC\x\bG^\SC\)-action (by left and right multiplications) on
    \(\bM^\circ\), which is a smooth map. It is well-known for very flat
    reductive monoids that the geometric fibers of \(\bM^\circ\to\bA_\bM\) are
    homogeneous \(\bG^\SC\x\bG^\SC\)-spaces with connected stabilizers. This
    means that the transporter in \(\OGT_*\bG^\SC\) from \(\gamma_1\) to
    \(\gamma_2\) is a torsor under a smooth group scheme over \(\cO_v\) with
    connected fibers. By Lang's theorem again, such torsor must be trivial, and
    so we can find some \(g\in\bG^\SC(\cO_v')\) such that
    \(g\gamma_1\bar{g}^{-1}=\gamma_2\). This finishes the proof.
\end{proof}

\begin{corollary}
    If \(C_1,C_2\subset \SYMS_n(F_v)\) are two \(\bG'(\cO_v)\)-cosets
    contained in \(\Cartan^\lambda\) for some dominant cocharacter \(\lambda\)
    in \(\GL_n\), then
    \begin{align}
        C_1=C_2.
    \end{align}
\end{corollary}
\begin{proof}
    It suffices to show that if
    \begin{align}
        x_1, x_2\in \Cartan^\lambda\cap \SYMS_n(F_v),
    \end{align}
    then we can find some \(g\in \GL_n(\cO_v')\) such that
    \eqref{eqn:Cartan_decomp_S_n_desired_equality} holds.

    By \Cref{lem:Cartan_decomp_S_n_crucial},
    we have already proved the same equality in \(\PGL_n(F_v')\) for some
    \(g\in\PGL_n(\cO_v')\). Since \(\GL_n(\cO_v')\) maps surjectively onto
    \(\PGL_n(\cO_v')\), we can lift \(g\) to \(\GL_n(\cO_v')\), still denoted by
    \(g\) for simplicity. Then there exists some \(c\in (F_v')^\x\) such that
    \(g^{-1}x_1\bar{g}=cx_2\). Moreover, \(c\) must have norm \(1\), hence is
    contained in
    \((\cO_v')^\x\). By Lang's theorem, there exists some \(z\in
    (\cO_v')^\x\) such that \(z\bar{z}^{-1}=c\). Consequently, we may replace
    \(g\) by \(z^{-1}g\) to make \eqref{eqn:Cartan_decomp_S_n_desired_equality}
    hold in \(\GL_n(F_v')\), and we are done.
\end{proof}

\subsubsection{}
Using Cartan decomposition of \(\SYMS_n\), we see that \(\cH_0\) has a natural
vector basis in bijection with dominant cocharacters \(\lambda\) of \(\GL_n\)
such that \(\lambda=\sigma_{\Out}(\lambda)\). Recall the element \(\dot{w}_0\) defined by
\eqref{eqn:w_0_matrix} satisfying \(\dot{w}_0^2=(-1)^{n-1}\). If \(v'\)
is non-split over \(v\), choose a
square root of \((-1)^{n-1}\) in \(F_v'\), and let
\begin{align}
    \ddot{w}_0\defeq\sqrt{(-1)^{n-1}}\dot{w}_0,
\end{align}
otherwise let \(\ddot{w}_0=1\) if \(v'\) is split.
Then \(\ddot{w}_0^2=1\), and \(\pi_v^{\lambda}\ddot{w}_0\in \SYMS_n(F_v)\) for
any \(\lambda=\sigma_{\Out}(\lambda)\). Let
\begin{align}
    \Cartan_{n}^\lambda\defeq\Cartan^\lambda\cap
    \SYMS_n(F_v)=\bG'(\cO_v)\pi_v^\lambda\ddot{w}_0.
\end{align}

\begin{corollary}
    \label[corollary]{cor:desc_of_Hecke_0_for_S_n}
    We have a disjoint union of \(\bG'(\cO_v)\)-orbits
    \begin{align}
        \SYMS_n(F_v)=\coprod_{\lambda}\Cartan_n^\lambda,
    \end{align}
    where \(\lambda\) ranges over dominant cocharacters of \(\bG'\) fixed by
    \(\sigma_{\Out}\). Consequently, we have a direct sum decomposition
    \begin{align}
        \cH_0=\bigoplus_\lambda \Qlb\cdot\One_{\Cartan_n^\lambda}.
    \end{align}
\end{corollary}

\subsubsection{}
The Cartan decomposition of \(G'(F_v)\) is well-known since \(G'\) is a
reductive group. It is even quasi-split when restricted to \(F_v\) because it
is unramified. To better compare with the story on \(\SYMS_n\)-side, however, we
will give a slightly different description using monoids.

Indeed, if we examine the proofs on the \(\SYMS_n\)-side, we see that the crucial
ingredient is simply that \(\FRM\) is a Frobenius twist of the split form
\(\bM\), and the \(\OGT_*\bG^\SC\)-action is geometrically isomorphic to the
\(\bG^\SC\x\bG^\SC\)-action on \(\bM\). The same is also true for \(\FRM'\) and
the \((G')^\SC\x (G')^\SC\)-action on \(\FRM'\). Therefore, by repeating the same argument, one can show that
\begin{proposition}
    \label[proposition]{cor:desc_of_Hecke_0_for_G_prime}
    For any dominant \(\lambda\) of \(\bG'\) fixed by
    \(\sigma_{\Out}\), let
    \begin{align}
        \Cartan_n^{\prime\lambda}\defeq \Cartan^\lambda\cap G'(F_v),
    \end{align}
    then it is a single \(G'(\cO_v)\x G'(\cO_v)\)-orbit, and we have
    a disjoint union
    \begin{align}
        G'(F_v)=\coprod_\lambda\Cartan_n^{\prime\lambda}.
    \end{align}
    As a result, we have a direct sum decomposition
    \begin{align}
        \cH_0'=\bigoplus_\lambda\Qlb\cdot\One_{\Cartan_n^{\prime\lambda}}.
    \end{align}
\end{proposition}

\subsubsection{}
One can also try to find a representative of the orbit
\(\Cartan_n^{\prime\lambda}\) similar to \(\pi_v^\lambda\ddot{w}_0\) for
\(\Cartan_n^\lambda\), and we will only describe the non-split case as the split
case is the same as on the \(\SYMS_n\)-side. Let \(c\in\cO_v'\) be such that
\(c\bar{c}^{-1}=(-1)^{n-1}\),
then the map \(\sigma\mapsto c\dot{w}_0\) gives a Frobenius cocycle
in \(G'\). Since \(G'\) is connected, this cocycle is trivial, and so we can
find \(u_0\in G'(\cO_v')\) such that \(c\dot{w}_0=u_0{\tp{\bar{u}}_0}\).
Then one can verify that \(u_0^{-1}\pi_v^\lambda u_0\in G'(F_v)\) and represents
\(\Cartan_n^{\prime\lambda}\).

\begin{remark}
    In \cite{Of04}, there is a simpler system of representatives: one first
    changes every \(-1\) entry in \(\dot{w}_0\) to \(1\), and then
    \(\pi_v^\lambda\dot{w}_0\) (with the modified \(\dot{w}_0\)) will represent
    both \(\Cartan_n^\lambda\) and \(\Cartan_n^{\prime\lambda}\). We choose a
    different system because we want to avoid matrix embedding (which works
    best for classical types only) as much as possible.
\end{remark}

\subsubsection{}
Combining \Cref{cor:desc_of_Hecke_0_for_S_n,cor:desc_of_Hecke_0_for_G_prime},
we see there exists a canonical \notion{spherical transfer map}
\begin{align}\label{eqn:spherical_transfer}
    \cH_0&\longto\cH_0'\\
    \One_{\Cartan_n^\lambda}&\longmapsto \One_{\Cartan_n^{\prime\lambda}}.
\end{align}
This is the same map given by \cite{Of04}. In geometric settings, however, it is
not very convenient to use characteristic functions
\(\One_{\Cartan_n^\lambda}\) and \(\One_{\Cartan_n^{\prime\lambda}}\); instead,
we use an alternative basis given by geometric Satake.

\subsubsection{}
The geometric Satake for reductive group \(G'\) (which becomes quasi-split after
restricting to \(F_v\)) is well-known and
we review this theory following \cite{Zh17}, and then we state its analogy for
\(\SYMS_n\). We will also reformulate both using monoidal language following
\cite{Wa25}.

Let \(\Loop_v{G'}\) (resp.~\(\Arc_v{G'}\)) be the loop (resp.~arc) group at
place \(v\). The affine Grassmannian of \(G'\) is the fpqc quotient
\begin{align}
    \Gr_{v,G'}=\Loop_v{G'}/\Arc_v{G'}.
\end{align}
The arc group \(\Arc_v{G'}\) acts on \(\Gr_{v,G'}\) on the left, and we may
consider the category of \(\Arc_v{G'}\)-equivariant perverse sheaves on
\(\Gr_{v,G'}\). Note that any Frobenius element of \(F_v\) acts on those
sheaves. For arithmetic purposes, such a category is too large; instead, we
consider the full subcategory \(\Sat_{G'}\) consisting of objects where the
Frobenius acts with finite order. This way there is an equivalence between
\(\Sat_{G'}\) and the category of \emph{algebraic} representations of the
\(L\)-group of \(G'\).

By Grothendieck's sheaf-function dictionary, any \(\cF\in\Sat_{G'}\) induces a
\(\Qlb\)-valued function on \(\Gr_{v,G'}(k_v)\) that is \(G'(\cO_v)\)-invariant,
which pulls back to a function on \(G'(F_v)\) that is \(G'(\cO_v)\x
G'(\cO_v)\)-invariant. In other words, we have a natural map
\begin{align}
    \abs{\Sat_{G'}}\longto \cH_0',
\end{align}
where \(\abs{\Sat_{G'}}\) means the set of isomorphism classes in \(\Sat_{G'}\).
For any dominant cocharacter \(\lambda\) in \(G'\),
or equivalently, for any dominant cocharacter of \(\bG'\) fixed by
\(\sigma_{\Out}\), we have a corresponding Schubert variety
\begin{align}
    \Gr_{v,G'}^{\le\lambda}
\end{align}
which is the closure of the \(\Arc_v{G'}\)-orbit corresponding to
\(\Cartan_n^{\prime\lambda}\). 
\begin{definition}
Let \(f_\lambda'\in\cH_0'\) be the function
given by the intersection complex of \(\Gr_{v,G'}^{\le\lambda}\).
\end{definition} 
Then one can
easily show by induction (on dominance order of \(\lambda\)) that \(f_\lambda'\)
for various \(\lambda\) form an alternative basis of \(\cH_0'\).

Sometimes, it is also convenient to consider, in addition to
\(\sigma_{\Out}\)-fixed \(\lambda'\), any
\(\sigma_{\Out}\)-orbit of dominant cocharacters \(\Set{\mu}\) of
\(G'_{\bar{k}}\). Each
cocharacter in such an orbit induces a Schubert variety in \(\Gr_{v,G'}\) over
\(\bar{k}\), and their union \(\SFQ\) is stable under the Frobenius, hence
defined over \(k_v\). The intersection complex of \(\SFQ\) then also induces a
function
\begin{align}
    f_{\Set{\mu}}'
    \in\cH_0',
\end{align}
which is a linear combination of those \(f_\lambda'\)
such that \(\lambda=\sigma_{\Out}(\lambda)\le\mu\) for every \(\mu\) in the
orbit.

\subsubsection{}
\label{ssub:symmetric_equiv_affGr}
For \(\SYMS_n\), we have something similar. Although there is no affine
Grassmannian, we may still consider the quotient stack
\begin{align}
    \Stack*{\Arc_v{\bG'}\backslash\Loop_v{\SYMS_n}},
\end{align}
which is pro-ind-algebraic: indeed, after passing to a quadratic extension of
\(k_v\), we see that the above is isomorphic to the quotient
\begin{align}
    \Stack*{\Arc_v{\bG}\backslash\Loop_v{\bG}/\Arc_v{\bG}}
    =\Stack*{\Arc_v{\bG}\backslash\Gr_{v,\bG}}.
\end{align}
For each \(\lambda=\sigma_{\Out}(\lambda)\), the Cartan coset
\(\Cartan_n^\lambda\) induces a closed substack
\begin{align}
    \label{eqn:truncated_Hecke_stack_for_S_n}
    \Stack*{\Arc_v{\bG'}\backslash\Loop_v{\SYMS_n}^{\le\lambda}},
\end{align}
which is isomorphic to
\begin{align}
    \label{eqn:truncated_Hecke_stack_split}
    \Stack*{\Arc_v{\bG}\backslash \Gr_{v,\bG}^{\le\lambda}}
\end{align}
after passing to a quadratic extension of \(k_v\). Since the dimension of
these stacks are not well-defined, there is no notion of ``intersection
complex'' in the conventional sense. However, the intersection complex
on \(\Gr_{v,\bG}^{\le\lambda}\) descends to \eqref{eqn:truncated_Hecke_stack_split},
and for simplicity we shall call this sheaf the intersection complex of the
latter. On the open orbit, the sheaf is
\begin{align}
    \Qlb[\Pair{2\rho}{\lambda}](\Pair{\rho}{\lambda}).
\end{align}
Since \(\lambda=\sigma_{\Out}(\lambda)\), it further descends (through the
quadratic Galois action) to
\eqref{eqn:truncated_Hecke_stack_for_S_n}. We denote this sheaf by
\(\IC_n^\lambda\).

%
\begin{definition}
Let \(f_\lambda\) be the function on \(\Loop_v\SYMS_n(k_v)\) induced by \(\IC_n^\lambda\).
\end{definition}
Under the canonical spherical transfer
map \(\cH_0\to\cH_0'\), the element \(f_\lambda\) maps to \(f_\lambda'\).
. More generally, we can similarly define for any
\(\sigma_{\Out}\)-orbit \(\Set{\mu}\), an ``intersection complex'' on
\(\Stack*{\Arc_v\bG'\backslash \Loop_v\SYMS_n}\), which is geometrically isomorphic
to the intersection complex on the stack
\begin{align}
    \bigcup_\mu \Stack*{\Arc_v\bG\backslash\Gr_{v,\bG}^{\le\mu}}.
\end{align}
Its induced function \(f_{\Set{\mu}}\) transfers to \(f_{\Set{\mu}}'\).

\subsubsection{}
The same discussions also apply to \(\SYMS_n^\AD\) and \((G')^\AD\). For any
dominant cocharacter \(\lambda\) of \(\SYMS_n^\AD\) (resp.~\((G')^\AD\)),
or more generally a \(\sigma_{\Out}\)-orbit \(\Set{\mu}\) of dominant
cocharacters of \((\SYMS_n^\AD)_{\bar{k}}\cong (G')^\AD_{\bar{k}}\), we have
matching Satake functions
\begin{align}
    f_{\Set{\mu}}\longmapsto f_{\Set{\mu}}',
\end{align}
compatible with the transfer map for \(\cH_0\) and \(\cH_0'\).


\subsection{Jacquet--Rallis fundamental lemma} 
\label{sub:jacquet_rallis_fundamental_lemma}

In this subsection, we formulate the spherical Jacquet--Rallis fundamental lemma
in full detail and interface it with the language of monoids that we have been
using throughout this paper. As always in this section, we
fix a place \(v\in \abs{X}\). Without loss
of generality we may assume that its residue field \(k_v\) is equal to \(k\).

\subsubsection{}
By class field theory, the quadratic \'etale algebra \(F_v'\) over \(F_v\)
induces a quadratic character
\begin{align}
    \eta\colon F_v^\x\longto F_v^\x/\Nm_{F_v'/F_v}(F_v')^\x\longto \Set{\pm 1}.
\end{align}
Since \(F_v'\) is unramified over \(F_v\), \(\eta\) is trivial if \(v'\) is
split over \(v\), and otherwise, \(\eta\) has a
simple description:
\begin{align}
    \eta(x)=(-1)^{\val_v(x)}.
\end{align}

\subsubsection{}
Let \(A\in \SYMS_n^\srs(F_v)\) be a strongly regular semisimple element, and
\(f\in\cH_0\) a spherical function. We fix a Haar measure \(\dd h\) on
\(H(F_v)\) so that the volume of \(H(\cO_v)\) is \(1\). We define the
\(\eta\)-twisted \(H\)-orbital integral associated with \(\dd h\), \(A\) and
\(f\) as
\begin{align}
    \OI_{A,H}^\eta(f,\dd h)=\int_{H(F_v)}f(\Ad_h^{-1}(A))\eta(\det{h})\dd h.
\end{align}
We shall simply use \(\OI_{A,H}^\eta(f)\) from now on and omit \(\dd h\) because the
latter is always fixed. The convergence of this integral will be apparent once
we relate it to affine Jacquet--Rallis fibers.

\subsubsection{}
Similarly, let \(A'\in G^{\prime,\srs}(F_v)\) be strongly regular semisimple,
and \(f'\in \cH_0'\). We fix a Haar measure \(\dd h'\) on \(H'(F_v)\) so that
the volume of \(H'(\cO_v)\) is \(1\). We define the \(H'\)-orbital
integral associated with \(\dd h'\), \(A'\) and \(f'\) as
\begin{align}
    \OI_{A',H'}(f',\dd h')=\int_{H'(F_v)}f'(\Ad_{h'}^{-1}(A'))\dd h'.
\end{align}
As in the \(\SYMS_n\)-case, we will omit \(\dd h'\) and simply use
\(\OI_{A',H'}(f')\) from now on. The convergence of this integral will also be
proved using affine Jacquet--Rallis fibers.

\subsubsection{}
The main theorem of this paper is as follows:
\begin{theorem}
    [Jacquet--Rallis Fundamental Lemma]
    \label[theorem]{thm:main_theorem}
    Let \(a\in\FRC_{n,\JR}^\srs(F_v)\) be a matching pair
    (cf.~\Cref{def:matching_orbits}), and choose an
    arbitrary lift \(A\in \SYMS_n(F_v)\) and \(A'\in G'(F_v)\) of \(a\).
    Then for any pair of matching functions \((f,f')\in \cH_0\x\cH_0'\), we have
    equality in orbital integrals
    \begin{align}
        \Delta(A)\OI_{A,H}^{\eta}(f)=\OI_{A',H'}(f'),
    \end{align}
    where \(\Delta(A)=\eta(\pi_v)^{\val_v(\bFe^\vee\wedge \bFe^\vee
    A\wedge\cdots\wedge \bFe^\vee A^{n-1})}\). Moreover, if \(a\) is not a
    matching pair, then both sides equal to \(0\).
\end{theorem}

\subsubsection{}
We now reformulate \Cref{thm:main_theorem} in geometric terms using affine
Jacquet--Rallis fibers. For simplicity, we restrict to functions \(f_\lambda\)
and \(f_\lambda'\), and the general case is similar by taking linear
combinations.

We first look at \(\SYMS_n\)-side. For an element \(h\in H(F_v)\) to contribute to
the orbital integral, it is necessary that \(h^{-1}Ah\) to be contained in the
closure of the coset \(\Cartan_n^\lambda\). Since we normalize \(\dd h\) so that
\(H(\cO_v)\) has volume \(1\), the orbital integral
\(\OI_{A,H}^\eta(f_\lambda)\) is the same as ``\((f_\lambda,\eta)\)-weighted''
point count of the set
\begin{align}
    \label{eqn:set_related_to_OI_symmetric}
    \Set*{h\in H(F_v)/H(\cO_v)\given \Ad_h^{-1}(A)\in\Cartan_n^{\le\lambda}=\bar{\Cartan_n^\lambda}}.
\end{align}
More precisely, we have
\begin{align}
    \OI_{A,H}^\eta(f_\lambda)=\sum_{h}f_\lambda\bigl(\Ad_h^{-1}(A)\bigr)\eta(\det{h}),
\end{align}
where \(h\) ranges over the set \eqref{eqn:set_related_to_OI_symmetric}.

Let \(A_\AD\) be the image of \(A\) in \(\SYMS_n^\AD(F_v)\), and \(\lambda_\AD\)
that of \(\lambda\). Then it is clear that the \((f_\lambda,\eta)\)-weighted
point count of \eqref{eqn:set_related_to_OI_symmetric} is the same as
\((f_{\lambda_\AD},\eta)\)-weighted point count of the set
\begin{align}
    \label{eqn:adjoint_set_related_to_OI}
    \Set*{h\in H(F_v)/H(\cO_v)\given
    \Ad_h^{-1}(A_\AD)\in\Cartan_{n,\AD}^{\le\lambda_\AD}}.
\end{align}
As we discussed in
\Cref{sub:matching_orbits}, we may find a lift \(\gamma\in\FRM(F_v)\) of
\(A_\AD\). When \eqref{eqn:set_related_to_OI_symmetric} is non-empty, by
\cite[Lemma~2.5.1 and Proposition~3.1.6]{Ch19} (see also
\cite[Lemma~4.1.5 and Proposition~4.1.9]{Wa25}), we may choose \(\gamma\) so
that its abelianization lies in \(\pi_v^{\lambda_\AD}\FRA_\FRM^\x(\cO_v)\). As
in the discussion before \Cref{lem:disc_of_group_vs_monoid}, \(\gamma\),
together with standard vector \(\bFe\) and co-vector \(\bFe^\vee\), induces a
point
\begin{align}
    a=\chi_{\sM}(\gamma,\bFe,\bFe^\vee)\in\FRC_\sM(\cO_v).
\end{align}
Similar to the case of MASFs (see
\cite[\S~3.2]{Ch19} or \cite[\S~4.1]{Wa25}), we shall see that
\(\OI_{A,H}^\eta(f_\lambda)\) is the same as certain weighted point count of the
set
\begin{align}
    \cM_{v}(a)(k_v).
\end{align}
To do this, we need to geometrize both \(f_\lambda\) and \(\eta\)
on \(\cM_v(a)\) via sheaf-function dictionary.

\subsubsection{}
The geometrization of \(f_\lambda\) is straightforward: we have a local
evaluation map
\begin{align}
    \ev_v\colon \cM_v(a)\longto \Stack*{\Arc_v(\bG')^\AD\backslash \Loop_v \SYMS_n^\AD},
\end{align}
and we pull back the sheaf \(\IC_n^\lambda\).

The geometrization of \(\eta\) is less obvious. We let \(\cF_\eta\) be the
rank-\(1\) local system on \(\Spec{k_v}\) associated with \(v'/v\).
By \Cref{def:affine_jacquet_rallis_fiber} and invariant theory in
\Cref{sub:deformed_quotient_stack}, for any \(\breve{\cO}_{\bar{v}}\)-point
\((\cE,x,e,e^\vee)\) of \(\sM\),  we have a pure co-tensor
\begin{align}
    f_n^\vee\colon \wedge^n \cE\longto \bbA^1,
\end{align}
whose valuation is well-defined. Since the valuation is discrete,
\(\val_v(f_n^\vee)\) is locally constant on \(\cM_v(a)\). Let
\begin{align}
    \cM_v^i(a)\subset\cM_v(a)
\end{align}
be the open and closed subscheme of points
with \(\val_v(f_n^\vee)=i\). Then we define \(L_{v,\eta}^i\) on \(\cM_v^i(a)\) to be
the local system \(\cF_\eta^{\otimes i}\). Let \(L_{v,\eta}\) be the disjoint union
of various \(L_{v,\eta}^i\).

We may consider the
\(\IC_n^\lambda\otimes L_{v,\eta}\)-weighted point count on \(\cM_v(a)\), denoted by
\begin{align}
    \Cnt_{\IC_n^\lambda\otimes L_{v,\eta}}\cM_v(a)(k_v)\defeq
    \sum_{x\in\cM_v(a)(k_v)}\Tr_{\IC_n^\lambda\otimes L_{v,\eta}}(x),
\end{align}
where \(\Tr_{\IC_n^\lambda\otimes L_{v,\eta}}\) is the function induced by
\(\IC_n^\lambda\otimes L_{v,\eta}\) by taking \(k_v\)-Frobenius traces.
\begin{lemma}
    \label[lemma]{lem:OI_vs_pt_count_in_affine_JR_fiber_symmetric}
    With \(a=\chi_\sM(\gamma,\bFe,\bFe^\vee)\) induced by \(A\) as above (in
    particular, the set \eqref{eqn:set_related_to_OI_symmetric} is non-empty),
    we have an equality
    \begin{align}
        \label{eqn:OI_vs_pt_count_in_affine_JR_fiber_symmetric}
        \Delta(A)\OI_{A,H}^\eta(f_\lambda)=\Cnt_{\IC_n^\lambda\otimes L_{v,\eta}}\cM_v(a)(k_v).
    \end{align}
\end{lemma}
\begin{proof}
    Since the vector \(\bFe\) and the co-vector \(\bFe^\vee\) are standard, any
    \(\cO_v\)-point in \(\sM\) lying over \(a\) is necessarily contained in
    \(\Stack*{\FRM/H}\). Given a point in \(\Stack*{\FRM/H}(\cO_v)\), in other
    words, an \(H\)-torsor \(\cE_H\) over \(\cO_v\) and an
    \(H\)-equivariant map \(\phi\colon\cE_H\to\FRM\). Trivialize \(\cE_H\) over
    \(\cO_v\), we obtain via \(\phi\) a point \(x\in\FRM(\cO_v)\). By
    assumption, \(x\) lies over \(a\), and so we can find \(h\in H(F_v)\) such
    that \(hxh^{-1}=\gamma\).

    By Cartan decomposition of \(\FRM(\cO_v)\cap\FRM^\x(F_v)\) (see
    \cite[Lemma~2.5.1]{Ch19} or \cite[Lemma~4.1.5]{Wa25}), \(x_\AD\) is
    contained in \(\Cartan_{n,\AD}^{\le\lambda_\AD}\).
    If we change the trivialization of \(\cE_H\), then
    \(x\) is changed by an \(H(\cO_v)\)-conjugation, and the coset
    \(hH(\cO_v)\) is unchanged. Therefore, we have a well-defined map
    \begin{align}
        \cM_v(a)(k_v)\longto \Set*{h\in H(F_v)/H(\cO_v)\given \Ad_h^{-1}(A_\AD)\in\Cartan_{n,\AD}^{\le\lambda_\AD}}.
    \end{align}
    Running the argument backwards, we obtain the map in the opposite direction,
    and one can easily verify that this is an isomorphism.

    It remains to verify that
    \begin{align}
        \Delta(A)\eta(\det{h})=\eta(\pi_v)^{\val_v(f_n^\vee(\cE_H,x,\bFe,\bFe^\vee))},
    \end{align}
    or equivalently,
    \begin{align}
        \val_v(\bFe^\vee\wedge \bFe^\vee A\wedge\cdots\wedge\bFe^\vee
        A^{n-1})+\val_v(\det{h})\equiv\val_v\bigl(f_n^\vee(\cE_H,x,\bFe,\bFe^\vee)\bigr)\bmod
        2.
    \end{align}
    Since we are computing the valuations, we fix a trivialization of \(\cE_H\)
    over \(\cO_v\). Let \(x_1\) be the image of \(x\) in \(\End_n(\cO_v')\). Let
    \(z_i\) be the coordinates of the abelianization of \(x\) (and \(\gamma\)).
    Then, similar to \Cref{prop:local_obstruction_and_valuation_parity}, we have
    \begin{align}
        \label{eqn:eta_twisting_f_n_vs_lattice}
        f_n^\vee(\cE_H,x,\bFe,\bFe^\vee)
        &=\prod_{i=1}^{n-1}z_i^{\frac{(n-i)(n-i-1)}{2}}(\bFe^\vee\wedge\bFe^\vee x_1\wedge\cdots\wedge\bFe^\vee x_1^{n-1})\\
        &=\prod_{i=1}^{n-1}z_i^{\frac{(n-i)(n-i-1)}{2}}(\bFe^\vee\wedge\bFe^\vee
        h^{-1}\gamma_1h\wedge\cdots\wedge\bFe^\vee h^{-1}\gamma_1^{n-1}h)\\
        &=\prod_{i=1}^{n-1}z_i^{\frac{(n-i)(n-i-1)}{2}}(\bFe^\vee\wedge\bFe^\vee
        \gamma_1\wedge\cdots\wedge\bFe^\vee \gamma_1^{n-1})\det{h},
    \end{align}
    and so we reduce to prove that
    \begin{align}
        \val_v(\bFe^\vee\wedge \bFe^\vee A\wedge\cdots\wedge\bFe^\vee A^{n-1})
        \equiv\val_v\left[\prod_{i=1}^{n-1}z_i^{\frac{(n-i)(n-i-1)}{2}}(\bFe^\vee\wedge\bFe^\vee
        \gamma_1\wedge\cdots\wedge\bFe^\vee \gamma_1^{n-1})\right]\bmod 2.
    \end{align}
    Since \(A_\AD=\gamma_\AD\), we may find \(c\in(F_v')^\x\) such that
    \(A=c\gamma_1\). So the left-hand side becomes
    \begin{align}
        \val_v(\bFe^\vee\wedge \bFe^\vee A\wedge\cdots\wedge\bFe^\vee A^{n-1})
        &= \val_v(\bFe^\vee\wedge \bFe^\vee (c\gamma_1)\wedge\cdots\wedge\bFe^\vee
        (c\gamma_1)^{n-1})\\
        &= \val_v(\bFe^\vee\wedge \bFe^\vee \gamma_1\wedge\cdots\wedge\bFe^\vee
        \gamma_1^{n-1})+\frac{n(n-1)}{2}\val_v(c).
    \end{align}
    In the proof of \Cref{prop:local_obstruction_and_valuation_parity}, we computed
    \begin{align}
        \det{\gamma_1}=\prod_{i=1}^{n-1}z_i^{n-i},
    \end{align}
    and since \(A\bar{A}=1\), we have \(\val_v(\det{\gamma_1})=-n\val_v(c)\).
    Therefore, it suffices to prove that
    \begin{align}
        \frac{n-1}{2}\sum_{i=1}^{n-1}(n-i)\val_v(z_i)
        -\sum_{i=1}^{n-1}\frac{(n-i)(n-i-1)}{2}\val_v(z_i)
        =\sum_{i=1}^{n-1}\frac{(n-i)i}{2}\val_v(z_i)
        \equiv 0\bmod 2.
    \end{align}
    Since \(\lambda\) is a cocharacter of \(\SYMS_n\),
    we necessarily have \(\Pair{\rho}{\lambda_\AD}\in\bbZ\). It is easy to
    compute that
    \begin{align}
        \sum_{i=1}^{n-1}\frac{(n-i)i}{2}\val_v(z_i)=\Pair{2\rho}{\lambda_\AD}\in
        2\bbZ.
    \end{align}
    This finishes the proof.
\end{proof}

\begin{remark}
    \label[remark]{rmk:f_n_of_base_pt_vs_actual_pt}
    Given a point in \(\cM_v(a)(\bar{k})\), we have a lattice description
    \eqref{eqn:JR_fiber_lattice_description}. The latter depends on the choice
    of \(m=(\gamma,\bFe,\bFe^\vee)\). By the same computation in
    \eqref{eqn:eta_twisting_f_n_vs_lattice}, we see that modulo \(2\), the
    valuation of \(f_n^\vee\) at a point \(\Lambda\) is the same as the
    valuation of \(f_n^\vee(m)/\wedge^n\Lambda\). This part of the computation
    clearly does not depend on the fact that \(a\) is induced by \(A\), hence
    valid for any \(a\).
\end{remark}

\subsubsection{}
On the unitary side, we start with \(A'\in G'(F_v)\) and
\(f_\lambda'\in\cH_0'\). We have the set similar to
\eqref{eqn:set_related_to_OI_symmetric}:
\begin{align}
    \label{eqn:set_related_to_OI_unitary}
    \Set*{h'\in H'(F_v)/H'(\cO_v)\given \Ad_{h'}^{-1}(A')\in\Cartan_n^{\prime\le\lambda}}.
\end{align}
We likewise have \(\gamma'\in \FRM'(F_v)\) lifting \(A'_\AD\), and if
\eqref{eqn:set_related_to_OI_unitary} is non-empty, \(\gamma'\) may be so chosen
that its abelianization has boundary divisor \(\lambda\), and together with the
standard vector \(\bFe\) and standard Hermitian form, it induces
\(a'\in\FRC_\sM(\cO_v)\). If \(A'\) matches \(A\in
\SYMS_n(F_v)\), then \(a'\) coincides with \(a\) in
\Cref{lem:OI_vs_pt_count_in_affine_JR_fiber_symmetric} and the preceding
discussions.

We may pull back the intersection complex \(\IC_n^{\prime\lambda}\) to
\(\cM_v'(a)\) through the evaluation map
\begin{align}
    \ev_v'\colon \cM_v'(a)\longto\Stack*{\Arc_v(G')^\AD\backslash
    \Gr_{v,(G')^\AD}}.
\end{align}
Unsurprisingly, we have the following result, whose proof is similar to that
of \Cref{lem:OI_vs_pt_count_in_affine_JR_fiber_symmetric}:
\begin{lemma}
    \label[lemma]{lem:OI_vs_pt_count_in_affine_JR_fiber_unitary}
    We have an equality
    \begin{align}
        \label{eqn:OI_vs_pt_count_in_affine_JR_fiber_unitary}
        \OI_{A',H'}(f_\lambda')=\Cnt_{\IC_n^{\prime\lambda}}\cM_v'(a')(k_v).
    \end{align}
\end{lemma}

\begin{remark}
    If the set \eqref{eqn:set_related_to_OI_symmetric} is empty, then the left-hand side
    of \eqref{eqn:OI_vs_pt_count_in_affine_JR_fiber_symmetric} is \(0\), and the
    right-hand side is \textit{a priori} not defined since there is no \(a\). We
    use the convention that the right-hand side is counting the empty set for
    consistency. The case of \eqref{eqn:set_related_to_OI_unitary} and
    \eqref{eqn:OI_vs_pt_count_in_affine_JR_fiber_unitary} is similar.
\end{remark}

\subsubsection{}
Conversely, given \emph{any} generically strongly regular semisimple
\(a\in\FRC_\sM(\cO_v)\cap\FRC_\sM^\srs(F_v)\), observe that the point counts on
the right-hand sides of \eqref{eqn:OI_vs_pt_count_in_affine_JR_fiber_symmetric}
and \eqref{eqn:OI_vs_pt_count_in_affine_JR_fiber_unitary} are still
well-defined.
Therefore, \Cref{thm:main_theorem} is a corollary of the following:
\begin{theorem}
    \label[theorem]{thm:main_theorem_geometric}
    Let \(a\in\FRC_\sM(\cO_v)\cap\FRC_\sM^\srs(F_v)\). Then we have equality
    \begin{align}
        \label{eqn:pt_count_version_of_FL}
        \Cnt_{\IC_n^\lambda\otimes
        L_{v,\eta}}\cM_v(a)(k_v)=\Cnt_{\IC_n^{\prime\lambda}}\cM_v'(a)(k_v),
    \end{align}
    and both sides are \(0\) if \(a\) is not a matching pair (i.e.,
    \(\Ob_v(a)\) is non-trivial).
\end{theorem}


\subsection{Some simple cases} 
\label{sub:some_simple_cases}

We shall prove \Cref{thm:main_theorem_geometric} using global methods. In order
to deduce the local result from the global result, we need to compute a few
cases of \Cref{thm:main_theorem_geometric} that are particularly simple.

\subsubsection{}
The simplest case is when \(v'\) is split over \(v\). In this case, the
character \(\eta\) is trivial hence so is \(L_{v,\eta}\). Moreover, it is
straightforward to see that
\(\FRM\) (resp.~\(G\), resp.~\(H\)) is isomorphic to \(\FRM'\) (resp.~\(G'\),
resp.~\(H'\)), and these isomorphisms are compatible. It implies that we have a
\(k_v\)-isomorphism
\begin{align}
    \cM_v(a)\cong\cM_v'(a),
\end{align}
as well as
\begin{align}
    \IC_n^\lambda\otimes L_{v,\eta}\cong \IC_n^\lambda\cong \IC_n^{\prime\lambda}.
\end{align}
Therefore, it is trivial to show the following:
\begin{lemma}
    \Cref{thm:main_theorem_geometric} holds when \(v'\) is split over \(v\).
\end{lemma}

\subsubsection{}
Another particularly simple case is when \(\Disc(a)\) has an odd
valuation. In this case, we have the following:
\begin{lemma}\label[lemma]{lem:vanishing_odd_val}
    Suppose \(v'\) is non-split over \(v\) and \(\Disc(a)\) has odd valuation.
    Then \Cref{thm:main_theorem_geometric} holds and both sides of
    \eqref{eqn:pt_count_version_of_FL} equal \(0\).
\end{lemma}
\begin{proof}
    By \Cref{prop:local_obstruction_and_valuation_parity}, \(\Ob_v(a)\) is
    non-trivial, and in this case \(\cM_v'(a)(k_v)\) is empty. Therefore, the
    right-hand side of \eqref{eqn:pt_count_version_of_FL} is \(0\).

    On the other hand, we have an involution
    \eqref{eqn:involution_on_symmetric_side} on \(\cM_v(a)(k_v)\), such that
    \begin{align}
        \val_v\bigl(\wedge^n \Lambda^*/f_n(\tilde{m}^*)\bigr)=\val_v(\Disc(a))-
        \val_v\bigl(\wedge^n \Lambda/f_n(\tilde{m})\bigr)
    \end{align}
    by \Cref{lem:valuation_under_involution}. Let \(g^*\in\Gr_{v,G}(k_v)\) be
    the image of \(g\) under this involution, then we have
    \begin{align}
        \val_v(\det{g^*})-\val_v(f_n(\tilde{m}^*))=\val_v(\Disc(a))-\val_v(\det{g})+\val_v(f_n(\tilde{m})).
    \end{align}
    By definition of \(\tilde{m}^*\), \(f_n(\tilde{m}^*)\) and
    \(f_n(\tilde{m})\) are \(\cO_v\)-duals, and so their valuations are
    inverses of each other. This means that
    \begin{align}
        \eta(\det{g^*})=\eta(\Disc(a))\eta(\det{g})=-\eta(\det{g}).
    \end{align}
    The function \(f_\lambda\) remains unchanged under this involution (see the
    discussion following \Cref{lem:valuation_under_involution}). As a result, we
    have
    \begin{align}
        \Cnt_{\IC_n^\lambda\otimes
        L_{v,\eta}}\cM_v(a)(k_v)=-\Cnt_{\IC_n^\lambda\otimes L_{v,\eta}}\cM_v(a)(k_v)=0.
    \end{align}
    This finishes the proof.
\end{proof}

\begin{remark}
    In the lemma above, we in fact proved a \notion{functional equation} for any
    \(v'/v\) and any generically strongly regular semisimple \(a\):
    \begin{align}
        \Cnt_{\cF\otimes
        L_{v,\eta}}\cM_v(a)(k_v)=\eta(\Disc(a))\Cnt_{\cF\otimes
    L_{v,\eta}}\cM_v(a)(k_v)
    \end{align}
    for any \(\Arc_v{\bG'}\)-equivariant constructible complex \(\cF\) on
    \(\Loop_v{\SYMS_n}\).
\end{remark}

\subsubsection{}
When \(v'\) is non-split over \(v\) and \(\val_v(\Disc(a))\) is even, the
situation is more subtle and we identify two different simple cases. It relies
on deeper knowledge of the MASFs
associated with \(a\) (more precisely, its image \(a^\Hit\in\FRC_\FRM(\cO_v)\)).

As before, let \(\lambda\) be the boundary divisor of \(a\), which is just
the boundary divisor of \(a^\Hit\). Let \(m\in\sM(F_v)\) be the unique point
over \(a\), which extends to a point \(\tilde{m}\in\sM(\cO_v')\). This, in
particular, induces a point \(\gamma\in\FRM(F_v)\) which is
\(G(F_v')\)-conjugate to a point \(\tilde{\gamma}\in\FRM(\cO_v')\).
By the construction of \(\tilde{m}\) in \Cref{sub:deformed_quotient_stack},
\(\tilde{\gamma}\) lies in the \(G\)-regular locus of \(\FRM\).

We view \(\gamma\) as a point in \((\bM')^\x(F_v)\). Following
\cite[\S~4]{Wa25}, there is a canonical \(W\)-orbit of rational cocharacters
associated with \(\gamma_\AD\), represented by a dominant rational cocharacter
\(\nu\in\CoCharG((\bT')^\AD)_\bbQ\). 

\begin{definition}
Both \(\nu\) and its \(W\)-orbit are
called the \notion{Newton point} of \(\gamma_\AD\). Similarly, the pair
\((\lambda,\nu)\) will be called the \notion{Newton point} of \(a^\Hit\).   
\end{definition}

Since
any \(W\)-orbit of rational
cocharacters contains a unique dominant element, \(\nu\) is necessarily fixed by
outer automorphism \(\sigma_{\Out}\).

\subsubsection{}
By \cite[Proposition~3.1.6]{Ch19} (see also
\cite[Proposition~4.1.9]{Wa25}), for \(\cM_a\) to be non-empty, it is necessary
that \(\cM_a^\Hit\) to be non-empty, which implies
\(\nu\le_\bbQ\lambda=-w_0(\lambda)\). In
other words, \(\lambda-\nu\) must be a non-negative \(\bbQ\)-combination of
positive roots. This condition also holds if \(\cM_a'\) is non-empty, because
over \(\bar{k}\) the scheme \(\cM_a\) and \(\cM_a'\) are isomorphic.

If \(\nu\) is not integral (i.e., not contained in the lattice
\(\CoCharG((\bT')^\AD)\)), the geometry of MASF \(\cM_a^\Hit\) is quite
complicated; fortunately in this subsection, we need only consider the case of
integral \(\nu\). More precisely, we will only consider the following two cases:
\begin{description}
    \item[Case A] \(\lambda=\nu=0\), and the discriminant valuation \(d_{v+}(a^\Hit)\) is
        \(1\) (and so the local \(\delta\)-invariant \(\delta_v(a^\Hit)=0\));
    \item[Case B] \(\lambda=\nu\) is either \(0\) or one of the fundamental coweights,
        and \(a^\Hit\) is unramified and
        \(\nu\)-regular semisimple in the sense of
        \cite[Definition~3.4.6]{Wa25}. Roughly speaking, it means that
        \(a^\Hit\) is ``as general as possible'' among all those with Newton
        point \((\lambda,\nu)\).
\end{description}
These cases are part of the table in \cite[(12.2.1)]{Wa25}. In both cases,
\(\cM_v^\Hit(a)\) is a torsor of the local Picard \(\cP_v^\Hit(a)\), and
similarly \(\cM_v^{\prime\Hit}(a)\) is a \(\cP_v^{\prime\Hit}(a)\)-torsor. See
\cite[Corollaries~3.5.3, 3.8.2, 3.8.3]{Ch19} (also \cite[Theorem~4.3.5,
Corollary~4.3.6]{Wa25}), as well as
\Cref{sub:the_cameral_cover_and_the_picard_action}. Moreover, the local Picard
\(\cP_v^\Hit(a)\) and \(\cP_v^{\prime\Hit}(a)\) are geometrically homeomorphic to a
\(\bbZ\)-lattice. Since we only care about point count problems, we will ignore
the non-reducedness of these schemes.

In both cases, the restriction of the Satake function \(f_\lambda\) and
\(f_\lambda'\) simplify to constant functions, because the image of the local
evaluation map in these cases always lands in the big affine Schubert cell.
Therefore, we may use functions \(\One_{\Cartan_n^\lambda}\) and
\(\One_{\Cartan_n^{\prime\lambda}}\) instead.

\subsubsection{}
In what follows, we will compute both cases A and B in two steps: first, we
compute the geometric structure of \(\cM_v(a)\) (resp.~\(\cM_v'(a)\),
resp.~\(\cM_v^\Hit(a)\), resp.~\(\cM_v^{\prime\Hit}(a)\)), namely the
\(\bar{k}\)-scheme \(\cM_{\bar{v}}(a)\) (resp.~\(\cM_{\bar{v}}'(a)\),
resp.~\(\cM_{\bar{v}}^\Hit(a)\), resp.~\(\cM_{\bar{v}}^{\prime\Hit}(a)\))
This step is the same for both the symmetric side and the unitary side.
Second, we compute the \(k_v\)-Frobenius structures and compare them to deduce
\Cref{thm:main_theorem_geometric} in these special cases.

\subsubsection{}
We first look at Case A. By the results in
\cite[\S~4.5]{Wa25}, the isomorphism class of \(\cM_{\bar{v}}^\Hit(a)\) (as
\(\bar{k}\)-variety) is determined by the cameral cover
\begin{align}
    \pi_a\colon \tilde{X}_{\bar{v},a^\Hit}\longto \breve{X}_{\bar{v}}
\end{align}
of \(a^\Hit\) around \(\bar{v}\).
In this case, the ramification of the cameral
cover is given by a single pair of roots \(\pm\Rt\). The root \(\Rt\) depends on
the choice of a base point on \(\tilde{X}_{\bar{v},a^\Hit}\), which may be
represented by an element \(t\in\FRT_\FRM(\bar{k}\powser{\pi_v^{1/2}})\) lying over
\(a^\Hit\). We fix such a \(t\), so that \(\Rt\) is the highest root.

\subsubsection{}
Let \(W_{\pm\Rt}\cong \bbZ/2\bbZ\) be the Weyl
subgroup generated by \(\pm\Rt\), then
\begin{align}
    \pi_{a}^{\pm\Rt}\colon\tilde{X}_{\bar{v},a^\Hit}/W_{\pm\Rt}\longto \breve{X}_{\bar{v}}
\end{align}
is an \'etale covering, hence a trivial covering since the residue field of
\(\bar{v}\) is algebraically closed.
Consequently, we can reduce to the Levi subgroup generated by \(\pm\Rt\),
and easily see (by looking at the cameral covers) that \(\cM_{\bar{v}}^\Hit(a)\)
is isomorphic to the geometric MASF associated with
\begin{align}
    \begin{pmatrix}
        0 & & & & \pi_v\\
          & t_2 & & &\\
          & & \ddots & &\\
          & & & t_{n-1} &\\
        1 & & & & 0
    \end{pmatrix}\in \Mat_n(\cO_v),
\end{align}
where \(t_2,\ldots,t_{n-1}\) are distinct elements in \(\cO_v^\x\).
By direct computation, one can see that
\begin{align}
    \label{eqn:geometric_MASF_case_A}
 \cM_{\bar{v}}^\Hit(a)(\bar{k})= \Set*{\begin{pmatrix}
            0 & & \pi_v\\
              & I_{n-2} &\\
            1 & & 0
        \end{pmatrix}^{d_1}\Diag(1,\pi_v^{d_2},\ldots,\pi_v^{d_{n-1}},1)\Lambda_0\given
    d_i\in \bbZ}\cong \bbZ\x\bbZ^{n-2}.
\end{align}
By taking the diagonal conjugate over \(\breve{F}_{\bar{v}}[\pi_v^{1/2}]\) and
then taking the inertia invariants, we see that
\eqref{eqn:geometric_MASF_case_A} is isomorphic to the set
\begin{align}
    \label{eqn:geometric_MASF_case_A_ramified}
    \Set*{\Diag(\pi_v^{d_1/2},\pi_v^{d_2},\ldots,\pi_v^{d_{n-1}},\pi_v^{d_1/2})\Lambda_0\given
    d_i\in \bbZ}.
\end{align}
Note that MASF for ramified elements usually cannot be computed using inertia
invariants, but in this especially simple case there is no problem.

\subsubsection{}
Let \(\bM_{\pm\Rt}^\x\) be the Levi subgroup of \(\bM^\x\) generated by
\(\pm\Rt\), and it is stable under \(\iota\sigma\). Let \(\FRM_{\pm\Rt}^\x\)
be the corresponding twist in \(\FRM^\x\), and let \(\FRM_{\pm\Rt}\subset\FRM\)
be its closure. There exists
\(\gamma_{\pm\Rt}\in\FRM_{\pm\Rt}^\x(\breve{F}_{\bar{v}})\) lying over
\(a^\Hit\), and it is \(G(\breve{F}_{\bar{v}})\)-conjugate to \(\gamma\). Let
\(e_{\pm\Rt}\) (resp.~\(e_{\pm\Rt}^\vee\)) be the corresponding conjugate of
\(e\) (resp.~\(e^\vee\)) associated with \(a\) (or rather, \(m\)).

Diagonalize \(\gamma_{\pm\Rt}\) over
\(\breve{F}_{\bar{v}}[\pi_v^{1/2}]\), then the geometric MASF may be regarded as
being \emph{equal} to the set
\eqref{eqn:geometric_MASF_case_A}. As a result,
\begin{align}
    \label{eqn:geometric_AJRF_case_A}
    \cM_{\bar{v}}(a)(\bar{k})\cong\Set*{\Lambda=\Diag(\pi_v^{d_1/2},\pi_v^{d_2},\ldots,\pi_v^{d_{n-1}},\pi_v^{d_1/2})\Lambda_0\given
    e_{\pm\Rt}\in\Lambda, e_{\pm\Rt}^\vee\in\Lambda^\vee}.
\end{align}
Clearly, this set depends only on the \(v\)-valuations of the coordinates of
\(e_{\pm\Rt}\) and \(e_{\pm\Rt}^\vee\). Let \(e_1,\ldots,e_n\) and
\(e_1^\vee,\ldots,e_n^\vee\) be those valuations of \(e_{\pm\Rt}\) and
\(e_{\pm\Rt}^\vee\) respectively. Since \(a\) is generically strongly regular
semisimple, \(e_2,\ldots,e_{n-1},e_2^\vee,\ldots,e_{n-1}^\vee\) are all
finite, and both \(\min\Set{e_1,e_n}\) and \(\max\Set{-e_1^\vee,-e_n^\vee}\) lie
in \(\bbZ/2\).
As a result, \(\cM_{\bar{v}}(a)(\bar{k})\) is isomorphic to the set
\begin{align}
    \label{eqn:geometric_AJRF_case_A_simplified}
    \Set*{(d_1,\ldots,d_{n-1})\given \begin{array}{l}
         e_i\ge d_i\ge -e_i^\vee,\forall\, 2\le i\le n-1,\\
         e_1\ge d_1 /2\ge -e_1^\vee,\\
         e_n\ge d_1 /2\ge -e_n^\vee
    \end{array} }.
\end{align}

\subsubsection{}
It remains to compute the respective Frobenius structures corresponding to
\(\cM_v(a)\) and \(\cM_v'(a)\) on \eqref{eqn:geometric_AJRF_case_A}.
The computation is very similar to the discussion leading to
\eqref{eqn:JR_fiber_lattice_description_unitary} and
\eqref{eqn:JR_fiber_lattice_description_symmetric}, except we replace \(m\) and
\(m'\) by \((\gamma_{\pm\Rt},e_{\pm\Rt},e_{\pm\Rt}^\vee)\).

On the symmetric side, since \(a\) as an \(F_v\)-point is strongly regular
semisimple, the tuple \((\gamma_{\pm\Rt},e_{\pm\Rt},e_{\pm\Rt}^\vee)\) is
defined over \(F_v\); namely we can find \(s\in G(F_{v}'[\pi_v^{1/2}])\) such that
\(s\bar{s}=1\) and
\begin{align}
    \bigl(\Ad_s^{-1}(\gamma_{\pm\Rt}),s^{-1}e_{\pm\Rt},e_{\pm\Rt}^\vee s\bigr)
    =\bigl(\iota\sigma(\gamma_{\pm\Rt}),\bar{e_{\pm\Rt}},\bar{e_{\pm\Rt}^\vee}\bigr).
\end{align}
In particular, \(s\) must also lie in the normalizer of the Levi subgroup
generated by \(\pm\Rt\). Therefore, \(s\) has the form
\begin{align}
    \begin{pmatrix}
        x_1 & & & & x_2\\
          & t_2 & & &\\
          & & \ddots & &\\
          & & & t_{n-1} &\\
        x_3 & & & & x_4
    \end{pmatrix}\dot{w},
\end{align}
where \(\dot{w}\) is a rank-\((n-2)\) permutation matrix placed in the middle
block, \(t_2,\ldots,t_{n-1}\in (F_v')^\x\), and
\begin{align}
    s_0\defeq\begin{pmatrix}
        x_1 & x_2\\
        x_3 & x_4
    \end{pmatrix}
\end{align}
is either diagonal or anti-diagonal.
Let \(c_2,\ldots,c_{n-1}\) be the respective valuation of
\(t_2,\ldots,t_{n-1}\). We treat \(\dot{w}\) as a permutation of the set
\(\Set{2,\ldots,n-1}\).
Since \(s\bar{s}=1\), we necessarily have \(\dot{w}^2=1\) and
\(c_i+c_{\dot{w}(i)}=0\).

The Frobenius-fixed point in \eqref{eqn:geometric_AJRF_case_A} are
those lattices \(\Lambda\) such that
\(s\Lambda=\bar{\Lambda}\). As \(\Lambda\) is always diagonalized,
we have \(\bar{\Lambda}=\Lambda\). In
this way, the effect of \(s_0\) may be replaced by
a number in \(F_v'[\pi_v^{1/2}]^\x\) whose square has norm \(1\). Using the
description \eqref{eqn:geometric_AJRF_case_A_simplified}, it translates to
\begin{align}
    \label{eqn:arithmetic_AJRF_case_A_symmetric}
    \cM_v(a)(k_v)\cong\Set*{(d_1,\ldots,d_{n-1})\given \begin{array}{l}
         e_i\ge d_{\dot{w}(i)}+c_i=d_i\ge -e_i^\vee, \forall\,2\le i\le n-1,\\
         e_1\ge d_1/2\ge -e_1^\vee,\\
         e_n\ge d_1/2\ge -e_n^\vee
    \end{array}}.
\end{align}
In particular, it turns out that \(s_0\) has no effect at all.

\subsubsection{}
On the unitary side, we similarly have a Hermitian matrix
\begin{align}
    h=\begin{pmatrix}
        x_1' & & & & x_2'\\
          & t_2' & & &\\
          & & \ddots & &\\
          & & & t_{n-1}' &\\
        x_3' & & & & x_4'
    \end{pmatrix}\dot{w}',
\end{align}
with \(c_2',\ldots,c_{n-1}'\) defined similarly.
Similar to \(s_0\), we have the Hermitian matrix
\begin{align}
    h_0\defeq\begin{pmatrix}
        x_1' & x_2'\\
        x_3' & x_4'
    \end{pmatrix},
\end{align}
and without loss of generality, its net effect on the Frobenius structure
may be replaced by an element in
\(F_v'[\pi_v^{1/2}]^\x\) whose square lies in \(F_v[\pi_v^{1/2}]\). Let \(c_0\)
be the valuation of such
number. As \(h\) is Hermitian and since \(\dot{w}'\tp{\dot{w}'}=1\), we still
have \((\dot{w}')^2=1\), which implies \(c_i'=c_{\dot{w}'(i)}'\).
Then \(\cM_v'(a)(k_v)\) is isomorphic to the set
\begin{align}
    \label{eqn:arithmetic_AJRF_case_A_unitary}
    \Set*{(d_1,\ldots,d_{n-1})\given \begin{array}{l}
         e_i\ge -d_{\dot{w}'(i)}+c_i'=d_i\ge -e_i^\vee, \forall\,2\le i\le n-1,\\
         e_1\ge d_1/2\ge -e_1^\vee, e_n\ge d_1/2\ge -e_n^\vee,\\
         -2d_1=c_0
    \end{array}}.
\end{align}

\subsubsection{}
Looking at the conditions involving \(d_1\) on the unitary side. When
\(2(e_1^\vee-e_1)=2(e_n^\vee-e_n)=-2c_0\) is odd, the set
\eqref{eqn:arithmetic_AJRF_case_A_unitary} is empty;
otherwise, there is a unique
\begin{align}
    d_1=e_1^\vee-e_1=e_n^\vee-e_n^\vee=-c_0
\end{align}
appearing in the set \eqref{eqn:arithmetic_AJRF_case_A_unitary}.

\subsubsection{}
Then we examine the conditions involving \(d_2,\ldots,d_{n-1}\) on both sides.
This requires us to relate the tuple \((\dot{w},c_i)\) to
\((\dot{w}',c_i')\). We claim that in fact \(\dot{w}=\dot{w}'\) and
\(c_i=c_i'\). Indeed, the image of \(\dot{w}\) in \(W\) is uniquely determined
by \(\gamma_{\pm\Rt}\) alone by regular-semisimplicity, and similarly for
\(\dot{w}'\).

Moreover, they are both determined by the image of \(\gamma_{\pm\Rt}\)
modulo the \(\SL_2\)-subgroup generated by \(\pm\Rt\), and such image is
diagonal. Note that modulo the same \(\SL_2\), \(\iota\sigma(\gamma_{\pm\Rt})\)
and \(\tau\sigma(\gamma_{\pm\Rt})\) coincide, which implies that
\(\dot{w}=\dot{w}'\). As a consequence, the valuations \(c_i\) and \(c_i'\) must
also be equal by comparing the relevant diagonal matrix entries.

In particular, we in fact have \(c_i=c_i'=0\) since \(c_i+c_{\dot{w}(i)}=0\) and
\(c_i'=c_{\dot{w}'(i)}'\). Since \(\Ob_v(a)=0\), we necessarily have
\begin{align}
    \val_v(\det{h})=\val_v(\det{h_0})=2c_0\equiv 0\bmod 2.
\end{align}
As a result, if we rearrange indices from \(2\) to
\(n-1\) and suppose that \(2,\ldots,l_1\) are representatives of the
\(\dot{w}\)-orbits of order \(2\), and \(l_1+1,\ldots,l_2\) are indices fixed
by \(\dot{w}\),
then \eqref{eqn:arithmetic_AJRF_case_A_symmetric} is isomorphic to the set
\begin{align}
    \Set*{(d_1,\ldots,d_{l})\in\bbZ^{l}\given \begin{array}{l}
         e_i\ge d_i\ge -e_i^\vee, \forall\,2\le i\le l_2,\\
         e_1\ge d_1/2\ge -e_1^\vee,\\
         e_n\ge d_1/2\ge -e_n^\vee
    \end{array}},
\end{align}
By \Cref{rmk:f_n_of_base_pt_vs_actual_pt}, we have
\begin{align}
    \Cnt_{L_{v,\eta}}\cM_v(a)(k_v)
    &=\sum_{d_1=2\max\Set{-e_1^\vee,-e_n^\vee}}^{2\min\Set{e_1,e_n}}
    \sum_{\substack{e_i\ge d_i\ge -e_i^\vee\\2\le i\le l_2}}
    (-1)^{(d_1-2\max\Set{-e_1^\vee,-e_n^\vee})+2\sum_{j=2}^{l_1}(d_j+e_j^\vee)+\sum_{j=l_1+1}^{l_2}(d_j+e_j^\vee)}\\
    &= (e_2+e_2^\vee+1)\cdots (e_{l_1}+e_{l_1}^\vee+1).
\end{align}
\subsubsection{}
Similarly, \eqref{eqn:arithmetic_AJRF_case_A_unitary} is isomorphic to the set
\begin{align}
    \Set*{(d_1,\ldots,d_{l})\in\bbZ^{l}\given \begin{array}{l}
         e_i\ge d_i\ge -e_i^\vee, \forall\,2\le i\le l_1,\\
         d_i=0, \forall\,l_1+1\le i\le l_2,\\
         d_1=-c_0
    \end{array}},
\end{align}
and so
\begin{align}
 \Cnt_{\IC_n^{\prime\lambda}}  \cM_v' (a)(k_v)
    = (e_2+e_2^\vee+1)\cdots (e_{l_1}+e_{l_1}^\vee+1).
\end{align}
This finishes the proof of \Cref{thm:main_theorem_geometric} in Case A.

\subsubsection{}
Then we look at Case B. In this case, we can choose \(\gamma\) to be a diagonal
element in \(\FRT_\FRM(\breve{\cO}_{\bar{v}})\). The geometric MASF of
\(\gamma\) is simply the lattice of the diagonal torus of \(\GL_n\). In other
words, we have
\begin{align}
\label{eqn:geometric_MASF_case_B}
    \cM_{\bar{v}}^\Hit(a)(\bar{k})\cong\bbZ^n,
\end{align}
and similarly to the part involving indices \(2,\ldots,n-1\) in Case A, we obtain
\begin{align}
\label{eqn:geometric_AJRF_case_B}
    \cM_{\bar{v}}(a)(\bar{k})
    \cong \Set*{(d_1,\ldots,d_n)\given e_i\ge d_i\ge -e_i^\vee,\forall\, 1\le i\le n }.
\end{align}

\subsubsection{}
The computation of the Frobenius structures is also similar to the
\(\Set{2,\ldots,n-1}\)-part of Case A. In particular, there is a rank-\(n\)
permutation matrix \(\dot{w}=\dot{w}'\) and diagonal matrices
\(\Diag(t_1,\ldots,t_n)\) and \(\Diag(t_1',\ldots,t_n')\) respectively on both
sides. We still have \(\dot{w}^2=1\), and the valuations of \(t_i\) and
\(t_i'\) are all \(0\). Rearranging the indices, we let \(1,\ldots,l_1\) be the
indices that are not fixed by \(\dot{w}\), then we similarly have
\begin{align}
    \Cnt_{L_{v,\eta}}\cM_v(a)(k_v)=\Cnt_{\IC_n^{\prime\lambda}} \cM_v'(a)(k_v)=(e_1+e_1^\vee+1)\cdots(e_{l_1}+e_{l_1}^\vee+1).
\end{align}
This proves \Cref{thm:main_theorem_geometric} in Case B.



\section{Global Formulations} 
\label{sec:global_formulations}

Let \(X\) be a smooth, projective, and geometrically connected curve over \(k\)
of genus \(g_X\). Let \(\OGT\colon X'\to X\) be an \'etale double cover  such
that as a curve over \(k\), \(X'\) is also geometrically connected. Let
\(\sigma\in\Gal(X'/X)\) be the unique non-trivial element. We let \(\breve{X}=X_{\bar{k}}\) be the
base change of \(X\) to \(\bar{k}\).

\subsection{The Jacquet--Rallis mH-fibrations} 
\label{sub:the_jacquet_rallis_mh_fibrations}

The stack \(\sM\) admits an action of \(Z_\FRM\x\Gm\x\Gm\), where \((z,t,t^\vee )\)
sends the tuple \((\cE,x,e,e^\vee )\) to \((\cE,zx,te,t^\vee e^\vee )\). Note that
\(x\in\FRM\x^G\cE\) and \(Z_\FRM\) acts on \(\FRM\x^G\cE\) by translation
(because the \(Z_\FRM\)-action on \(\FRM\) commutes with that of \(G\)), and
so \(zx\) is well-defined. For our purposes, we will only consider the action of
the subgroup \(Z_\FRM\x\Gm\subset Z_\FRM\x\Gm\x\Gm\) where the \(\Gm\)-factor of
the former embeds diagonally into \(\Gm\x\Gm\).

\subsubsection{}
Consider the \(2\)-stack \(\Stack*{\sM/Z_\FRM\x\Gm}\) over \(X\). Its
\(S\)-points (\(S\) being an \(X\)-scheme) are tuples
\begin{align}
    (\cL,\cD,\cE,x,e,e^\vee ),
\end{align}
where \(\cL\) is a \(Z_\FRM\)-torsor, \(\cD\) is a line bundle (associated with
a \(\Gm\)-torsor), \(\cE\) is a \(G\)-torsor, all over \(S\), \(x\) is a section
\begin{align}
    S\longto \FRM\x^{Z_\FRM\x G}(\cL\x_{S}\cE),
\end{align}
and \(e,e^\vee \) are morphisms:
\begin{align}
    \cD^{-1}\stackrel{e}{\longto}V_\cE\stackrel{e^\vee }{\longto}\cD.
\end{align}
It is easily verified that \(\Stack*{\sM/Z_\FRM\x\Gm}\) is in fact an algebraic
\(1\)-stack.

\subsubsection{}
Similarly, the algebraic stack \(\Stack*{\sM'/Z_\FRM\x\Gm}\) classifies tuples
\begin{align}
    (\cL,\cD,\cE',x',e'),
\end{align}
where \(\cL\) and \(\cD\) are as above, \(\cE'\) is a \(G'\)-torsor
corresponding to a rank-\(n\) vector bundle \(V_{\cE'}'\) on \(S'=S\x_X X'\)
with a Hermitian form, namely an isomorphism
\begin{align}
    h\colon V_{\cE'}'\stackrel{\sim}{\longto}\bigl(\sigma^*V_{\cE'}'\bigr)^\vee 
\end{align}
such that \((\sigma^*h)^\vee \simeq h\), \(x'\) is a section
\begin{align}
    S\longto \FRM'\x^{Z_\FRM\x G'}(\cL\x_{S}\cE'),
\end{align}
and \(e'\) is a linear map
\begin{align}
    e'\colon\cD_{S'}^{-1}\longto V_{\cE'}',
\end{align}
which, through \(h\), also induces a unique linear map
\begin{align}
    e^{\prime\vee}\colon V_{\cE'}'\longto \cD_{S'}.
\end{align}

\subsubsection{}
The action of \(Z_\FRM\x\Gm\) descends to an action on \(\FRC_\sM\),
which is also compatible with the \(Z_\FRM\)-actions on \(\FRM\), \(\FRC_\FRM\)
and \(\FRA_\FRM\). Thus, we have commutative diagram of algebraic stacks
\begin{equation}
    \label{eqn:absolute_stacks}
    \begin{tikzcd}
        \Stack*{\sM/Z_\FRM\x\Gm} \ar[r]\ar[d] & \BG{\Gm}\x\Stack*{\FRM/G\x Z_\FRM}
        \ar[d] & \\
        \Stack*{\FRC_\sM/Z_\FRM\x\Gm}\ar[r] & \BG{\Gm}\x\Stack*{\FRC_\FRM/Z_\FRM} \ar[r] &
        \BG{\Gm}\x\Stack*{\FRA_\FRM/Z_\FRM}
    \end{tikzcd},
\end{equation}
and similarly for the unitary version.

\subsubsection{}
Apply the mapping \(k\)-stack functor \(\IHom_k(X,\blank)\) to
\eqref{eqn:absolute_stacks}, we obtain the 
following diagram:
\begin{equation}
    \begin{tikzcd}
        \cM^+ \ar[r]\ar[d] & \PicS_X\x\cM^{\Hit+} \ar[d] &\\
        \cA^+  \ar[r] & \PicS_X\x\cA^{\Hit+}\ar[r] & \PicS_X\x\cB^+
    \end{tikzcd}.
\end{equation}
Let \(\cB\subset\cB^+\) be the open substack of boundary divisors, and take its
preimages, then we reach the diagram
\begin{equation}
    \begin{tikzcd}
        \cM \ar[r]\ar[d, "h", swap] & \PicS_X\x\cM^{\Hit} \ar[d, "h^\Hit"] &\\
        \cA  \ar[r] & \PicS_X\x\cA^{\Hit}\ar[r] & \PicS_X\x\cB
    \end{tikzcd}.
\end{equation}
Again, we also have the unitary version
\begin{equation}
    \begin{tikzcd}
        \cM' \ar[r]\ar[d, "h'", swap] & \PicS_X\x\cM^{\prime\Hit} \ar[d, "h^{\prime\Hit}"] &\\
        \cA  \ar[r] & \PicS_X\x\cA^{\Hit}\ar[r] & \PicS_X\x\cB
    \end{tikzcd}.
\end{equation}

\subsubsection{}
More concretely, if we fix a \(Z_\FRM\)-torsor \(\cL\) and line bundle
\(\cD\), the fiber of \(\cA\) over \((\cL,\cD)\) is the affine space
\(\cA_\cL^\Hit\x \cA_{\cL,\cD}''\) where
\begin{align}
    \cA_\cL^\Hit&=\bigoplus_{i=1}^{n-1}\RH^0\bigl(X,\Wt_i(\cL)\bigr),\\
    \cA_{\cL,\cD}''&=\RH^0\bigl(X,\cD^{2}\bigr)\oplus\bigoplus_{i=1}^{n-1}\RH^0\bigl(X,\Wt_i(\cL)\otimes\cD^{2}\bigr).
\end{align}
Note that due to the outer twist, each individual line bundle \(\Wt_i(\cL)\)
may only make sense over \(X'\), but the sum \(\Wt_i(\cL)\oplus\Wt_{n-i}(\cL)\)
is always well-defined over \(X\).

\subsubsection{}
The geometric and topological properties of \(h^\Hit\) and \(h^{\prime\Hit}\) are
studied in \Cref{sec:supplement_on_the_usual_mh_fibrations}. We let
\(\cA^\heartsuit\) be the open subset of \(\cA\) such that the generic point of
\(X\) is mapped to the strongly regular semisimple locus. By definition, the
image of \(\cA^\heartsuit\) in \(\cA^\Hit\) is contained in
\(\cA^{\Hit,\heartsuit}\). Similar to \(\cA^{\Hit,\heartsuit}\),
\(\cA^\heartsuit\) is non-empty when \(\cL\) is sufficiently \(G\)-ample
(equivalently, \(G'\)-ample) and \(\cD\) is sufficiently ample.
Let \(\cA^\SIM\) (resp.~\(\cA^\diamondsuit\)) be the
preimage \emph{in \(\cA^\heartsuit\)} of the corresponding simple loci in \(\cA^\Hit\) (see \Cref{sub:the_simple_locus}),
and \(\cM^\SIM\), \(\cM^{\prime\SIM}\), etc. are defined similarly.


\subsection{The geometry of \texorpdfstring{\(\cM^\SIM\)}{Msim}} 
\label{sub:the_geometry_of_cM_sim}

In this subsection, we study some important geometric properties of
\(\cM^\SIM\), including the properness of \(h^\SIM\) and a local model of
singularity of \(\cM^\SIM\). We also have a product formula that connects local
and global geometry.

\subsubsection{}
For any \(m=(\cL,\cD,\cE,x,e,e^\vee )\in\cM(S)\), we have the pure tensors
\begin{align}
    f_n(m)\colon \bigl(\cD^n\otimes\rho(\cL)\bigr)^{-1}\longto \wedge^n V_\cE
\end{align}
and
\begin{align}
    f_n^\vee (m)\colon \wedge^n V_\cE\longto \cD^n\otimes\rho(\cL),
\end{align}
where \(\rho\) is the half-sum of positive roots of \(G\).
The pairing \(f_n^\vee f_n\) is just the discriminant divisor and hence depends only on the
image of \(m\) in \(\cA\). If \(m\in\cM^\SIM\), then both \(e\) and \(e^\vee \) are
necessarily non-zero because \(f_n^\vee f_n\) is non-zero over any geometric point
\(s\in S\).

\subsubsection{}
Over \(X'\x S\),
we also have pure tensors
\begin{align}
    \bFf_i(\OGT^*m)\colon \OGT^*\bigl(\cD^i\otimes\rho_i(\cL)\bigr)^{-1}\longto \wedge^i \OGT^*V_\cE
\end{align}
and
\begin{align}
    \bFf_i^\vee (\OGT^*m)\colon \wedge^i \OGT^*V_\cE\longto \OGT^*\bigl(\cD^i\otimes\rho_i(\cL)\bigr),
\end{align}
where \(\rho_i\) is
\begin{align}
    \rho_i=\sum_{j=1}^i\Wt_j.
\end{align}
Again, \(\rho_i(\cL)\) may not be well-defined on \(X\x S\), but its pullback
to \(X'\x S\) is.

\subsubsection{}
By definition, the total stack \(\cM\) fits into the Cartesian diagram
\begin{equation}
    \begin{tikzcd}
        \cM \ar[r, "\orto{h}"]\ar[d, "\olto{h}", swap] & \orto{\cM} \ar[d]\\
        \olto{\cM} \ar[r] & \PicS_X\x \cM^\Hit
    \end{tikzcd}
\end{equation}
where \(\olto{\cM}\) (resp.~\(\orto{\cM}\)) classifies tuples
\((\cL,\cD,\cE,x,e)\) (resp.~\((\cL,\cD,\cE,x,e^\vee )\)). We define
\(\olto{\cM}^\heartsuit\) (resp.~\(\orto{\cM}^\heartsuit\)) be the substack
where the mHiggs bundle lies in \(\cM^{\Hit,\heartsuit}\) and the tensor
\(f_n(m)\) (resp.~\(f_n^\vee(m)\)) is non-zero.

Recall the \(N\)-truncated evaluation map relative to \(\cB\)
(cf.~\Cref{sub:the_local_model_of_singularity}):
\begin{align}
    \ev_N^\Hit\colon \cM^\Hit\longto \Stack*{\SFQ_X}_N,
\end{align}
and by \Cref{thm:topological_local_model_of_singularity}, we have, up to
a cohomological shift and Tate twist, an isomorphism
\begin{align}
    \bigl(\ev_N^{\Hit*}\IC_{\Stack*{\SFQ_X}_N}\bigr)_{|\cM^{\Hit,\dagger}}\simeq \IC_{\cM^{\Hit,\dagger}},
\end{align}
where \(\cA^{\Hit,\dagger}\subset\cA^{\Hit,\heartsuit}\) is an open substack
whose complement has bounded dimension. Let \(\cA^{\dagger}\) be the pullback of \(\cA^{\Hit,\dagger} \) to \(\cA\).
\begin{definition}
Let \(\cA^{\ddagger}=\cA^{\SIM}\cap \cA^{\dagger}\) and \(\cM^\ddagger\) its
preimage under \( h: \cM \to \cA \).
\end{definition}
We also have \(\olto{\cM}^\ddagger\) and \(\orto{\cM}^\ddagger\)
defined in an obvious way. We let \(\ev_N\) (resp.~\(\olto{\ev}_N\),
resp.~\(\orto{\ev}_N\)) the restriction of \(\ev_N^\Hit\) to \(\cM\)
(resp.~\(\olto{\cM}\), resp.~\(\orto{\cM}\)).

\begin{proposition}
    \label[proposition]{prop:one_sided_local_model_of_singularity}
    We have an isomorphism up to cohomological shifts and Tate twists
    \begin{align}
        \bigl(\olto{\ev}_N^{*}\IC_{\Stack*{\SFQ_X}_N}\bigr)_{|\olto{\cM}^{\ddagger}}&\simeq
        \IC_{\olto{\cM}^{\ddagger}},\\
        \bigl(\orto{\ev}_N^{*}\IC_{\Stack*{\SFQ_X}_N}\bigr)_{|\orto{\cM}^{\ddagger}}&\simeq
        \IC_{\orto{\cM}^{\ddagger}}.
    \end{align}
\end{proposition}
\begin{proof}
    The proof is similar to Yun's Claim~1 in the proof of
    \cite[Proposition~3.2.6]{Yu11}.
    We first prove this for \(\orto{\cM}^\ddagger\). Let \(S\) be the spectrum
    of an Artinian \(\bar{k}\)-algebra with residue field \(\bar{k}\), and
    \(m=(\cL,\cD,\cE,x,e^\vee)\in \orto{\cM}^\ddagger(S)\) is an \(S\)-point.
    Similar to the case of \(\cM^\Hit\), since \(X\) is a curve, the deformation
    of \(m\) relative to \(\PicS_X\x\cB\) (so that \(\cD\) is fixed) up to its
    obstructions, is controlled by the truncated (tangent) complex
    \begin{equation}
        \begin{tikzcd}[column sep=large]
            \cT^{\le 0}\defeq [\ad(\cE) \ar[r, "\symup{act}\oplus e^{\vee*}"] &
            m^*\CoTB_{\FRM/\FRA_\FRM}^\vee\oplus \bigl(V_\cE^\vee\otimes\cD\bigr)],
        \end{tikzcd}
    \end{equation}
    where \(\CoTB_{\FRM/\FRA_\FRM}\) is the relative cotangent sheaf of \(\FRM\)
    over \(\FRA_\FRM\), and the terms are placed at degrees \(-1\) and \(0\)
    respectively. Here we use the fact that \(\ad(\cE)=\End(V_\cE)\) and the
    map
    \begin{align}
        e^{\vee*}\colon\ad(\cE)\longto V_\cE^\vee\otimes\cD
    \end{align}
    sends \(\psi\) to \(e^\vee\circ\psi\).

    Let
    \begin{align}
        \cC&=\ker\Bigl(m^*\CoTB_{\FRM/\FRA_\FRM}^{\vee\vee}\oplus \bigl(V_\cE\otimes\cD^{-1}\bigr)\longto \ad(\cE)^\vee\Bigr)
    \end{align}
    and
    \begin{align}
        \cC'&=\ker\Bigl(m^*\CoTB_{\FRM/\FRA_\FRM}^{\vee\vee}\longto \ad(\cE)^\vee\Bigr),
    \end{align}
    then we have an exact sequence
    \begin{align}
        0\longto \cC'\longto \cC\longto V_\cE\otimes\cD^{-1},
    \end{align}
    and by Serre duality,
    \begin{align}
        \RH^1(X\x S, \cT^{\le 0})=\RH^0(X\x S,\cC\otimes\omega_X)^\vee,
    \end{align}
    where \(\omega_X\) is the canonical sheaf of \(X\).

    We claim that \(\cC'\) is in fact equal to \(\cC\). Since \(\cC/\cC'\) is
    torsion-free, it suffices to show that \(\cC\) and \(\cC'\) coincide after
    restricting to \(\Spec{R}\defeq\Spec{K}\x S\) where \(\Spec{K}\) is the
    generic point of \(X\). Over \(\Spec{R}\), we may trivialize \(\OGT\),
    \(\cL\), \(\cD\), and \(\cE\) in such a way that the mHiggs field \(x\) may be
    identified with an element in \(\SL_n(R)\). Since \(m\) is contained in
    \(\cM^\heartsuit\), it
    is  strongly regular semisimple over \(\Spec{R}\), and so we may further
    assume that \(x\) is a diagonal matrix
    \begin{align}
        \Diag(x_1,\ldots,x_n)\in\SL_n(R)
    \end{align}
    with \(x_i\in R\), and \(e^\vee\) is a row vector
    \begin{align}
        (e_1,\ldots,e_n)\in (R^n)^\vee,
    \end{align}
    such that \(\Set*{e^\vee,e^\vee x,\ldots,e^\vee x^{n-1}}\) form a basis of
    \((R^n)^\vee\). In particular, this implies that \(e_i\neq 0\) for all
    \(1\le i\le n\).

    The relative cotangent sheaf \(\CoTB_{\FRM/\FRA_\FRM}\) may be identified
    with \(\La{sl}_n^\vee\) at \(x\), and so the map of coherent sheaves given
    by the transpose of \(\cT^{\le 0}\) can be written as
    \begin{align}
        \label{eqn:dual_map_in_deform_over_R}
        \La{sl}_n(R)^\vee\oplus R^n&\longto \La{gl}_n(R)^\vee.
    \end{align}
    Since \(x\) is diagonal, we may describe the extended map
    as follows: if \(a_{ij}\in\La{gl}_n(R)^\vee\) is the \((i,j)\)-th matrix coefficient
    on \(\La{gl}_n(R)\) and \(u=\tp(u_1,\ldots,u_n)\in R^n\), then
    \eqref{eqn:dual_map_in_deform_over_R}
    sends \(A\oplus u\) (where \(A=(a_{ij})\)) to the function
    \begin{align}
        y=(y_{ij})\longmapsto \sum_{i,j=1}^n y_{ij}(1-x_i x_j^{-1}+e_iu_j).
    \end{align}
    Therefore, if \(A\oplus u\) is sent to the zero function, we must have
    \begin{align}
        1-x_i x_j^{-1}+e_iu_j=0
    \end{align}
    for each \(i,j\). In particular, if \(i=j\), then \(e_iu_i=0\), and
    so \(u_i=0\) for all \(i\) since \(e_i\neq 0\). As a result, the kernel of
    \eqref{eqn:dual_map_in_deform_over_R} must lie in the direct summand
    \(\La{sl}_n(R)^\vee\), which implies \(\cC=\cC'\) as claimed.

    Let \(a\) be the image of \(m\) in \(\cA^\Hit\). If \(m\) is very
    \((G,\delta_a)\)-ample, then this proposition is proved in a neighborhood of
    \(m\) because our previous discussion shows that the obstruction to
    deforming \(m\) relative to the map \(\orto{\ev}_N\) is the same as that
    related to the mHitchin version \(\ev_N^\Hit\), which vanishes by
    \Cref{prop:weak_local_model_of_singularity}. In the case that \(m\) is not
    necessarily very \((G,\delta_a)\)-ample, we may replace \(\ev_N\) by
    \(\ev_N^1\)
    (cf.~\Cref{prop:strong_local_model_of_singularity}). Since the above
    analysis on \(\cC\) and \(\cC'\) is insensitive to the change of the
    boundary divisor, we still have the relevant obstruction for the
    map \(\orto{\ev}_N^1\) coincides with the obstruction related to the map
    \(\ev_N^1\), which vanishes on \(\cM^{\Hit,\dagger}\). In any case, we
    always have the isomorphism as claimed in the proposition.

    The case for \(\olto{\cM}^\ddagger\)
    can be proved similarly by replacing \(V_\cE\) with \(V_\cE^\vee\) and
    \(e^\vee\) by the transpose of \(e\). The discussion about
    \eqref{eqn:dual_map_in_deform_over_R} is insensitive to the vector bundle
    structure of \(V_\cE\) but depends only on the fact that its generic fiber
    is a vector space. This finishes the proof.
\end{proof}

\begin{lemma}
    \label[lemma]{lem:simple_relative_smoothness}
    Let \(m=(\cL,\cD,\cE,x,e,e^\vee)\in \cM^\SIM(\bar{k})\). Suppose
    \begin{align}
        \label{eqn:degree_for_relative_smoothness}
        n\deg\cD>\rho(\cL)+2n(g_X-1)
    \end{align}
    and \(\deg{V_\cE}\ge 0\) (resp.~\(\deg{V_\cE}\le 0\)), then the map
    \(\olto{h}\) (resp.~\(\orto{h}\)) is smooth at \(m\).
\end{lemma}
\begin{proof}
    The deformation of \(\cM\) relative to \(\olto{\cM}\) is controlled by the
    locally-free sheaf \(V_\cE\otimes\cD\). More precisely, \(\olto{h}\) is
    smooth at \(m\) if
    \begin{align}
        \RH^1(\breve{X},V_\cE\otimes\cD)
    \end{align}
    vanishes. By Serre duality, it is the same as the condition
    \begin{align}
        \label{eqn:the_vectors_to_vanish_for_deformation}
        \Hom_{\breve{X}}(\cD\otimes\omega_X^{-1},V_\cE^\vee)=0.
    \end{align}
    Recall the pure tensors \(f_n(m)\) and \(f_n^\vee(m)\), which fits
    \(\wedge^nV_\cE\) in between a line bundle and its dual as follows:
    \begin{align}
        \bigl(\cD^n\otimes\rho(\cL)\bigr)^{-1}\longto \wedge^n V_\cE \longto \cD^n\otimes\rho(\cL).
    \end{align}
    In fact, the construction of \(f_n(m)\) still makes sense if we replace
    \(e\) by an arbitrary (twisted) vector
    \begin{align}
        u\colon \cD'\longto V_\cE,
    \end{align}
    where \(\cD'\) is an arbitrary line bundle. We denote the resulting pure
    tensor by \(f_n(m,u)\). Similarly, we have a pure (co-)tensor
    \(f_n^\vee(m,u^\vee)\) if we replace \(e^\vee\) by another co-vector.
    In addition, by the definition of \(\FRM\), the mHiggs field
    \(\sigma^*x=\iota^*x\) inside \(\OGT_*\bsM\) (recall \(\bsM\) is the split
    form of \(\sM\)) also induces pure tensors \((\sigma^*f_n)(m,u)\) from a
    vector
    \begin{align}
        u\colon\cD'&\longto V_\cE^\vee,
    \end{align}
    and similarly we have \((\sigma^*f_n^\vee)(m,u^\vee)\).

    Let \(\cD'=\cD\otimes\omega_X^{-1}\) and suppose for the sake of
    contradiction we have a non-zero vector \(u\) in
    \eqref{eqn:the_vectors_to_vanish_for_deformation}. If the pure tensor
    \((\sigma^*f_n)(m,u)\neq 0\), then we have an inclusion of line bundles
    \begin{align}
        (\sigma^*f_n)(m,u)\colon (\cD')^n\otimes\rho(\cL)^{-1}\longto\wedge^n
        V_\cE^\vee.
    \end{align}
    However, if \(\deg{V_\cE}\ge 0\), then by the assumption on the degree of
    \(\cD\) we have
    \begin{align}
        \deg{\cD'}-\deg{\cL}>0\ge \deg{V_\cE^\vee}=\deg{\wedge^n V_\cE^\vee},
    \end{align}
    which is a contradiction. If \((\sigma^*f_n)(m,u)=0\), then since \(u\neq
    0\), we may find the largest \(i\le n\) such that
    \(\bFf_i(\OGT^*m,\OGT^*u)\neq 0\). Restricting to the generic point of
    \(\breve{X}'\), we see that generically \(\OGT^*x\) (viewed as an
    element of \(\Mat_n\) via the first fundamental representation of
    \(\bM\)) stabilizes a non-trivial partial flag, which is impossible by
    \Cref{cor:simple_locus_no_reduction} because \(m\in\cM^\SIM\).

    Finally, the case of \(\orto{h}\) is similar: the analogue of
    \eqref{eqn:the_vectors_to_vanish_for_deformation} is the condition
    \begin{align}
        \Hom_{\breve{X}}(\cD\otimes\omega_X^{-1},V_\cE)=0,
    \end{align}
    and the rest of proof proceeds accordingly when \(\deg{V_\cE}\le 0\).
\end{proof}

\begin{theorem}
    \label[theorem]{thm:local_model_of_singularity_JR_symmetric}
    Over each irreducible component of \(\cM^\ddagger\) such that
    \begin{align}
        n\deg\cD>\rho(\cL)+2n(g_X-1)
    \end{align}
    holds, we have an isomorphism
    up to a cohomological shift and a Tate twist:
    \begin{align}
        \bigl(\ev_N^{*}\IC_{\Stack*{\SFQ_X}_N}\bigr)_{|\cM^{\ddagger}}\simeq
        \IC_{\cM^{\ddagger}}.
    \end{align}
\end{theorem}
\begin{proof}
    Since we either have \(\deg{V_\cE}\ge 0\) or \(\deg{V_\cE}\le 0\), the
    theorem immediately follows from
    \Cref{prop:one_sided_local_model_of_singularity,lem:simple_relative_smoothness,thm:topological_local_model_of_singularity}.
\end{proof}

\subsubsection{}
Another important aspect of \(\cM^\ddagger\) is its equidimensionality, which is
essential for proving smallness in \Cref{prop:smallness}.

\begin{proposition}
    \label[proposition]{prop:equidimensionality_of_cM_sim}
    For any irreducible component of \(\PicS_X\x\cB\) such that
    \eqref{eqn:degree_for_relative_smoothness} holds, its preimage in
    \(\cM^\ddagger\) is equidimensional.
\end{proposition}
\begin{proof}
    Once the component of \(\PicS_X\x\cB\) is fixed, the degrees of
    \(\rho(\cL)\) and \(\cD\) are fixed. For simplicity, we replace
    \(\cM^\ddagger\) (resp.~\(\olto{\cM}^\ddagger\), \(\orto{\cM}^\ddagger\)) by
    the corresponding preimage. We first compute the local dimensions of
    \(\olto{\cM}^\ddagger\) and \(\orto{\cM}^\ddagger\). The proof of
    \Cref{prop:one_sided_local_model_of_singularity} shows that the preimage
    of the big-cell locus of \(\Stack*{\SFQ_X}\) in \(\olto{\cM}^\ddagger\)
    (resp.~\(\orto{\cM}^\ddagger\)) is smooth and dense, and so we only need to
    compute the dimension over the smooth locus. We compute this for
    \(\orto{\cM}^\ddagger\); the case for \(\olto{\cM}^\ddagger\) is similar.

    Over the smooth locus, the dimension at any point is the difference between
    the dimension of the tangent space and that of the infinitesimal
    automorphism group. Similar to
    \Cref{prop:one_sided_local_model_of_singularity} (and the discussions in
    \cite[\S~7.1]{Wa25}), for any \(\orto{m}\in \orto{\cM}(\bar{k})\), the
    infinitesimal automorphism group is given by
    \begin{align}
        \RH^0\bigl(\breve{X},\ker\cT^{\le 0}\bigr),
    \end{align}
    whereas before \(\cT^{\le 0}\) is the two-step complex
    \begin{align}
        \ad(\cE)\longto \orto{m}^*\CoTB_{\FRM/\FRA_\FRM}^\vee\oplus \bigl(V_\cE^\vee\otimes\cD\bigr),
    \end{align}
    also viewed as a map between coherent sheaves. Since \(m\) lies in the
    big-cell locus, the sheaf \(\orto{m}^*\CoTB_{\FRM/\FRA_\FRM}\) is in fact
    locally-free. The tangent space fits into a short exact sequence
    \begin{align}
        0\longto\RH^1(\breve{X},\ker{\cT^{\le 0}})\longto \RH^0(\breve{X},\cT^{\le 0})\longto
        \RH^0(\breve{X},\coker{\cT^{\le 0}})\longto 0.
    \end{align}
    As discussed, we have
    \begin{align}
        \dim_{\orto{m}}\orto{\cM}^\ddagger
        &=\dim_{\bar{k}}\RH^0(\breve{X},\cT^{\le
        0})-\dim_{\bar{k}}\RH^0(\breve{X},\ker\cT^{\le 0})\\
        &=\dim_{\bar{k}}\RH^0(\breve{X},\coker{\cT^{\le
        0}})+\dim_{\bar{k}}\RH^1(\breve{X},\ker\cT^{\le 0})
        -\dim_{\bar{k}}\RH^0(\breve{X},\ker\cT^{\le 0}).
    \end{align}
    Since we proved in \Cref{prop:one_sided_local_model_of_singularity} that
    \begin{align}
        \RH^1(\breve{X},\coker{\cT^{\le 0}})=0,
    \end{align}
    we can write \(\dim_{\orto{m}}\orto{\cM}^\ddagger\) in terms of Euler
    characteristics
    \begin{align}
        \dim_{\orto{m}}\orto{\cM}^\ddagger
        &=\chi(\coker{\cT^{\le 0}})-\chi(\ker{\cT^{\le 0}})\\
        &=-\chi\bigl(\ad(\cE)\bigr)+\chi\bigl(m^*\CoTB_{\FRM/\FRA_\FRM}^\vee\oplus
        (V_\cE^\vee\otimes\cD)\bigr).
    \end{align}
    Since we know \(\cM^{\Hit,\dagger}\) is equidimensional (recall we
    already fixed the component in \(\cB\)), with the same
    argument we know that
    \begin{align}
        \dim{\cM^{\Hit,\dagger}}=-\chi\bigl(\ad(\cE)\bigr)+\chi\bigl(m^*\CoTB_{\FRM/\FRA_\FRM}^\vee\bigr)
    \end{align}
    is constant, and as a result
    \begin{align}
        \dim_{\orto{m}}{\orto{\cM}^\ddagger}=\dim{\cM^{\Hit,\dagger}}+\chi(V_\cE^\vee\otimes\cD).
    \end{align}
    Similarly,
    \begin{align}
        \dim_{\olto{m}}{\olto{\cM}^\ddagger}=\dim{\cM^{\Hit,\dagger}}+\chi(V_\cE\otimes\cD).
    \end{align}
    If \(\deg{V_\cE}\ge 0\), then the map \(\orto{h}^\ddagger\) is
    smooth by \Cref{lem:simple_relative_smoothness} and obviously
    schematic. In particular, we know
    \begin{align}
        \RH^1(\breve{X},V_\cE\otimes\cD)=0.
    \end{align}
    Let \(m\in\cM^\ddagger(\bar{k})\) be any point with
    \(\orto{h}(m)=\orto{m}\). Then
    \begin{align}
        \dim_m{\cM^\ddagger}
        &=\dim_{\orto{m}}{\orto{\cM}^\ddagger}+\dim_{\bar{k}}\RH^0(\breve{X},V_\cE\otimes\cD)\\
        &=\dim_{\orto{m}}{\orto{\cM}^\ddagger}+\chi(V_\cE\otimes\cD)\\
        &=\dim{\cM^{\Hit,\dagger}}+\chi(V_\cE^\vee\otimes\cD)+\chi(V_\cE\otimes\cD)\\
        &=\dim{\cM^{\Hit,\dagger}}+2n\deg\cD-2n(g_X-1),
    \end{align}
    which is a constant as \(\deg\rho(\cL)\) and \(\deg\cD\) are fixed. If
    \(\deg{V_\cE}\le 0\), we switch the role of \(\olto{\cM}^\ddagger\) and
    \(\orto{\cM}^\ddagger\) and we are done.
\end{proof}

\subsubsection{}
As the last part of this subsection, we show that \(h^\heartsuit\) is a proper
relative algebraic space using valuative criteria.

\begin{proposition}
    \label[proposition]{prop:h_heartsuit_is_proper_alg_space}
    The map \(h^\heartsuit\colon\cM^\heartsuit\to \cA^\heartsuit\) is
    representable by proper algebraic spaces.
\end{proposition}
\begin{proof}
    Since \(\cM^\Hit\) is an algebraic stack locally of finite type, it is clear
    that so is \(\cM\). Thus, we only need to verify valuative criteria for
    properness. Let \(R\) be a discrete valuation ring over \(k\) and \(K\)
    its field of fractions.

    Suppose \(m=(\cL,\cD,\cE,x,e,e^\vee)\in\cM^\heartsuit(K)\) that maps to a
    point \(a\in\cA^\heartsuit(R)\). We view \(m\) as a map
    \begin{align}
        X\x\Spec{K}\longto \Stack*{\sM/Z_\FRM\x \Gm}.
    \end{align}
    Since over the strongly regular semisimple
    locus, the map \(\sM^\srs\to \FRC_\sM^\srs\) is an isomorphism by (the
    Galois twisted version of) \Cref{prop:srs_implies_iso_to_GIT},
    we can find some open dense subset \(U\subset X\) such that \(m_{|U\x\Spec{K}}\)
    extends to a map
    \begin{align}
        U\x\Spec{R}\longto \Stack*{\sM/Z_\FRM\x \Gm}.
    \end{align}
    This gives a \(Z_\FRM\)-torsor, a line bundle, and a \(G\)-torsor, all over
    \(U\x\Spec{R}\) and compatible with \(\cL\), \(\cD\), \(\cE\), respectively
    over \(U\x\Spec{K}\). Since \(X\x \Spec{R}\) is normal of dimension \(2\)
    and the complement \(X\x\Spec{R}-U\x\Spec{R}-X\x\Spec{K}\) has codimension
    \(2\), we may extend the tuple \((\cL,\cD,\cE)\) to the entire
    \(X\x\Spec{R}\). Then by Hartogs' theorem, we can also extend the mHiggs
    field \(x\), the vector \(e\), and the co-vector \(e^\vee\) as well. This
    shows the existence part of the valuative criteria.

    For the uniqueness part, again since \(\sM\) is generically isomorphic
    to \(\FRC_\sM\), the extension of \((\cL,\cD,\cE)\) to \(U\x\Spec{R}\) is
    uniquely determined by \(a\). The rest part of the extension is then
    uniquely determined as well. The same consideration also shows that the
    automorphism groups of the fibers of \(\cM^\heartsuit\) over
    \(\cA^\heartsuit\) are trivial. This finishes the proof.
\end{proof}


\subsection{The geometry of \texorpdfstring{\(\cM^{\prime\SIM}\)}{M'sim}} 
\label{sub:the_geometry_of_cM_sim_prime}

In this subsection, we prove the results for \(\cM^{\prime\SIM}\) analogous to
those in \Cref{sub:the_geometry_of_cM_sim}. The proofs will be essentially the
same, and in some cases easier.

\subsubsection{}
Similar to the \(\cM\) case, for any \(m'=(\cL,\cD,\cE',x',e')\in\cM'(S)\), we
have a pure tensor
\begin{align}
    f_n'(m')\colon \OGT^*\bigl(\cD^n\otimes\rho(\cL)\bigr)^{-1}\longto \wedge^n V'_{\cE'},
\end{align}
which induces a pure co-tensor (and vice versa)
\begin{align}
    f_n^{\prime\vee}(m')\colon\wedge^n V'_{\cE'}\longto
    \OGT^*\bigl(\cD^n\otimes\rho(\cL)\bigr).
\end{align}
Since \(f_n'(m')\) and \(f_n^{\prime\vee}(m')\) determines each other (with
given \(m'\)), we do not have analogues to \(\olto{\cM}\) or
\(\orto{\cM}\). Instead, deformations of \(m'\) are much simpler than
\(m\in\cM\).

\begin{lemma}
    \label[lemma]{lem:unitary_relative_smoothness}
    Let \(m'=(\cL,\cD,\cE',x',e')\in \cM^{\SIM}(\bar{k})\). Suppose
    \eqref{eqn:degree_for_relative_smoothness} holds,
    then the map \(\cM'\to\cM^{\prime\Hit}\) is smooth at \(m'\).
\end{lemma}
\begin{proof}
    The proof is similar to \Cref{lem:simple_relative_smoothness}. The
    deformation of \(\cM'\) relative to \(\cM^{\prime\Hit}\) is controlled by
    the locally-free sheaf \(\OGT_*V'_{\cE'}\otimes\cD\), and by Serre duality,
    it suffices to prove that
    \begin{align}
        \label{eqn:unitary_case_the_vectors_to_vanish_for_deformation}
        \Hom_{\breve{X}}(\cD\otimes\omega_X^{-1},\OGT_*(V'_{\cE'})^\vee)=0.
    \end{align}
  
    Let \(\cD'=\cD\otimes\omega_X^{-1}\) and suppose for the sake of
    contradiction we have a non-zero vector \(u'\) in
    \eqref{eqn:unitary_case_the_vectors_to_vanish_for_deformation}. Similar
    to \Cref{lem:simple_relative_smoothness}, we have the pure tensor
    \begin{align}
        f_n'(m',u')\colon \OGT^*\bigl((\cD')^n\otimes
        \rho(\cL)^{-1}\bigr)\longto \wedge^n (V'_{\cE'})^\vee\simeq\wedge^n\sigma^*V'_{\cE'}.
    \end{align}
    If \(f_n'(m',u')\neq 0\), then it is an inclusion of line bundles.
    However, since \(V'_{\cE'}\) is Hermitian, \(\deg{V'_{\cE'}}=0\). Then by
    the assumption on the degree of \(\cD\) we have
    \begin{align}
        \deg{\OGT^*\cD'}-\deg{\OGT^*\cL}>0=\deg{(V'_{\cE'})^\vee},
    \end{align}
    a contradiction. If \(f_n'(m',u')=0\), then since \(u'\neq 0\), we may find
    the largest \(i\le n\) such that \(\bFf_i(m',u')\neq 0\). Then the restriction
    of \(\OGT^*x'\) to the generic point of \(\breve{X}'\) stabilizes a
    non-trivial partial flag, which is impossible by
    \Cref{cor:simple_locus_no_reduction} because \(m'\in\cM^{\prime\SIM}\). This
    finishes the proof.
\end{proof}

\subsubsection{}
The evaluation map \(\ev_N^{\prime\Hit}\) on \(\cM^{\prime\Hit}\) induces the
evaluation map
\begin{align}
    \ev_N'\colon \cM'\longto \Stack*{\SFQ_X}_N.
\end{align}
We define \(\cM^{\prime\ddagger}\) to be the preimage of
\(\cM^{\prime\Hit,\dagger}\) in \(\cM^{\prime\SIM}\).

\begin{theorem}
    \label[theorem]{thm:local_model_of_singularity_JR_unitary}
    Over each irreducible component of \(\cM^{\prime\ddagger}\) such that
    \eqref{eqn:degree_for_relative_smoothness}
    holds, we have an isomorphism up to a cohomological shift and a Tate twist:
    \begin{align}
        \bigl(\ev_N^{\prime*}\IC_{\Stack*{\SFQ_X}_N}\bigr)_{|\cM^{\prime\ddagger}}\simeq
        \IC_{\cM^{\prime\ddagger}}.
    \end{align}
\end{theorem}
\begin{proof}
    Immediate from \Cref{lem:unitary_relative_smoothness,thm:topological_local_model_of_singularity}.
\end{proof}

\subsubsection{}
Similar to \Cref{prop:equidimensionality_of_cM_sim}, we also have
\begin{proposition}
    \label[proposition]{prop:equidimensionality_of_cM_prime_sim}
    For any irreducible component of \(\PicS_X\x\cB\) such that
    \eqref{eqn:degree_for_relative_smoothness} holds, its preimage in
    \(\cM^{\prime\ddagger}\) is equidimensional.
\end{proposition}
\begin{proof}
    The proof is similar to \Cref{prop:equidimensionality_of_cM_sim} and in fact
    much simpler. We leave it to the reader.
\end{proof}

\subsubsection{}
Lastly, we have the results about properness and representability of
\(h^{\prime\heartsuit}\). The proof proceeds exactly the same as in
\Cref{prop:h_heartsuit_is_proper_alg_space} and so we
omit.
\begin{proposition}
    \label[proposition]{prop:h_prime_heartsuit_is_proper_alg_space}
    The map \(h^{\prime\heartsuit}\colon\cM^{\prime\heartsuit}\to \cA^\heartsuit\) is
    representable by proper algebraic spaces.
\end{proposition}


\subsection{The product formulae} 
\label{sub:the_product_formulae}

In this subsection, we connect the global formulations with the local ones using
the \notion{product formula}. The statement for Jacquet--Rallis mH-fibers is
much simpler than that for the usual mH-fibers, because there are no Picard
stacks in the picture.

\subsubsection{}
Given \(a\in\cA^\heartsuit(k')\) be a \(k'\)-point where \(k'\) is a
\(k\)-field, we have the mH-fibers \(\cM_a\) and \(\cM'_a\). Let \(\abs{X_{k'}}\) be
the set of closed points in \(X_{k'}\). For any point \(v\in\abs{X_{k'}}\), we
also have the Jacquet--Rallis fibers \(\cM_{v}(a)\) and \(\cM'_v(a)\). Using
Beauville-Laszlo descent, we may define maps of functors defined over \(k'\):
\begin{align}
    \label{eqn:product_formulae_gluing_map}
    \prod_{v\in\abs{X_{k'}}}\cM_v(a)\longto \cM_a,\quad\text{resp.~}
    \prod_{v\in\abs{X_{k'}}}\cM'_v(a)\longto \cM'_a,
\end{align}
where the direct products on the left-hand side are necessarily finite.

\subsubsection{}
Concretely, the above gluing map is described as
follows: the \(k'\)-scheme \(\cM_v(a)\) classifies tuples, for any
\(k'\)-algebra \(R\)
\begin{align}
    (\cL,\cD,\cE,x,e,e^\vee)
\end{align}
over the formal disc \(X_{k',v}\hat{\x}_{k'}R\) whose image in
\(\Stack*{\FRC_\sM/Z_\FRM\x\Gm}\) is the restriction of \(a\) to \(X_{k',v}\hat{\x}_{k'}R\).
Since \(\cL\) and \(\cD\) are already determined by \(a\), we only need to
extend the tuple \((\cE,x,e,e^\vee)\) from the formal discs to the whole
curve compatibly. However, since \(a\in\cA^\heartsuit\) and
\(\sM^\srs\to\FRC_\sM^\srs\) is an isomorphism, such extension exists and is
uniquely determined by the restriction of \(a\) to \emph{any} punctured formal
disc \(X_{k',v}^\bullet\). The case of \(\cM_v'(a)\) is similar.

\subsubsection{}
Conversely, the restriction functor gives the maps in the opposite direction
\begin{align}
    \label{eqn:product_formulae_restriction_map}
    \cM_a\longto \prod_{v\in\abs{X_{k'}}}\cM_v(a),\quad\text{resp.~}
    \cM'_a\longto\prod_{v\in\abs{X_{k'}}}\cM'_v(a).
\end{align}
It is clear that \eqref{eqn:product_formulae_gluing_map} and
\eqref{eqn:product_formulae_restriction_map} are inverses to each
other as maps of sheaves. Therefore, we have proved that:

\begin{proposition}
    \label[proposition]{prop:product_formulae}
    For any \(k\)-field \(k'\) and \(a\in\cA^\heartsuit(k')\), we have
    isomorphisms
    \begin{align}
        \cM_a\simeq\prod_{v\in\abs{X_{k'}}}\cM_v(a),\quad\text{resp.~}
        \cM'_a\simeq\prod_{v\in\abs{X_{k'}}}\cM'_v(a).
    \end{align}
    In particular, \(\cM_a\) and \(\cM_a'\) are projective schemes.
\end{proposition}


\subsection{Stratified smallness} 
\label{sub:stratified_smallness}

Similar to the Lie algebra case, the fibrations \(h\) and \(h'\) turn out to be
small maps after restricting to a certain open subset of \(\cA^\SIM\). In fact, it
is even stratified-small in the sense of \cite{MiVi07}, which is necessary for
the support theorem since the total stacks \(\cM\) and \(\cM'\) are not smooth.
The proof is based
on a simple dimension estimate with the help with the usual mH-fibration.
Although the result is admittedly rather crude, it is sufficient for the
purposes of this paper. It would be interesting to see a more conceptual proof
(likely with a stronger conclusion), preferably without the use of the usual
mH-fibration.

\subsubsection{}
To avoid flooding the argument with notations, we will start by proving a
weaker result in smallness rather than stratified smallness. The latter will be
a consequence of the former by utilizing the inductive structure in
mH-fibrations.

Let \(a\in\cA^\heartsuit(\bar{k})\) and \(a^\Hit\) its image in \(\cA^\Hit\). To
simplify notations, we will effectively not distinguish between \(a\) and
\(a^\Hit\) and denote the usual mH-fiber \(\cM_{a^\Hit}^\Hit\) by
\(\cM_a^\Hit\) and the global (resp.~local) \(\delta\)-invariants
\(\delta_{a^\Hit}\) (resp.~\(\delta_v(a^\Hit)\)) by \(\delta_a\)
(resp.~\(\delta_v(a)\)), etc. This should not cause any confusion.

\begin{lemma}
    \label[lemma]{lem:dimension_esitmate_using_delta}
    For any \(a\in\cA^\heartsuit(\bar{k})\) and any closed point \(\bar{v}\in
    \abs{\breve{X}}\), we have
    \begin{align}
        \dim\cM_{\bar{v}}(a)=\dim\cM_{\bar{v}}'(a)\le\delta_{\bar{v}}(a).
    \end{align}
    Consequently, we have
    \begin{align}
        \dim\cM_a=\dim\cM'_a\le\delta_a.
    \end{align}
\end{lemma}
\begin{proof}
    For the local statement: the first equality follows from the fact that
    \(\cM_{\bar{v}}(a)\) and \(\cM_{\bar{v}}'(a)\) are isomorphic
    \(\bar{k}\)-schemes, and the inequality is because that \(\cM_{\bar{v}}(a)\)
    is a closed subscheme of \(\cM_{\bar{v}}^\Hit(a)\). The global statement
    follows from the local one and \Cref{prop:product_formulae}.
\end{proof}

\subsubsection{}
For each \(\delta\ge 0\), let \(\cA_{\ge\delta}^\Hit\)
(resp.~\(\cA_{\le\delta}^\Hit\)) be the closed (resp.~open) subset of
\(\cA^{\Hit,\heartsuit}\) such that \(\delta_a\ge
\delta\) (resp.~\(\delta_a\le\delta\)). By \Cref{prop:delta_regularity}, for a
given \(\delta\), there exists \(N=N(\delta)\) such that for all
\(\delta'\le\delta\), any irreducible component of \(\cA_{\ge\delta'}\) that is
very \((G,N)\)-ample has codimension at least \(\delta'\) in \(\cA^\Hit\). Let
\(\cA_{\ge\delta}\) (resp.~\(\cA_{\le\delta}\)) be the preimage of
\(\cA_{\ge\delta}^\Hit\) (resp.~\(\cA_{\le\delta}^\Hit\)) in \(\cA^\heartsuit\).

\begin{proposition}
    \label[proposition]{prop:smallness}
    Fix \(\delta\ge 0\) and \(N=N(\delta)\) as above. Let
    \(\cA_{(\delta)}\subset\cA_{\le\delta}^\ddagger\) be the union of
    irreducible components that are very \((G,N)\)-ample and satisfies
    \begin{align}
        \label{eqn:deg_of_cD_relative_to_cL_and_delta}
        n\deg\cD\ge \rho(\cL)+2n(g_X-1)+n\delta.
    \end{align}
    Then the maps
    \begin{align}
        h_{(\delta)}\colon \cM_{(\delta)}&\longto\cA_{(\delta)},\\
        h_{(\delta)}'\colon \cM_{(\delta)}'&\longto\cA_{(\delta)}
    \end{align}
    are small.
\end{proposition}
\begin{proof}
    The proof is similar to that of \cite[Proposition~3.5.2]{Yu11}.
    Since the statement is topological, we base change to \(\bar{k}\) and omit
    \(\bar{k}\) from notations for simplicity. For the same reason, we will
    not distinguish \(X\) with the geometric curve \(\breve{X}\). By
    \Cref{lem:dimension_esitmate_using_delta}, for the codimension-estimate part
    of smallness, it suffices to prove  for \(h_{(\delta)}'\).

    Recall that the pairing \(f_n^\vee f_n=f_n^{\prime\vee}f_n'\) is a morphism
    of line bundles on \(X\):
    \begin{align}
        f_n^\vee f_n\colon\bigl(\cD^n\otimes\rho(\cL)\bigr)^{-1}\longto
        \cD^n\otimes\rho(\cL),
    \end{align}
    which depends only on the invariant \(a\in\cA\). If
    \(a\in\cA^\heartsuit\), then \(f_n^\vee f_n\) is an inclusion. Let
    \(c\in\cA^\Hit\) be the image of \(a\), then we have
    \(\FRD_c\) be the usual extended discriminant divisor. Let
    \(\FRD_c'\subset\FRD_c\) be the \emph{reduced} subdivisor such that
    \(v\in\abs{\FRD_c}\) if and only if \(\delta_v(c)\neq 0\).

    For each fixed \(c\in\cA^{\Hit}(\bar{k})\) in the image of
    \(\cA_{(\delta)}\) (so that \(Z_\FRM\)-torsor \(\cL\) is fixed) and line
    bundle \(\cD\) such that \(\deg\cD\) satisfies
    \eqref{eqn:deg_of_cD_relative_to_cL_and_delta}, let
    \(\cA_{c,\cD}''\subset\cA_{\cL,\cD}''\) be the open subset where \(f_n^\vee
    f_n\) is an isomorphism at all points in \(\FRD_c'\). By
    \Cref{prop:product_formulae}, for any \(c''\in\cA_{c,\cD}''\) and
    \(a=(c,c'')\), the Jacquet--Rallis mH-fiber \(\cM_a'\) (and \(\cM_a\)) is
    non-empty and finite.

    By \Cref{lem:unitary_relative_smoothness}, the map from \(\cM_c'\) (the
    fiber of \(\cM'\) over \(c\)) to \(\cM_c^{\prime\Hit}\) is smooth. Moreover,
    the proof shows that the fibers of this map are open dense subsets of
    vector spaces, therefore, it induces an injection of sets of irreducible
    components:
    \begin{align}
        \Irr\cM_c'\longto\Irr\cM_c^{\prime\Hit}.
    \end{align}
    As a result, \(\cM_c'\) is equidimensional since \(\cM_c^{\prime\Hit}\) is by
    \Cref{prop:usual_mH_fiber_equidimensional}. Since \(\cM_c'\to\cA_{\cL,\cD}''\)
    is generically over the target a quasi-finite map with non-empty fibers, we
    have
    \begin{align}
        \dim Z=\dim \cA_{c,\cD}'',
    \end{align}
    for any irreducible component \(Z\) of \(\cM_c'\), because of
    equidimensionality.

    We shall prove in \Cref{lem:generic_dominance_for_smallness} that for any
    such \(Z\), the map \(Z\to\cA_{\cL,\cD}''\) is dominant, hence generically
    finite. Assuming this result, then for any \(c\in\cA^\Hit\) and any \(j\ge
    0\), the locus of \(c''\in\cA_{\cL,\cD}''\) such that
    \(\dim\cM_{(c,c'')}'=j\) has codimension at least \(j+1\).

    The codimension-estimate part of smallness is
    then straightforward: for any \(a\in\cA_{(\delta)}\), we know by
    \Cref{lem:dimension_esitmate_using_delta} that \(\dim{\cM_a'}\le\delta\).
    For any \(0\le j\le \delta\), let \(\cA_{(\delta),j}\subset\cA_{(\delta)}\)
    be the locus where \(\dim\cM_a'=j\). Then the image \(A_j\) of \(\cA_{(\delta),j}\)
    in \(\cA^\Hit\) is contained in \(\cA_{\ge
    j}^\Hit\cap\cA_{\le\delta}^\Hit\). Therefore, we have
    \begin{align}
        \codim_{\cA_{(\delta)}}\cA_{(\delta),j}\ge\codim_{\cA_{\le\delta}^\Hit}A_j+(j+1)\ge 2j+1.
    \end{align}

    Finally, the codimension estimates above together with
    \Cref{prop:equidimensionality_of_cM_sim,prop:equidimensionality_of_cM_prime_sim}
    imply that every irreducible component of \(\cM_{(\delta)}\)
    (resp.~\(\cM_{(\delta)}'\)) must map dominantly to
    \(\cA_{(\delta)}^\ddagger\). This finishes the proof of smallness.
\end{proof}

\begin{lemma}
    \label[lemma]{lem:generic_dominance_for_smallness}
    For any \(Z\) as in the proof of \Cref{prop:smallness}, the map
    \(Z\to\cA_{\cL,\cD}''\) is dominant, hence generically finite.
\end{lemma}
\begin{proof}
    We retain all the notations from the proof of \Cref{prop:smallness}.
    We will show that the preimage of \(\cA_{c,\cD}''\) in \(Z\) is non-empty
    (in particular, it also shows that \(\cA_{c,\cD}''\) must be dense in
    \(\cA_{\cL,\cD}''\)). Let \(Z^\Hit\) be the image of
    \(Z\) in \(\cM_c^{\prime\Hit}\), then \(Z\) is the preimage of \(Z^\Hit\).
    Pick a geometric point in \(Z^\Hit\) such that its (usual) mHiggs bundle
    \((\cL,\cD,\cE',x')\) is regular. We want to show that there exists some
    element
    \begin{align}
        e'\in\Hom_X(\cD^{-1},\OGT_*V'_{\cE'})=\RH^0(X,\OGT_*V'_{\cE'}\otimes\cD)
    \end{align}
    such that its induced \(f_n'(m')\) is non-vanishing at any point in
    \(\FRD_c'\). By duality, this would imply that
    \(f_n^{\prime\vee}(m')f_n'(m')\neq 0\), hence proving this lemma.

    By \eqref{eqn:delta_local_and_global_summation}, we have that
    \begin{align}
        \delta\ge \delta_c\ge\deg{\FRD_c'}=\Cnt\abs{\FRD_c'},
    \end{align}
    where \(\Cnt\abs{\FRD_c'}\) is the number of closed points contained in
    \(\FRD_c'\). At each \(v\in \abs{\FRD_c'}\), we look at the fiber of the
    sheaf \(\OGT_*V'_{\cE'}\otimes\cD\), which is a \(2n\)-dimensional
    \(\bar{k}\)-vector space \(V_v'\). For any vector \(u_v'\in V_v\), the fiber
    of \((\cL,\cD,\cE',x')\) at
    \(v\) induces a \(\bar{k}\)-point \(m_v'\in \Stack*{\sM/Z_\FRM\x \Gm}\) for
    which the construction of the pure tensor \(f_n'(m_v',u_v')\) makes sense.
    Since \(\bar{k}\) is algebraically closed, we may trivialize \(\OGT\), \(\cL\),
    \(\cD\) and \(\cE\) at \(v\) and identify the fiber of mHiggs field \(x'\)
    at \(v\) with a diagonal point in \(\bM\x \bM\) and \(u_v'\) with a vector in
    \(\bar{k}^n\oplus\bar{k}^n\).

    Since we are free to choose among the mHiggs bundles over \(c\), we may so
    choose it that \(x'\) is in fact the (diagonal embedding of the) image of
    the section \(\epsilon_\bM\) induced by the companion matrix. Then, if we
    choose \(u_v'=(\bFe,\bFe)\), we
    can explicitly compute (similar to what is done in
    \Cref{sec:invariant_theory}) that \(f_n'(m_v',u_v')\neq 0\). Let
    \(L_v\subset V_v'\) be the line spanned by such a \(u_v\).

    Consider the short exact sequence of coherent sheaves on \(X\) induced by
    the evaluation maps at \(v\in\abs{\FRD_c'}\):
    \begin{align}
        0\longto \cK\longto \OGT_*V'_{\cE'}\otimes\cD\longto
        \bigoplus_{v\in\abs{\FRD_c'}}L_v\longto 0.
    \end{align}
    Then \(\cK\) is locally-free of rank \(2n\). We claim that
    \begin{align}
        \RH^1(X,\cK)\simeq\RH^0(X,\cK^\vee\otimes\omega_X)^\vee=0.
    \end{align}
    Indeed, suppose we have a non-zero map \(\omega_X^{-1}\to\cK^\vee\), then
    the intersection of \(\omega_X^{-1}\) with
    \((\OGT_*V'_{\cE'}\otimes\cD)^\vee\) is a line bundle \(\omega'\) of
    degree at least \(2-2g_X-\delta_c\), and we have inclusion
    \begin{align}
        u_{\omega'}\colon\omega'\longto (\OGT_*V'_{\cE'}\otimes\cD)^\vee.
    \end{align}
    Form the pure tensor \(f_n'\) using the mHiggs bundle and \(u_{\omega'}\),
    which is a map of line bundles
    \begin{align}
        (\omega')^n\otimes\rho(\cL)^{-1}\longto \wedge^n(\OGT_*V'_{\cE'}\otimes\cD)^\vee.
    \end{align}
    Since we have
    \begin{align}
        n\deg{(\omega')}-\deg\rho(\cL)\ge
        n(2-2g_X-\delta_c)-\deg\rho(\cL)> -n\deg\cD=\deg\bigl((\OGT_*V'_{\cE'}\otimes\cD)^\vee\bigr),
    \end{align}
    this pure tensor must be \(0\). Similar to
    \Cref{lem:unitary_relative_smoothness}, one can then use \(\bFf_i'\) on
    \(X'\) to show that the mHiggs bundle has a non-trivial parabolic reduction
    over \(X'\), which is a contradiction since we assume that \(c\) lies in the
    simple locus \(\cM^{\Hit,\SIM}\). Therefore, \(\RH^1(X,\cK)\) must vanish,
    and as a result we have a surjective linear map
    \begin{align}
        \RH^0(X,\OGT_*V'_{\cE'}\otimes\cD)\longto \bigoplus_{v\in\abs{\FRD_c'}}L_v.
    \end{align}
    We may then take \(e'\) to be the preimage of any point in the right-hand
    side that does not vanish at any \(v\in\abs{\FRD_c'}\). This finishes the
    proof.
\end{proof}

\subsubsection{}
Suppose we have a map of irreducible varieties over \(\bar{k}\)
\begin{align}
    f\colon S\longto T.
\end{align}
Assume that \(\cF\defeq\Qlb[\dim S]\) is a pure perverse sheaf of weight \(0\) (e.g.,
when \(S\) is smooth) and \(f\) is proper and small. Then it is well-known that
\(f_*\cF\) is pure perverse of weight \(0\) whose support is a singleton
\(\Set{T}\). In other words, for any open dense subset \(j\colon U\to T\), we
have
\begin{align}
    f_*\cF\cong j_{!*}j^*(f_*\cF).
\end{align}
The proof of this fact is a simple dimension counting coupled with
Poincar\'e-Verdier duality, which we briefly review here.

First off, \(\cF\) is self-dual by assumption, and since \(f\) is proper,
\(f_*\cF\) is also self-dual of pure weight \(0\). By the decomposition theorem,
we have a non-canonical isomorphism
\begin{align}
    f_*\cF\cong\bigoplus_n \PH^n(f_*\cF)[-n].
\end{align}
Suppose \(Z\subset T\) is an irreducible closed subset that supports a perverse
summand. By upper-semicontinuity of fiber dimensions, we may find an open subset
\(U\subset T\) with \(U\cap Z\neq \emptyset\), such that the fibers of \(f\)
have constant dimension \(d_Z\) over \(U\cap Z\). Let \(n\in \bbZ\) be such that
there exists an irreducible direct summand \(\cK\subset\PH^n(f_*\cF)\) that is
supported on \(Z\). Since \(f_*\cF\) is self-dual, we may assume \(n\ge 0\).
Shrink \(U\) if necessary, we may also assume that \(K\defeq\cK_{|U\cap Z}\) is
an irreducible local system; in particular, \(K[-n]\) is supported on a single
cohomological degree \(n-\dim Z\).

Taking the usual cohomology, we see that \(K[-n]\) is a direct
summand of \(\RH^{n-\dim Z}(f_*\cF)_{|U}\). Since the fibers of \(f\) have
constant dimension \(d_Z\) over \(U\cap Z\), we must have \(n-\dim Z\le
2d_Z-\dim S\), or in other words,
\begin{align}
    \dim S-\dim Z\le 2d_Z-n\le 2d_Z.
\end{align}
Since \(f\) is small, we have by definition \(\dim S=\dim T\), and
so we obtain a variant of the \notion{Goresky-MacPherson inequality}:
\begin{align}
    \label{eqn:GM_inequality}
    \codim_T Z=\dim T-\dim Z\le 2d_Z,
\end{align}
which is impossible unless \(d_Z=0\) and \(Z=T\) due to smallness.

\subsubsection{}
For a general irreducible variety \(S\), the constant sheaf \(\Qlb[-\dim
S]\) is usually not self-dual or pure, so we need to replace it with the
intersection complex \(\IC_S\). However, doing so may render the elementary
amplitude estimate \(n-\dim Z\le 2d_Z-\dim S\) in the above discussion to
fail. Therefore, in this more general setting, we need a more sophisticated
version of smallness condition in order to reach a similar conclusion about the
perverse supports. The formulation we use here is slightly different from
\cite{MiVi07} but functionally the same.

\begin{definition}
    Let \(f\colon S\to T\) be a proper map of algebraic varieties over
    \(\bar{k}\). We call \(f\) a \notion{stratified-small map} if the followings
    hold:
    \begin{enumerate}
        \item for any irreducible component \(Z\subset S\), the image \(f(Z)\)
            is an irreducible component of \(T\), and \(\dim{Z}=\dim{f(Z)}\);
        \item there exists a finite stratification \(\Set{S_i}\) of \(S\) and
            \(\Set{T_j}\) of \(T\) into \emph{smooth}, irreducible,
            locally-closed subsets, such that for every \(i\), \(f(S_i)\) is a
            union of some strata \(T_j\);
        \item for any \(i\), the induced map \(f_{|\bar{S_i}}\colon\bar{S_i}\to
            f(\bar{S_i})\) is small.
    \end{enumerate}
\end{definition}

\begin{lemma}
    \label[lemma]{lem:stratified_small_supports}
    If \(f\colon S\to T\) is stratified-small, then for any open dense
    subset \(j\colon U\to T\), we have
    \begin{align}
        f_*\IC_S\cong j_{!*}j^*(f_*\IC_S).
    \end{align}
\end{lemma}
\begin{proof}
    Without loss of generality, we may assume that both \(S\) and \(T\) are
    irreducible. Suppose \(Z\subset T\) appears in the set of supports. The
    argument for \eqref{eqn:GM_inequality} works for
    \(\IC_S\) verbatim until the estimate of cohomological amplitude, namely the
    inequality \(n-\dim Z\le 2d_Z-\dim S\). Therefore, we only need to show
    that this inequality holds. We retain open subset \(U\subset T\) in the
    discussion about \eqref{eqn:GM_inequality}.

    Let \(S_i\) be any stratum such that \(f(S_i)\cap Z\) is dense in
    \(Z\). For convenience, we denote by \(S_0\) the unique open stratum. Shrink
    \(U\) if necessary, we may assume that \(f_{|S_i}\)
    still has constant fiber dimension over \(U\cap Z\), which we denote by
    \(d_{Z,i}\). Let \(n_i\) be the codimension of \(S_i\) in \(S\), then by the
    construction of \(\IC_S\), we know that \((\IC_S)_{|S_i}\) is supported on
    cohomological degrees in the interval \([-\dim S,-\dim S+n_i-1]\) if
    \(n_i>0\), or is the constant sheaf \(\Qlb[-\dim S]\) if \(n_i=0\) (i.e.,
    when \(i=0\)).

    If \(n_i>0\), an elementary use of the spectral sequence for cohomology with
    compact support (see
    \cite[\href{https://stacks.math.columbia.edu/tag/0BKK}{Tag 0BKK}]{stackP25},
    applied to the filtration of \(S\) by open subsets formed by unions of
    strata) shows that the complex
    \begin{align}
        \bigl[f_{|S_i!}(\IC_S)_{|S_i}\bigr]_{|U\cap Z}
    \end{align}
    is bounded above by cohomological degree \(-\dim S+n_i-1+2d_{Z,i}\).
    If \(n_i=0\) (and \(f(S_0)\cap Z\) is dense in \(Z\)), then the same complex
    is bounded above by \(-\dim S+2d_{Z,0}\le -\dim S+2d_Z\).

    The same spectral sequence argument shows then \((f_*\IC_S)_{|Z\cap U}\) is
    then bounded above by degree
    \begin{align}
        \max\Set*{\max_{i\neq 0}\Set{-\dim S+n_i-1+2d_{Z,i}}, -\dim S+2d_Z},
    \end{align}
    where \(i\) ranges over indices such that \(f(S_i)\cap Z\) is dense in
    \(Z\). If the maximum is achieved at \(-\dim S+2d_Z\), then we obtain
    inequality \eqref{eqn:GM_inequality}, and it implies \(Z=T\) as before. If
    the maximum is achieved at some \(i\neq 0\), then we have
    \begin{align}
        n-\dim Z\le -\dim S+n_i-1+2d_{Z,i},
    \end{align}
    where \(n\ge 0\) as before. This implies that
    \begin{align}
        \codim_T Z=\dim T-\dim Z=\dim S-\dim Z\le 2d_{Z,i}+n_i-1.
    \end{align}
    On the other hand, since \(f\) is
    stratified-small, \(f_{|\bar{S_i}}\) is small, which implies that
    \begin{align}
        \dim f(\bar{S_i})=\dim \bar{S_i}=\dim S-n_i=\dim T-n_i.
    \end{align}
    Thus, we have
    \begin{align}
        \codim_{f(\bar{S_i})}Z=\codim_T Z-\codim_T f(\bar{S_i})\le 2d_{Z,i}-1,
    \end{align}
    which is impossible because \(f_{|\bar{S_i}}\) is small. This finishes the
    proof.
\end{proof}

\begin{remark}
    It is clear that \Cref{lem:stratified_small_supports} still holds if \(S\)
    and \(T\) are separated Deligne-Mumford stacks of finite types over
    \(\bar{k}\).
\end{remark}

\begin{theorem}
    \label[theorem]{thm:stratified_smallness}
    With the notations and assumptions in \Cref{prop:smallness}, the maps
    \(h_{(\delta)}\) and \(h_{(\delta)}'\) are stratified-small.
\end{theorem}
\begin{proof}
    As in \Cref{prop:smallness}, we base change everything to \(\bar{k}\) and drop
    \(\bar{k}\) from the notations for simplicity.
    We first stratify the moduli of boundary divisors \(\cB\) in an obvious way:
    topologically, each irreducible component of \(\cB\) is homeomorphic to a
    direct product of some symmetric power of curves that are finite \'etale
    covers of \(X\) (see \cite[\S~5]{Wa25}), and we stratify it in such a way
    that the boundary divisor is locally constant along each stratum.
    We stratify \(\cA\), \(\cM\), and \(\cM'\) accordingly.

    We then refine these stratifications using the local model of
    singularity. For this purpose we may fix the \(Z_\FRM\)-torsor \(\cL\) and
    the boundary divisor \(b\). Write the cocharacter-valued divisor \(\lambda_b\)
    as
    \begin{align}
        \lambda_b=\sum_{i=1}^d \lambda_i\cdot v_i,
    \end{align}
    where \(\lambda_i\) stays locally constant. The local model
    \(\Stack*{\SFQ_X}\) may then be identified with the stack
    \begin{align}
        \prod_{i=1}^d\Stack*{\Arc_{v_i}{G^\AD}\backslash\Gr_{G^\AD,v_i}^{\le
        -w_0(\lambda_i)}},
    \end{align}
    where \(w_0\) is the longest element in the Weyl group, and similarly for
    \(\Stack*{\SFQ_X'}\). The Schubert cells in \(\Stack*{\SFQ_X}\)
    (resp.~\(\Stack*{\SFQ_X'}\)) induces a stratification on \(\cM_{(\delta)}\)
    (resp.~\(\cM_{(\delta)}'\)) into smooth locally-closed substacks. Both
    stratifications on \(\cM_{(\delta)}\) and \(\cM_{(\delta)}'\) induce the
    same stratification on \(\cA_{(\delta)}\), which we can easily describe
    below using the inductive structure on the usual mH-fibrations.

    We look at the Schubert cell
    \begin{align}
        \prod_{i=1}^d\Stack*{\Arc_{v_i}{G^\AD}\backslash\Gr_{G^\AD,v_i}^{-w_0(\mu_i)}},
    \end{align}
    where \(\mu_i<\lambda_i\), and we denote its closure by
    \(\Stack*{\SFQ_\mu}\). The divisor
    \begin{align}
        \mu_b=\sum_{i=1}^d \mu_i\cdot v_i
    \end{align}
    can be lifted to a \(Z_\FRM\)-torsor \(\cL'\) because \(\lambda_b\) can.
    Then locally along the strata where \(\lambda_b\) is locally constant, the
    preimage of \(\Stack*{\SFQ_\mu}\) in \(\cM_{(\delta)}\) is isomorphic to
    the analogue of \(\cM_{(\delta)}\) with \(\cL\) replaced by \(\cL'\).
    As a result, when \(\mu_b\) and \(\lambda_b\) are fixed,
    the corresponding stratum in \(\cA_{(\delta)}\) is induced by the
    linear embedding
    \begin{align}
        \RH^0\Bigl(X,\cD^{2}\oplus\bigoplus_{i=1}^{n-1}\cO_X(\Pair{\Wt_i}{\mu_b})\otimes\bigl(\cO_X\oplus\cD^{2}\bigr)\Bigr)
        \subset
        \RH^0\Bigl(X,\cD^{2}\oplus\bigoplus_{i=1}^{n-1}\cO_X(\Pair{\Wt_i}{\lambda_b})\otimes\bigl(\cO_X\oplus\cD^{2}\bigr)\Bigr).
    \end{align}
    The numerical conditions in \Cref{prop:smallness} clearly still hold for
    \(\cL'\), and so we may apply \Cref{prop:smallness} to the restriction of
    \(h_{(\delta)}\) to the preimage of \(\Stack*{\SFQ_\mu}\), which is
    locally isomorphic to an analogue of \(h_{(\delta)}\) with \(\cL\) replaced
    by \(\cL'\). This proves that \(h_{(\delta)}\) is stratified-small, and the
    case of \(h_{(\delta)}'\) is similar.
\end{proof}



\section{The Proof of the Fundamental Lemma} 
\label{sec:the_proof_of_the_fundamental_lemma}

In this section, we prove \Cref{thm:main_theorem_geometric} by
combining local results in \Cref{sec:local_formulations} and global results in
\Cref{sec:global_formulations}. In particular, we also prove the Jacquet--Rallis
fundamental lemma stated in \Cref{thm:main_theorem}.

\subsection{Approximation property} 
\label{sub:approximation_property}

In order to connect local results with global ones, we will need the following
local-constancy result about affine Jacquet--Rallis fibers:

\begin{proposition}
    \label[proposition]{prop:approximation_property}
    Let \(a_0\in\FRC_\sM(\cO_v)\) be generically strongly regular semisimple.
    Then there exists some \(N\) such that for
    any \(a\in\FRC_\sM(\cO_v)\) with \(a_0\equiv a\bmod \pi_v^N\), the
    followings hold:
    \begin{enumerate}
        \item \(a\) is generically strongly regular semisimple.
        \item \(\delta_v(a^\Hit)=\delta_v(a_0^\Hit)\).
        \item For any field extension \(K/k\), we have canonical isomorphisms
            \begin{align}
                \cM_v^i(a_0)(K)&\simeq\cM_v^i(a)(K),\forall i,\\
                \cM_v'(a_0)(K)&\simeq\cM_v'(a)(K).
            \end{align}
    \end{enumerate}
\end{proposition}
\begin{proof}
    Let \(d\) be the valuation of the \emph{extended} \(H\)-discriminant
    \(\Disc_{\bsM}^+\) (cf.~\Cref{def:extended_disc_divisor}) at \(a_0\).
    Then if \(N>d\), the valuation of \(\Disc_{\bsM}^+\) stays the same for any
    such \(a\), and so \(a\) is generically regular semisimple.

    For the second claim, by \cite[Proposition~4.5.1]{Wa25}, we may enlarge
    \(N\) if necessary, so that the isomorphism class of MASF stays constant
    among those \(a\). In particular, we also have
    \(\delta_v(a)=\delta_v(a_0)\).

    For the third claim, we use the lattice descriptions
    \eqref{eqn:JR_fiber_lattice_description_unitary} and
    \eqref{eqn:JR_fiber_lattice_description_symmetric}.
    We prove the symmetric
    case and the unitary case is similar. We also assume \(K=k=k_v\), because
    the general case is also similar. The lattice description depends on the
    fixed section \(\tilde{m}=(\cE,\gamma,e,e^\vee)\) over \(\cO_v'\), which is
    defined uniformly over \(\FRC_\sM\) because over \(\cO_v'\) the stack
    \(\sM\) becomes split. By the construction of \(\tilde{m}\), \(\cE\) is
    always trivial, and the matrix coefficients of any element in the tuple
    \(\tilde{m}\) do not change modulo \(\pi_v^N\).

    By \Cref{lem:coarse_lattice_bounds}, the lattices appearing in \(\cM_v(a)(k)\)
    are bounded both below and above, and the bounds can be explicitly given by
    \(\gamma_1\), \(e\) and \(e^\vee\). Let \(N_0\) be one such bound.
    Therefore, if \(N\) is sufficiently large compared to \(N_0\), the bounds
    for \(a\) are the same for \(a_0\). Then we claim that any
    lattice \(\Lambda\in\cM_v(a_0)(k)\) is also contained in \(\cM_v(a)(k)\),
    and vice versa. Let \((\cE,\gamma,e,e^\vee)\) be the section over \(a\) and
    \((\cE_0,\gamma_0,e_0,e_0^\vee)\) the one over \(a_0\). Let \(s\) and
    \(s_0\) be the matrices appearing in
    \eqref{eqn:JR_fiber_lattice_description_symmetric} (where it is denoted by
    \(s\)) associated with \(a\) and
    \(a_0\) respectively. Then as long as
    \(N\) is much larger than \(N_0\), the conditions
    \begin{align}
        \gamma_i\Lambda\subset\Lambda, e\in\Lambda, e^\vee\in\Lambda^\vee
    \end{align}
    and
    \begin{align}
        \gamma_{0,i}\Lambda\subset\Lambda, e_0\in\Lambda, e_0^\vee\in\Lambda^\vee
    \end{align}
    are equivalent. Therefore, we only need to show that the conditions
    \begin{align}
        s\Lambda=\bar{\Lambda}
    \end{align}
    and
    \begin{align}
        s_0\Lambda=\bar{\Lambda}
    \end{align}
    are also equivalent. To see this, we note that \(s\) is the (necessarily)
    unique solution of the equation
    \begin{align}
        (\Ad_s^{-1}(\gamma),s^{-1}e,e^\vee s)=\sigma(\gamma,e,e^\vee),
    \end{align}
    which is a system of linear equations on the matrix coefficients of \(s\).
    Consequently, when \(N\) is sufficiently large, the difference between \(s\)
    and \(s_0\) is sufficiently small. Taking \(N\) large enough so that
    \(s\equiv s_0\bmod \pi_v^{2N_0}\), then it is easy to see that the
    conditions \(s\Lambda=\bar{\Lambda}\) and \(s_0\Lambda=\bar{\Lambda}\) are
    the same.
\end{proof}


\subsection{Global matching} 
\label{sub:global_matching}

We retain the setup at the beginning of \Cref{sec:global_formulations}. In
particular, we fix a smooth, projective, geometrically connected curve \(X\)
over \(k\) and an \'etale double cover \(\OGT\colon X'\to X\) such that \(X'\)
is also geometrically connected.

\subsubsection{}
Form the Jacquet--Rallis mH-fibrations \(h\colon\cM\to\cA\) and
\(h'\colon\cM'\to\cA\) associated with the cover \(\OGT\). We now define a local
system \(L_\eta^X\) on \(\cM^\heartsuit\) which is the global incarnation of the
character \(\eta\).

We have a
canonical line bundle \(\det\) over \(X\x\Bun_n\) by taking the determinant of
the universal vector bundle. We abuse notation and still denote the pullback of
\(\det\) to \(X\x\cM\) by \(\det\). There is also another canonical line bundle over
\(X\x\cM\) given by \(\cD^n\otimes\rho(\cL)\) where \(\cD\) is the
line bundle and \(\cL\) is the \(Z_\FRM\)-torsor in the definition of \(\cM\).
The co-tensor \(f_n^\vee\) is then a section of the line bundle
\(\det^{-1}\otimes\cD^n\otimes\rho(\cL)\) over \(X\x\cM\) that is non-vanishing
over \(X\x\cM^\heartsuit\). Taking the associated divisor, we have a canonical
map
\begin{align}
    \Div(f_n^\vee)\colon \cM^\heartsuit\longto \coprod_{i\ge 0}X_i,
\end{align}
where \(X_i\) is the \(i\)-th symmetric power of \(X\) over \(k\).

\subsubsection{}
In \cite{Yu11}, there is a canonical line bundle on \(X_i\) for each \(i\),
denoted by \(L_i^X\), which we now briefly review. For \(i=0\), we let
\(L_i^X=\Qlb\). For each \(i>0\), we
have a finite cover
\begin{align}
    (X')^i\longto X_i,
\end{align}
on which \((\bbZ/2\bbZ)^i\rtimes\FRS_i\) acts (\(\FRS_i\) is the \(i\)-th
permutation group). Let
\(\Gamma\subset(\bbZ/2\bbZ)^i\rtimes\FRS_i\) be the index-\(2\) subgroup
consisting of elements whose \((\bbZ/2\bbZ)^i\) component lies in the kernel
of the summation map \((\bbZ/2\bbZ)^i\to\bbZ/2\bbZ\). The GIT quotient
\((X')^i\git \Gamma\) is an \'etale double cover of \(X_i\), for which we can
associate a local system \(L_i^X\). Let \(L_\eta^X\) be the union of all
\(L_i^X\).

\subsubsection{}
Let \(\cA^\ddagger\subset\cA^\heartsuit\) be the locus in
\Cref{thm:local_model_of_singularity_JR_symmetric,thm:local_model_of_singularity_JR_unitary},
over which we have a local model of singularity for \(\cM^\ddagger\)
(resp.~\(\cM^{\prime\ddagger}\)). Let \(K/k\) be a finite extension, and fix a
point \(a\in\cA^\ddagger(K)\). We denote the boundary divisor of \(a^\Hit\) by
\begin{align}
    \lambda=\sum_{v\in\abs{X_K}}\lambda_v\cdot v.
\end{align}
For each \(\lambda_v\), we write it as a unique linear combination of
fundamental coweights:
\begin{align}
    \lambda_v=\sum_{i=1}^{n-1}d_{v,i}\CoWt_i,
\end{align}
and let \(\cF^{\lambda_v}\) (resp.~\(\cF^{\prime\lambda_v}\)) be the local
Satake sheaf of \(\SYMS_n^\AD\) (resp.~\((G')^\AD\)) corresponding to the dual
representation
\begin{align}
    \label{eqn:symmetric_power_of_reps}
    V^{\lambda_v}\defeq\bigoplus_{i=1}^{n-1}\Sym^{d_{v,i}}V_{\CoWt_i},
\end{align}
where \(V_{\CoWt_i}\) is the \(i\)-th fundamental representation of the group
\(\SL_n\), dual to \(\bG^\AD\). Explicitly, if \(\Set{\mu_v}\) is the set of highest
coweights in the decomposition of
\eqref{eqn:symmetric_power_of_reps} into irreducible \(\SL_n\)-representations
(which contains \(\lambda_v\)), then we have
\begin{align}
    \cF^{\lambda_v}=\IC^{\lambda_v}\oplus\bigoplus_{\mu_v\neq
    \lambda_v}\IC^{\mu_v}\otimes\Hom_{\SL_n}(V_{\mu_v},V^{\lambda_v}).
\end{align}

\begin{lemma}
    \label[lemma]{lem:point_count_product_formulae}
    We have equalities in point counts
    \begin{align}
        \bigl(-q_K^{\OneHalf}\bigr)^{\dim_a\cA}\Cnt_{\IC_{\cM^\ddagger}\otimes L_\eta^X}\cM_a(K)
        &=\prod_{v\in\abs{X_K}}\bigl(-q_K^{\OneHalf}\bigr)^{[K_v:K]\Pair{2\rho}{\lambda_v}}\Cnt_{\cF^{\lambda_v}\otimes L_{v,\eta}}\cM_v(a)(K),\\
        \bigl(-q_K^{\OneHalf}\bigr)^{\dim_a\cA}\Cnt_{\IC_{\cM^{\prime\ddagger}}}\cM_a'(K)
        &=\prod_{v\in\abs{X_K}}\bigl(-q_K^{\OneHalf}\bigr)^{[K_v:K]\Pair{2\rho}{\lambda_v}}\Cnt_{\cF^{\prime\lambda_v}}\cM_v'(a)(K),
    \end{align}
    where \(q_K\) is the number of elements in \(K\), \(K_v\) is the residue
    field of \(v\), and \(\dim_a{\cA}\) is the
    local dimension of \(\cA\) at \(a\).
\end{lemma}
\begin{proof}
    Without loss of generality, we may base change and replace \(k\) by \(K\),
    and so we can assume that \(K=k\).
    By product formulae \Cref{prop:product_formulae}, we have canonical
    isomorphisms
    \begin{align}
        \cM_a(k) &=\prod_{v\in\abs{X}}\cM_v(a)(k),\\
        \cM_a'(k) &=\prod_{v\in\abs{X}}\cM_v'(a)(k).
    \end{align}
    Under the above isomorphisms, by
    \Cref{thm:local_model_of_singularity_JR_symmetric,thm:local_model_of_singularity_JR_unitary},
    as well as \cite[\S~5.5]{Wa25},
    we also have canonical isomorphisms up to shifts and Tate twists
    \begin{align}
        \bigl(\IC_{\cM^\ddagger}\bigr)_{|\cM_a}
        &\simeq \bigotimes_{v\in\abs{X}}\bigl(\cF^{\lambda_v}\bigr)_{|\cM_v(a)},\\
        \bigl(\IC_{\cM^{\prime\ddagger}}\bigr)_{|\cM_a'}
        &\simeq
        \bigotimes_{v\in\abs{X}}\bigl(\cF^{\prime\lambda_v}\bigr)_{|\cM_v'(a)}.
    \end{align}
    It is straightforward to see that
    \begin{align}
        L_\eta^X\simeq \bigotimes_{v\in\abs{X}}L_{v,\eta}.
    \end{align}
    Therefore, to prove the lemma, we only need to match up the \(q_K\)-power
    factors on both sides, which are induced by Tate twists, and the signs,
    which are induced by shifts.

    On the right-hand side, the Tate twist is induced by
    the dimension of the product of all relevant affine Schubert varieties,
    while on the left-hand side, it is induced by the (local) dimension of \(\cM\)
    (resp.~\(\cM'\)). Since \(h\) and \(h'\) are both generically
    finite, the latter is the same as the (local) dimension of \(\cA\). This
    explains both the \(q_K\)-power factors and signs on both sides, and we are done.
\end{proof}

\subsubsection{}
We fix a \(\delta\)
which will be given later in \Cref{sub:from_global_to_local} when we prove
\Cref{thm:main_theorem_geometric}.
We fix an arbitrary component of \(\cA\) such that the
boundary divisor is very \((G,N(\delta))\)-ample.
Then, we restrict to the case such
that the line bundle \(\cD\) is sufficiently ample so that
\eqref{eqn:deg_of_cD_relative_to_cL_and_delta} holds.
This way, we obtain an open substack \(U\subset\cA_{(\delta)}\), so that the local models of
singularity in
\Cref{thm:local_model_of_singularity_JR_symmetric,thm:local_model_of_singularity_JR_unitary}
hold over \(U\) and the restrictions of \(h\) and \(h'\) to \(U\) are stratified
small by \Cref{thm:stratified_smallness}. Let \(\cM_U=h^{-1}(U)\) and
\(\cM_U'=h^{\prime-1}(U)\).

\begin{proposition}
    \label[proposition]{prop:global_matching}
    We have an isomorphism of complexes
    \begin{align}
        \label{eqn:global_matching}
        h_{|U*}(\IC_{\cM_U}^\ssim\otimes L_\eta^X)\cong h_{|U*}'\IC_{\cM_U'}^\ssim,
    \end{align}
    where the superscript \(\ssim\) denotes semisimplification with respect to
    \(\Gal(\bar{k}/k)\).
\end{proposition}
\begin{proof}
    By \Cref{lem:stratified_small_supports} and stratified smallness of
    \(h_{|U}\) and \(h_{|U}'\), both complexes are pure perverse sheaves with
    full support. Therefore, it suffices to prove the isomorphism over any open
    dense subset \(U_0\) of \(U\). We restrict to the locus where:
    \begin{enumerate}
        \item the boundary divisor is multiplicity-free in the sense of
            \cite[Definition~5.1.28]{Wa25}; concretely, at any point \(v\),
            the boundary divisor is either \(0\) or a fundamental
            coweight;
        \item the usual extended discriminant divisor (given by \(\Disc_+\)) is
            multiplicity-free and disjoint from the boundary divisor.
    \end{enumerate}
    This locus is necessarily non-empty by our \(G\)-ampleness assumption.

    Further shrinking \(U_0\), we may assume that both sides of
    \eqref{eqn:global_matching} are geometrically semisimple local systems.
    Thus, by Grothendieck-Lefschetz trace formula and Chebotarev's density
    theorem, it suffices to prove that for any finite extension \(K\) of \(k\)
    and \(a\in U_0(K)\), we have
    \begin{align}
        \Cnt_{\IC_{\cM^\ddagger}\otimes L_\eta^X}\cM_a(K)
        =\Cnt_{\IC_{\cM^{\prime\ddagger}}}\cM_a'(K).
    \end{align}
    By \Cref{lem:point_count_product_formulae}, we may replace the above global
    point counts by local ones, and by our assumptions on \(U_0\), every local
    case is isomorphic to one of the four cases in
    \Cref{sub:some_simple_cases}: the split case, the \(\Ob_v(a)\neq 0\) case,
    and Case A or Case B.
\end{proof}


\subsection{From global to local} 
\label{sub:from_global_to_local}

In this subsection, we finish the proof of \Cref{thm:main_theorem_geometric}.

\subsubsection{}
We fix a \(k\)-point \(v\in \abs{X}\) such that \(\OGT\) is non-split over
\(v\). The split case is already covered by \Cref{sub:some_simple_cases}, as is
the case where \(\Ob_v(a)\neq 0\).
Let \(a_0\in\FRC_\sM(\cO_v)\) be the point (denoted by \(a\)) in the statement of
\Cref{thm:main_theorem_geometric} such that \(\Ob_v(a)=0\), and let
\(\delta_0=\delta_v(a_0^\Hit)\).
Let \(N\) be the number that is larger than
both the number in \Cref{prop:approximation_property} and \(N(\delta_0)\), and
let \(U\subset\cA\) be the open subset in \Cref{sub:global_matching} associated
with \(\delta=\delta_0\).

Let \(Z\subset U\) be the locally closed subset consisting of points \(a\) such that:
\begin{enumerate}
    \item its restriction to \(\FRC_\sM(\cO_v/\pi_v^N\cO_v)\) is isomorphic to
        \(a_0\);
    \item for any place other than \(v\), the boundary divisor is disjoint from
        the usual extended discriminant divisor (induced by \(\Disc_+\)), and
        the latter is either \(0\) or multiplicity-free.
\end{enumerate}
This set is non-empty by our ampleness assumptions and Riemann-Roch theorem. Note
that \(Z\) is geometrically irreducible, and there exists some number \(e\) such
that for any \(K/k\) that \([K:k]\ge e\), \(Z(K)\) is non-empty.

By Chebotarev's density theorem, to prove \Cref{thm:main_theorem_geometric} for
\(a_0\), it suffices to prove the same statement for all sufficiently large
\(K\). Therefore, we may replace \(a_0\) by the restriction of any \(a\in Z(K)\)
to \(\cO_v\otimes_k K\).

\subsubsection{}
We aim to prove the equality
\begin{align}
    \Cnt_{\IC_n^{\lambda_v}\otimes L_{v,\eta}}\cM_v(a)(K)
    =\Cnt_{\IC_n^{\prime\lambda_v}}\cM_v'(a)(K),
\end{align}
however, it is easy to see that by induction on the height of \(\lambda_v\), it
suffices to replace the irreducible representation of the dual group \(\SL_n\)
with highest-coweight \(\lambda_v\) by the representation
\eqref{eqn:symmetric_power_of_reps}, and instead prove the alternative equality
\begin{align}
    \Cnt_{\cF^{\lambda_v}\otimes L_{v,\eta}}\cM_v(a)(K)
    =\Cnt_{\cF^{\prime\lambda_v}}\cM_v'(a)(K).
\end{align}
We choose to do this because the latter equality has a more direct connection
with the global setting via \Cref{lem:point_count_product_formulae}.

\subsubsection{}
By \Cref{prop:global_matching} and \Cref{lem:point_count_product_formulae}, we
have an equality
\begin{align}
    \label{eqn:prod_of_local_matching}
    \prod_{u\in\abs{X_K}}\Cnt_{\cF^{\lambda_u}\otimes L_{u,\eta}}\cM_u(a)(K)
    =\prod_{u\in\abs{X_K}}\Cnt_{\cF^{\prime\lambda_u}}\cM_u'(a)(K),
\end{align}
and for any place \(u\) other than \(v\), we also have
\(\cF^{\lambda_u}=\IC^{\lambda_u}\) and
\(\cF^{\prime\lambda_u}=\IC^{\prime\lambda_u}\). By \Cref{sub:some_simple_cases},
\begin{align}
    \Cnt_{\IC_n^{\lambda_u}\otimes L_{u,\eta}}\cM_u(a)(K)
    =\Cnt_{\IC_n^{\prime\lambda_u}}\cM_u'(a)(K).
\end{align}
In order to deduce the equality at \(v\), we need to find \(a\in Z(K)\) such
that the local point counts are non-zero at every \(u\neq v\).

\subsubsection{}
We first choose a random point \(a\in Z(K)\). Replacing \(k\) by \(K\), we may
assume that \(k=K\). Let \(\cL\) be the associated \(Z_\FRM\)-torsor, and
\(\cD\) the \(\Gm\)-torsor.

For any place \(u\neq v\) not contained in \(\FRC_\sM^\srs\), we have a local
obstruction \(\Ob_u(a)\), while \(\Ob_v(a)= 0\) by assumption. Let
\(Y\subset X\) be the open subset contained in strongly regular semisimple
locus, then we have a canonical point \((\cE_Y,x_Y,e_Y)\in
\sM'_{\cL,\cD}(Y)\). At each \(u\neq v\), we choose a point
\(\gamma_u\in\FRM'(F_u)\), and so we have an MASF
\(\cM_u^{\prime\Hit}(a^\Hit)\). Moreover, by the lemma below, we may even choose
\(\gamma_u\in\FRM'(\cO_u)\).

\begin{lemma}
    For any \(u\neq v\) as above, the MASF \(\cM_u^{\prime\Hit}(a^\Hit)\)
    contains at least one \(K\)-point.
\end{lemma}
\begin{proof}
    We know that the MASF at \(u\) is especially
    simple, and is geometrically just a \(\bbZ\)-lattice
    (cf.~\Cref{sub:some_simple_cases}).
    If \(\OGT\) is split over \(u\), then locally we have a Steinberg
    quasi-section \(\FRC_\FRM\to\FRM'\) over \(\cO_u\) and so the conclusion is
    trivial.

    If \(u\) belongs to Case A of \Cref{sub:some_simple_cases}, since \(G'\)
    becomes quasi-split over \(\cO_u\),
    we may choose \(\gamma'\in\FRM'(F_u)\)
    lying over \(a^\Hit\), such that \(\gamma'\) is contained (and elliptic) in
    a subgroup of \((\FRM')^\x\) of semisimple rank \(1\) containing a maximal
    torus of \((\FRM')^\x\) that is induced by a pinning on \(G'\). Such
    a semisimple-rank-\(1\) subgroup is also a Levi subgroup induced by a pair of
    roots of \(G'\). Let \(L'\subset G'\) be the corresponding Levi subgroup.
    Passing to the Levi monoid (see \cite[\S~2.5.39]{Wa25}),
    the MASF in question is then isomorphic to certain MASF of \(L'\).
    Since Steinberg quasi-sections always exist for very flat monoids
    associated with a quasi-split rank-\(1\) groups (see \cite[\S\S~2.2.7,
    2.4.23--2.4.24]{Wa25}), the lemma follows in this special case.

    If \(u\) belongs to Case B, then similar to the considerations in
    \cite[\S~4.6.14]{Wa25}, we may consider the twisted Frobenius at \(u\), and
    unlike the more general setting in \textit{loc.~cit.}, we necessarily have that
    \(a^\Hit\) is regular semisimple at the special point \(u\). This means
    that the twisted Frobenius is contained in
    \begin{align}
        \Norm_{G'}(T')(\cO_u)\rtimes\sigma,
    \end{align}
    namely, the ``cocharacter'' portion of the twisted Frobenius is trivial.
    This means that the element \(h_{\bar{v}}\) (\(\bar{v}\) corresponds to a
    \(\bar{K}\)-point over \(u\)) as in \cite[\S~4.6.14]{Wa25} in
    fact descends to a \(K\)-point in the affine Grassmannian that is
    also a geometric point in \(\cM_u^{\prime\Hit}(a^\Hit)\). In other words, we
    obtain a \(K\)-point in \(\cM_u^{\prime\Hit}(a^\Hit)\).
\end{proof}

\subsubsection{}
The restriction of the pair \((\cE_Y,x_Y)\) to
\(F_u\) induces a map
\begin{align}
    x_{F_u}\colon\cE_{F_u}\longto\FRM_{\cL,\cD,F_u}'.
\end{align}
Both \(\cL\) and \(\cD\) may be trivialized over \(F_u\) since \(Z_\FRM\) is an
induced torus, and we shall fix an arbitrary choice of trivializations. Then
the preimage \(x_{F_u}^{-1}(\gamma_u)\) is a torsor of the regular centralizer
\(\FRJ_{a^\Hit,F_u}'\), whose isomorphism class does not depend on the
trivializations of \(\cL\) and \(\cD\), and is precisely the obstruction
\(\Ob_u(a)\) (see the discussions in
\Cref{sub:frobenius_structures_and_obstructions}).

The collection \(\Ob(a)\defeq (\Ob_u(a))_{u\in\abs{X}}\) (including \(u=v\))
induces a class in the product
\begin{align}
    \prod_u \RH^1(F_u,\FRJ_{a^\Hit}').
\end{align}
At each place \(u\), the cohomology class factors through
\(\Aut(X'/X)\), and for all but finitely many places, the obstruction is trivial
by \Cref{prop:local_obstruction_and_valuation_parity}.
Consequently, we obtain an adelic cohomology class in
\begin{align}
    \RH^1\bigl(F,\FRJ_{a^\Hit}'(\bbA_{\bar{F}})\bigr).
\end{align}
By \cite{Ko86}, we have a commutative diagram
\begin{equation}
    \begin{tikzcd}
        \RH^1\bigl(F,\FRJ_{a^\Hit}'(\bbA_{\bar{F}})\bigr) \ar[r]\ar[d,"{\det}",swap] &
        \RH^1\bigl(F,\FRJ_{a^\Hit}'(\bbA_{\bar{F}})/\FRJ_{a^\Hit}'(\bar{F})\bigr)\simeq\pi_0\bigl(\dual{\FRJ}^{\prime\Gal(\bar{F}/F)}\bigr)^*
        \ar[d,"{\det}"]\\
        \RH^1(F,\bbA_{\bar{F}}^\x) \ar[r] &
        \RH^1(F,\bbA_{\bar{F}}^\x/\bar{F}^\x)\simeq\pi_0\bigl(Z_{\dual{G'}}^{\Gal(\bar{F}/F)}\bigr)^*
    \end{tikzcd}.
\end{equation}
Since \(a^\Hit\) lies in the simple locus, \(\FRJ_{a^\Hit,F}\) is an elliptic
torus, and so we have
\begin{align}
    \pi_0\bigl(\dual{\FRJ}^{\prime\Gal(\bar{F}/F)}\bigr)^*\simeq\pi_0\bigl(Z_{\dual{G'}}^{\Gal(\bar{F}/F)}\bigr)^*\cong\bbZ/2\bbZ.
\end{align}
The image of \(\Ob(a)\) in
\(\pi_0\bigl(Z_{\dual{G'}}^{\Gal(\bar{F}/F)}\bigr)^*\) is then identified with
the summation map
\begin{align}
    \bigoplus_u \bbZ/2\bbZ\longto \bbZ/2\bbZ.
\end{align}
But by \Cref{prop:local_obstruction_and_valuation_parity}, such summation is
simply the degree of the discriminant divisor
\(f_n^{\prime\vee}f_n'=\bFf_n^\vee\bFf_n\) modulo \(2\), and the degree of such
divisor is the same as the degree of the line bundle \(\rho(\cL)^{\otimes
2}\cD^{\otimes 2n}\), which is necessarily even.

This shows that \(\Ob(a)\) is
trivial, and so there is a \(\FRJ_{a^\Hit}'\)-torsor over \(F\) compatible with
all the local torsors. Thus, we obtain a point in the global Picard
\(\cP_a'(K)\), such that it transports the mHiggs bundle \((\cE_Y,x_Y)\) to
another one \((\cE_Y',x_Y')\) that is compatible with the elements \(\gamma_u\).
This way we obtain an mHiggs bundle over \(X-\Set{v}\), whose image in
\(\FRC_{\FRM,\cL,\cD}(X-\Set{v})\) is the same as that of \(a\). Note that over
\(F_v\), the restrictions of \((\cE_Y,x_Y)\) and \((\cE_Y',x_Y')\) are the same,
because \(\Ob_v(a)\) is trivial.

\subsubsection{}
Over the double cover \(\cO_v'\) of \(\cO_v\), the group \(G'\) is split, and so
we have a point in \(\cM_v'(a)(K')\) (\(K'\) being the quadratic extension of
\(K\) in \(F_v'\)). Since \(\GL_n\) does not have endoscopy
(namely, \(\RH^1(F_v',\FRJ_{a^\Hit}')\) always vanishes),
this point can be glued with the mHiggs bundle over \(X-\Set{v}\) above to get
an mHiggs bundle \((\cE',x')\) (\(\cL\) is fixed) over \(X'\).
We also have a vector \(e_v'\in(\cO_v')^n\) using the explicit section in
\Cref{sub:deformed_quotient_stack}.

Again using the fact that \(a^\Hit\) is simple and using
the similar argument as in \Cref{lem:unitary_relative_smoothness}, we may lift
\(e_v'\) to a vector
\begin{align}
    e'\colon \cD^{-1}\longto \cE'.
\end{align}
The restriction of the tuple \((\cE',x',e')\) to \(X-\Set{v}\) maps to a
point in \(\FRC_{\sM,\cL,\cD}(X-\Set{v})\), and its restriction to the punctured
disc \(X_v^\bullet\) is the same as that of \(a\). Gluing them together, we
obtain another point in \(Z(K)\) and we replace \(a\) by this point.

\subsubsection{}
For any \(u\neq v\), the restriction of the tuple \((\cE',x',e')\) to \(X_u\) is
a \(K\)-point in \(\cM_u'(a)\), and so
\begin{align}
    \Cnt_{\IC_n^{\lambda_u}\otimes L_{u,\eta}}\cM_u(a)(K)=
    \Cnt_{\IC_n^{\prime\lambda_u}}\cM_u'(a)(K)\neq 0.
\end{align}
Cancelling out these terms in \eqref{eqn:prod_of_local_matching}, we
finally obtain the desired equality
\begin{align}
    \Cnt_{\cF^{\lambda_v}\otimes L_{v,\eta}}\cM_v(a)(K)=
    \Cnt_{\cF^{\prime\lambda_v}}\cM_v'(a)(K).
\end{align}
This finishes the proof of \Cref{thm:main_theorem_geometric} hence our main result \Cref{thm:main_theorem}.



\appendix

\section{Supplement on the Usual MH-fibrations} 
\label{sec:supplement_on_the_usual_mh_fibrations}

Since the usual multiplicative Hitchin fibration associated with
\(\Stack*{\FRM/G}\) (as \(\FRM\) here is a monoidal symmetric space rather than
a monoid), is not covered by \cite{Wa25}, we prove the relevant results here.
All results will just be a slight
tweak from \textit{loc.~cit.} Similarly, we also need to cover the case
\(\Stack*{\FRM'/G'}\) because the unitary group \(G'\) may not be quasi-split
globally.

We will try to refer to \cite{Wa25} as much as possible to keep this
section concise. As a result, we will keep the notations
consistent with \textit{loc. cit.}, and may deviate slightly from the ones we
use in the main body of this paper.
As before, we let \(\OGT\colon X'\to X\) be the \'etale double cover inducing
\(\FRM\) from \(\bM\), and we still denote by \(\bM'\) (resp.~\(\bC_{\bM'}\),
resp.~\(\bA_{\bM'}\), resp.~\(\bG'\)) the Weil restriction of \(\bM\) (resp.~\(\bC_\bM\),
resp.~\(\bA_\bM\), resp.~\(\bG\)) along \(\OGT\).

\subsection{The basic construction} 
\label{sub:the_basic_construction}

Before we start, we note that since later on we will need to study similar
constructions related to twisted Levi subgroups, we can no
longer afford to restrict ourselves to the case where \(\bM\) is the universal
monoid. Instead, we let \(\bM\) to be an arbitrary very flat monoid (see
\cite[\S~2.3]{Wa25}) associated
with \(\bG^\SC\) such that the \(\sigma\)-action on the abelianization of the
universal monoid lifts to an action on \(\bA_\bM\). In this case, one can
easily verify that the monoidal symmetric space \(\FRM\) is well-defined, as is
the unitary monoid \(\FRM'\).

\subsubsection{}
Recall we have maps of quotient stacks
\begin{align}
    \Stack*{\FRM/G\x Z_\FRM}\longto
    \Stack*{\FRC_\FRM/Z_\FRM}\longto\Stack*{\FRA_\FRM/Z_\FRM},
\end{align}
whose mapping stacks from curve \(X\) give
\begin{align}
    \cM_X^+\longto \cA_X^+\longto \cB_X^+.
\end{align}
There is an open substack \(\cB_X\subset \cB_X^+\) consisting of
\notion{boundary divisors}, and by taking its preimages we obtain maps of
algebraic stacks
\begin{align}
    \cM_X\stackrel{h_X}{\longto} \cA_X\longto \cB_X.
\end{align}

\subsubsection{}
Note that the abelianization \(\FRA_\FRM\) is the same as that of the
(quasi-split) unitary
monoid \(\FRM'\), and \(Z_\FRM\simeq Z_{\FRM'}\) as well, so the moduli
\(\cB_X\) is the same as the moduli of boundary divisors associated with the
unitary monoid, which is covered by the general framework in \cite{Wa25}
(because the boundary divisors are insensitive of inner twists).
Similarly, \(\FRC_\FRM\) is also the same as \(\FRC_{\FRM'}\), and so \(\cA_X\)
is also the mH-base for \(\FRM'\). Therefore, we have maps of moduli stacks
\begin{align}
    \cM_X'\stackrel{h_X'}{\longto} \cA_X\longto \cB_X.
\end{align}

\subsubsection{}
Both the extended discriminant divisor and the regular semisimple locus are
well-defined on \(\FRC_\FRM\) and are \(Z_\FRM\)-equivariant, and we let
\(\cA_X^\heartsuit\) to be the open substack of \(\cA_X\) where the image of
the generic point of \(X\) is contained in the regular semisimple
locus. Note that the notion of being regular semisimple is the same for both
\(\FRM\) and \(\FRM'\).


\subsection{The cameral cover and the Picard action} 
\label{sub:the_cameral_cover_and_the_picard_action}

\subsubsection{}
We have the regular centralizer group scheme for the split model
\begin{align}
    \bJ_\bM\longto \bC_\bM,
\end{align}
which is smooth and commutative, and its generic fiber is a torus. Since
the anti-involution \(\iota\) of \(\bM\) takes \(\bG\)-orbits to \(\bG\)-orbits,
\(\bJ_\bM\) induces a unique group scheme
\begin{align}
    \FRJ_\FRM\longto\FRC_\FRM
\end{align}
such that there is a canonical \(G\)-equivariant homomorphism
\begin{align}
    \chi_\FRM^*\FRJ_\FRM\longto I_\FRM,
\end{align}
where \(I_\FRM\to \FRM\) is the universal stabilizer group scheme of \(G\)
acting on \(\FRM\). Moreover, the above homomorphism is an isomorphism over the
regular locus \(\FRM^\reg\) (meaning where the dimension of the stabilizer
reaches minimal value).

Similarly, we have the regular centralizer group scheme
\begin{align}
    \FRJ_{\FRM'}\longto \FRC_\FRM,
\end{align}
as well as a canonical \(G'\)-equivariant homomorphism
\begin{align}
    \chi_{\FRM'}^*\FRJ_{\FRM'}\longto I_{\FRM'}.
\end{align}

\subsubsection{}
The maximal toric variety \(\bar{\bT}_\bM\) of \(\bM\) is preserved by
\(\iota\), so it induces a twisted toric variety \(\FRT_\FRM\subset \FRM\),
which further induces a canonical isomorphism
\begin{align}
    \FRT_\FRM\git W\stackrel{\sim}{\longto}\FRC_\FRM.
\end{align}

\begin{definition}
    The GIT quotient map \(\pi_\FRM\colon\FRT_\FRM\to\FRC_\FRM\) is called the
    \notion{cameral cover}.
\end{definition}

When restricted to \(\bar{\bT}_\bM\), the anti-involution \(\iota\) is
the same as the involution \(\tau\), and so \(\FRT_\FRM\) is the same as the
maximal toric variety \(\FRT_{\FRM'}\) of \(\FRM'\), and
so the cameral cover \(\pi_\FRM\) coincides with the cameral cover
\(\pi_{\FRM'}\) of \(\FRM'\). In addition, we note that the Weyl group of
\(G'\) is canonically isomorphic to that of \(G\), which we denote by \(W\),
and both can be identified with the (split) Weyl group \(\bW\) of \(\bG\),
because the Galois action on \(\bW\) is trivial.

\subsubsection{}
Like in the case of quasi-split groups, we have a Galois description of
\(\FRJ_\FRM\) (resp.~\(\FRJ_{\FRM'}\)):
\begin{lemma}
    We have canonical isomorphisms
    \begin{align}
        \FRJ_\FRM&\stackrel{\sim}{\longto}\pi_{\FRM*}T^W,\\
        \FRJ_{\FRM'}&\stackrel{\sim}{\longto}\pi_{\FRM*}(T')^W,
    \end{align}
    where the Weyl group \(W\) acts on \(\pi_{\FRM*}T\)
    (resp.~\(\pi_{\FRM*}T'\)) diagonally.
\end{lemma}
\begin{proof}
    The maps are defined the same way as in the group case, and isomorphism can be
    checked by base changing to algebraically closed fields. Note that the maps
    are isomorphisms (instead of merely open embeddings) because the groups \(G\)
    and \(G'\) have simply-connected derived subgroups.
\end{proof}

\subsubsection{}
Using \(\FRJ_\FRM\) (resp.~\(\FRJ_{\FRM'}\)), we have the global Picard stack
\(p_X\colon\cP_X\to \cA_X\) (resp.~\(p_X'\colon\cP_X'\to \cA_X\)), acting on
\(h_X\) (resp.~\(h_X'\)).
Using the cameral cover \(\pi_\FRM\), we have the universal cameral curve just
like the case of quasi-split groups
\begin{align}
    \tilde{\pi}\colon \tilde{X}\longto X\x \cA_X,
\end{align}
and the basic topological properties of \(\cP_X\) (resp.~\(\cP_X'\)) can be
similarly studied using \(\tilde{\pi}\) and the Galois description of
\(\FRJ_\FRM\) (resp.~\(\FRJ_{\FRM'}\)). The argument will be verbatim so we will
only list a small subset of the results essential to the current paper and the
proofs will be omitted. The reader can refer to \cite[\S~6]{Wa25} for more
detailed argument.

\subsubsection{}
By replacing the curve with a formal disc \(X_v\), where \(v\in\abs{X}\) is a
closed point, we can define the multiplicative affine Springer fibers
\(\cM_v(a)\) and the local Picard scheme \(\cP_v(a)\). Let
\(\breve{X}_v=X_v\otimes_k\bar{k}\) be the unramified closure of \(X_v\).

Since \(\FRM\) becomes
split after base changing to \(\breve{X}_v\), all the
topological and geometric properties of \(\cM_v(a)\) (resp.~\(\cP_v(a)\)) that
we need in this paper are already included in \cite[\S~4]{Wa25}. In particular,
we have the local \(\delta\)-invariant \(\delta_v(a)\) which is the common
dimension of \(\cP_v(a)\) and \(\cM_v(a)\), \emph{viewed as a \(k\)-scheme}.

Given a geometric point \(a\in \cA_X^\heartsuit(\bar{k})\), we have the cameral
curve \(\tilde{X}_a\) and the N\'eron model \(\FRJ_a^\flat\) of the regular
centralizer \(\FRJ_a\). Let \(\cP_a^\flat\) be the stack of
\(\FRJ_a^\flat\)-bundles on \(\breve{X}\defeq X\x_k\bar{k}\), then we have a
surjective homomorphism of abelian stacks
\begin{align}
    \cP_a\longto\cP_a^\flat,
\end{align}
whose kernel \(\cR_a\) is an affine group scheme over \(\bar{k}\). The
dimension of \(\cR_a\) is the global \(\delta\)-invariant \(\delta_a\), and we
have the summation formula
\begin{align}
    \label{eqn:delta_local_and_global_summation}
    \delta_a=\sum_{v\in\abs{X}}\delta_{v}(a).
\end{align}

\subsubsection{}
Similarly, since \(\cA_X\) is also the mH-base for \(\FRM'\), we have
the global \(\delta\)-invariant \(\delta_a'\) for the mH-fibration
\(h_X'\). Since over any \(\breve{X}_v\), \(\FRM'\)
(resp.~\(G'\)) is isomorphic to \(\FRM\) (resp.~\(G\)), we have
\begin{align}
    \delta_v(a)=\delta_v'(a),
\end{align}
and so
\begin{align}
    \delta_a=\delta_a'
\end{align}
for any \(a\in\cA_X^\heartsuit(\bar{k})\).

\subsubsection{}
We may also consider the mH-fibration and multiplicative affine Springer fiber
associated with monoid \(\bM'\) and \(\bG'\). The mH-base \(\cA_X\) naturally
embeds into the mH-base \(\cA_{\bG',X}\), and for any \(a\in
\cA_X^\heartsuit(\bar{k})\), its image in \(\cA_{\bG',X}\) is necessarily
contained in \(\cA_{\bG',X}^\heartsuit\), and we have respective relations of
local and global \(\delta\)-invariants:
\begin{align}
    \delta_{\bG',v}(a)&=2\delta_v(a),\\
    \delta_{\bG,a}&=2\delta_a.
\end{align}
All these claims can be proved by
noticing that over \(\breve{X}_v\) the inclusion \(G\subset
\bG'\) is isomorphic to the diagonal embedding \(\GL_n\subset\GL_n\x\GL_n\), and
similarly for the monoids \(\FRM\subset\bM'\).

\subsubsection{}
One may stratify the mH-base \(\cA_X^\heartsuit\) by the \(\delta\)-invariants,
and with the same argument as in \cite{Wa25} one can show that \(\delta\) is
upper-semicontinuous. Therefore, we have the closed subset \(\cA_{\ge
\delta}\subset\cA_X^\heartsuit\) consisting of \(a\) such that \(\delta_a\ge
\delta\). Using the same local-global argument as in
\cite[Proposition~8.4.5]{Wa25}, we have:
\begin{proposition}
    \label[proposition]{prop:delta_regularity}
    For any \(\delta\ge 0\), there exists some \(N=N(\delta)\) such
    that if \(A\subset\cA_X^\heartsuit\) is an irreducible component that is
    very \((G,N)\)-ample, we have
    \begin{align}
        \codim_{A}(A_{\ge\delta'})\ge \delta'
    \end{align}
    for all \(\delta'\le\delta\).
\end{proposition}


\subsection{The local model of singularity} 
\label{sub:the_local_model_of_singularity}

The total stacks \(\cM_X\) and \(\cM_X'\) are not smooth in general; instead,
they admit respective local models of singularity similar to the case of
quasi-split groups.

\subsubsection{}
First, there is a global affine Schubert scheme \(\SFQ_X'\)
parametrized by \(\cB_X\),
defined in the same way using \(\FRM'\) as in
\cite[\S~5.3]{Wa25}. If a geometric point \(b\in\cB_X(\bar{k})\) has an
associated cocharacter \(\lambda_{\bar{v}}\) at each \(\bar{v}\in X(\bar{k})\),
then \((\SFQ_X')_b\) is just a (necessarily finite) direct product of
\(\Gr_{\bar{v},G'}^{\le -w_0(\lambda)}\) for all \(\bar{v}\).
We also have the equivariant
version \(\Stack*{\SFQ_X'}\) by quotienting the arc
group \(\Arc_{\cB_X}(G')^\AD\).

For the symmetric space, there is no \(\SFQ_X\)
defined globally over \(\cB_X\), because there is no group acting ``on one
side'' of the monoidal symmetric space \(\FRM\); rather we have \((\bG')^\SC\) acting
``on both sides'' of \(\FRM\), and \(G\) acting by conjugation (similar to
\Cref{ssub:review_of_SYMS}). Therefore, the stack
\(\Stack*{\SFQ_X}\) is still well-defined over all of \(\cB_X\), but \(\SFQ_X\)
is only defined over some \'etale chart of \(\cB_X\), which also depends on
local splittings \(\FRM\cong \bM\), \(G\cong \bG\), and so on. In any case, it
is not a serious obstacle because our statements in this subsection is local in
\(\cB_X\). We have the canonical evaluation maps
\begin{align}
    \ev\colon \cM_X&\longto \Stack*{\SFQ_X},\\
    \ev'\colon \cM_X'&\longto \Stack*{\SFQ_X'}.
\end{align}

\begin{proposition}
    \label[proposition]{prop:weak_local_model_of_singularity}
    Let \(m\in\cM_X^\heartsuit(\bar{k})\)
    (resp.~\(m'\in\cM_X^{\prime\heartsuit}(\bar{k})\)) and let \(a=h_X(m)\)
    (resp.~\(a=h_X'(m')\)). Suppose
    \(a\) is very \((G,\delta_a)\)-ample (equivalently, very
    \((G',\delta_a)\)-ample), then the evaluation map \(\ev\) (resp.~\(\ev'\)) is
    formally smooth at \(m\) (resp.~\(m'\)).
\end{proposition}
\begin{proof}
    We first prove this for \(\ev\).
    Let \(\Omega=\CoTB_{\FRM/\FRA_\FRM}\) be the relative cotangent sheaf of
    \(\FRM\) over \(\FRA_\FRM\), and
    \(\boldsymbol{\Omega}'=\CoTB_{\bM/\bA_{\bM'}}\). Let
    \(\symup{T}_m\) be the \(\cO_{\breve{X}}\)-dual of \(m^*\Omega\), and
    \(\symbf{T}_m'\) that of \(m^*\boldsymbol{\Omega}'\). Then since
    \(\breve{X}\) is a curve, \(\symup{T}_m\) (resp.~\(\symbf{T}_m\)) must be
    locally-free. Moreover, since \(\symup{T}_m\) is the \(\iota\sigma\)-fixed-point
    subbundle of \(\symbf{T}_m\), it is also a direct summand of the latter
    since \(\Char(k)\neq 2\).

    As we have seen in \cite[\S~7.1]{Wa25}, the obstruction to the local model of
    singularity is contained in the cohomological group
    \begin{align}
        \label{eqn:local_model_of_singularity_coh_obstruction}
        \RH^1\bigl(\breve{X},\coker(\ad(E)\to \symup{T}_m)\bigr).
    \end{align}
    The tangent action map
    \begin{align}
        \Lie{\bG'}\x^G E\longto \symup{\bT}_m
    \end{align}
    is equivariant with respect to the \(\sigma\)-action on the source and the
    \(\iota\sigma\)-action on the target. Taking the fixed-point subbundles, one
    realizes \eqref{eqn:local_model_of_singularity_coh_obstruction} as a direct
    summand of the cohomological group
    \begin{align}
        \RH^1\bigl(\breve{X},\coker(\Lie{\bG'}\x^G E\to \symup{\bT}_m)\bigr).
    \end{align}
    But this is none other than the obstruction to the local model of
    singularity for the mH-fibration associated with the monoid \(\bM'\) and
    group \(\bG'\). If \(m\) is very \((G,\delta_a)\)-ample, then its induced
    mHiggs bundle for \(\bM'\) is very \((\bG',\delta_{\bG',a})\)-ample, and so
    we have
    \begin{align}
        \RH^1\bigl(\breve{X},\coker(\Lie{\bG'}\x^G E\to
        \symup{\bT}_m)\bigr)=0
    \end{align}
    by \cite[Theorem~6.10.2]{Wa25}. This implies that
    \eqref{eqn:local_model_of_singularity_coh_obstruction} also vanishes. The
    case of \(\ev'\) is similar, and we only need to replace both the
    \(\sigma\)-action on \(\bG'\) and the \(\iota\sigma\)-action by the
    respective \(\tau\sigma\)-actions. This finishes the proof.
\end{proof}

\subsubsection{}
There is also a strong version of the local model of singularity similar to
\cite[Theorem~6.10.10]{Wa25} which does not
depend on \(\delta_a\) and the proof can be similarly deduced with the
help of \(\bG'\). Such a proof will be very long to completely spell out, but
its core idea is very simple and not very different from
\Cref{prop:weak_local_model_of_singularity}. In order to state the result, we
will need a bit of technical preparation.

Recall the moduli of boundary divisors \(\cB_X\), which is a smooth
Deligne-Mumford stack. Its coarse space \(B_X\) is a countable disjoint
union of direct product of symmetric powers of smooth projective curves, and
the map \(\cB_X\to B_X\) is an \'etale gerbe. As such, for the purpose of
deformation theory, the difference between \(\cB_X\) and \(B_X\) is not very
important, so for conciseness we will not distinguish the two. A more
rigorous exposition can be found in \cite[\S~5.4, \S\S~7.3.9--7.3.13]{Wa25}.

Let \(m\in\cM_X^\heartsuit(\bar{k})\) be a geometric point. To further simplify
this exposition, we assume (and the general case is not very different):
\begin{enumerate}
    \item the monoid \(\bM\) is the universal monoid of \(\bG^\SC\);
    \item the irreducible component \(B\) of \(\cB_X\) containing the image of
        \(m\) is isomorphic to \(X_{d_1}\x X_{d_2}\), where \(X_d\) means
        \(d\)-th symmetric power of \(X\);
    \item suppose \(b=(D_1,D_2)\) is the image of \(m\) in \(X_{d_1}\x X_{d_2}\),
        then \(D_1\) and \(D_2\) are disjoint divisors on \(X\);
    \item The mHiggs field of \(m\) is regular semisimple at every point in
        \(D_2\).
\end{enumerate}
We let \(p_i\colon B\to X_{d_i}\) be the projection map.

We let \(b\) vary in a neighborhood \(U\) in \(\cB_X\) in such a way that
\(D_1\) and \(D_2\) stay disjoint. Then passing to an \'etale cover \(U'\) of
\(U\), the global affine Schubert
scheme \(\SFQ_X\) is defined and can be written as a direct product
\begin{align}
    (\SFQ_X)_{|U'}=p_1^*\SFQ_{D_1}\x_{U'} p_2^*\SFQ_{D_2},
\end{align}
where \(\SFQ_{D_i}\) is similarly defined as \(\SFQ_X\), with \(\cB_X\) replaced
by \(X_{d_i}\). We similarly have the equivariant
version which is defined over \(U\):
\begin{align}
    \Stack*{\SFQ_X}_{|U}=p_1^*\Stack*{\SFQ_{D_1}}\x p_2^*\Stack*{\SFQ_{D_2}}.
\end{align}
Therefore, we have a modified evaluation map by projecting to the first factor:
\begin{align}
    \ev^1 \colon (\cM_X)_{|U}\longto \Stack*{\SFQ_{D_1}},
\end{align}
which is a map relative to \(X_{d_1}\).
The following is an analogue of \cite[Theorem~6.10.10]{Wa25} in this
simplified setting:
\begin{proposition}
    \label[proposition]{prop:strong_local_model_of_singularity}
    Suppose \(\bM\) is the universal monoid, and
    \begin{align}
        d_2> 2\abs{\bW}(n-1)(g_X-1),
    \end{align}
    then \(\ev^1\) is formally smooth, and a similar result holds for \(\cM_X'\)
    as well.
\end{proposition}
\begin{proof}
    The proof is a modified version of that of
    \Cref{prop:weak_local_model_of_singularity} and we sketch the gist of it.
    The problem with \Cref{prop:weak_local_model_of_singularity} is that when
    \(m\) is not very \((G,\delta_a)\)-ample, the cohomology group
    \eqref{eqn:local_model_of_singularity_coh_obstruction} may not vanish,
    mainly due to lack of sufficient ampleness. By considering \(\ev^1\) instead
    of \(\ev\), we effectively are making a sheaf \(\cF\) of which
    \begin{align}
        \coker(\ad(E)\to \symup{T}_m)
    \end{align}
    is a subsheaf and is ``\(d_2\)-more ample'' than the latter. Similarly, we
    also have a sheaf \(\cF_{\bG'}\) containing
    \begin{align}
        \coker(\Lie{\bG'}\x^G E\to \symup{\bT}_m).
    \end{align}
    A subtle analysis (see \cite[\S~7.2, \S\S~7.3.9--7.3.13]{Wa25}) involving the
    cameral curves and \(\bW\)-equivariance
    shows that when \(d_2\) is larger than the given threshold,
    the first cohomology of \(\cF_{\bG'}\) vanishes, and so by taking
    Galois invariants, the same is true for \(\cF\). As a result, \(\ev^1\) is
    formally smooth, and the argument for \(\cM_X'\) is the same.
\end{proof}

\begin{remark}
    The assumption that \(\bM\) is the universal monoid is not
    essential. The result remains true (with a slightly modified condition for
    \(d_2\) depending on \(\bM\)) as long as the boundary divisor \(D_2\) is
    ``good'' in the sense of \cite[Definition~6.10.8]{Wa25}.
\end{remark}

\subsubsection{}
Recall from \cite[\S~5.3]{Wa25} that we have a truncated stack
\(\Stack*{\SFQ_X'}_N\) as well as truncated
evaluation map \(\ev_N'\), where \(N\)
is a sufficiently large number (depending on the degree of the boundary
divisor). The \(N\)-truncation here serves a
purely technical purpose so that the intersection complexes are well-defined on
\(\Stack*{\SFQ_X'}_N\), because the stack \(\Stack*{\SFQ_X'}\) has
infinite-dimensional automorphism groups.

For the symmetric space case, the truncation is more involved since it can only
be defined with a ``one-sided'' action on \(\FRM\), which does not exist. It is
the same issue we faced in \Cref{ssub:symmetric_equiv_affGr}, and the workaround
is essentially the same: we first pass to an \'etale cover \(U\) of \(\cB_X\) so that
\(\Stack*{\SFQ_X}_{|U,N}\) is defined, then we take the constant sheaf with a
predetermined shift (related to the boundary divisor) on the big-cell locus, and
do intermediate extension on \(\Stack*{\SFQ_X}_{|U,N}\) (which makes sense
because the latter is finite dimensional); finally we pull this sheaf back to
\(\Stack*{\SFQ_X}_{|U}\) and now it descends to \(\Stack*{\SFQ_X}\).

We denote these ``intersection complex'' on \(\Stack*{\SFQ_X}\)
(resp.~\(\Stack*{\SFQ_X'}\)) by \(\IC_{\Stack*{\SFQ_X}}\)
(resp.~\(\IC_{\Stack*{\SFQ_X'}}\)). Now the \(N\)-truncation has served its
purpose and we forget about it. The following theorem is a corollary of
\Cref{prop:weak_local_model_of_singularity,prop:strong_local_model_of_singularity}.

\begin{theorem}
    \label[theorem]{thm:topological_local_model_of_singularity}
    There is an open substack \(\cA_X^\dagger\subset\cA_X^\heartsuit\) such that
    the \emph{dimension} of its complement \(\cA_X^\heartsuit-\cA_X^\dagger\) is
    bounded by a number depending only on \(X\) and \(\bM\), over which we have
    isomorphisms up to cohomological shifts and Tate twists:
    \begin{align}
        \ev^{\dagger*}\IC_{\Stack*{\SFQ_X}}&\cong\IC_{\cM_X^\dagger},\\
        \ev^{\prime\dagger*}\IC_{\Stack*{\SFQ_X'}}&\cong\IC_{\cM_X^{\prime\dagger}}.
    \end{align}
\end{theorem}



\subsection{The simple locus} 
\label{sub:the_simple_locus}

As the last part of this appendix, we discuss an open locus
\(\cA_X^{\SIM}\) of \(\cA_X^\heartsuit\) that is particularly nice. We shall
call it the \notion{simple locus}. It is closely related to the notions of
elliptic locus or anisotropic locus. In fact, it is the locus where the
pullbacks of mHiggs bundles to \(X'\) are elliptic as \(\bG\)-mHiggs bundles.

\subsubsection{}
Let
\begin{align}
    h_{X'}\colon\cM_{X'}\to\cA_{X'}
\end{align}
be the mH-fibration associated with stack \(\Stack*{\bM/\bG\x \bZ_\bM}\) and
curve \(X'\). Then we have pullback functors that fit into the commutative
diagram
\begin{equation}
    \begin{tikzcd}
        \cM_X \ar[r]\ar[d] & \cM_{X'} \ar[d] & \cM_X'\ar[l]\ar[d]\\
        \cA_X \ar[r] \ar[d]& \cA_{X'}\ar[d] & \cA_X\ar[l]\ar[d]\\
        \cB_X \ar[r] & \cB_{X'} & \cB_X\ar[l]
    \end{tikzcd}.
\end{equation}

\subsubsection{}
Define the \notion{simple locus} \(\cA_X^\SIM\subset\cA_X^\heartsuit\) to be the
preimage of the \(\bG\)-elliptic locus in \(\cA_{X'}^\heartsuit\), then it is an
open subset of \(\cA_X^\heartsuit\). We also define the \notion{smooth locus}
\(\cA_X^\diamondsuit\) to be where the image of \(X\) intersects with the
extended discriminant divisor transversally. By \cite[Lemma~8.2.8]{Wa25},
\(\cA_{X'}^\diamondsuit\) is contained in the
\(\bG\)-elliptic locus, and it is clear that
\(\cA_X^\diamondsuit\) coincides with the preimage of \(\cA_{X'}^\diamondsuit\).
Using the same argument as in \cite[Proposition~6.3.13]{Wa25},
we see that \(\cA_X^\diamondsuit\) is non-empty, and so \(\cA_X^\SIM\) is also
non-empty.

\subsubsection{}
The \(\bG\)-elliptic locus in \(\cA_{X'}\) is studied in
\cite[\S\S~8.2, 9.2]{Wa25} by
studying the monodromy of the cameral cover. Since \(\bG\) is split, the
discussions in \textit{loc. cit.} can be simplified in that we no longer
need an \(\Out(\bG)\)-torsor over \(X'\). The upshot is that the
\(\bG\)-elliptic locus consisting of points \(a\) such that the \(\bW\)-cameral
cover
\begin{align}
    \tilde{X}_a'\longto X'
\end{align}
does not contain a subcovering
\begin{align}
    \tilde{X}_{a,I}'\subset \tilde{X}_a'\longto X'
\end{align}
that is a finite flat \(\bW_I\)-cover for some proper subset \(I\subset\SimRts\)
of the simple roots. The following result is implicit in \cite{Wa25}:

\begin{lemma}
    If \(m\in\cM_{X'}(\bar{k})\) is a point lying over a
    \(\bG\)-elliptic \(a\), then it does not admit a parabolic reduction.
\end{lemma}
\begin{proof}
    If \(m\) admits a parabolic reduction, then the discussion in
    \cite[\S~9.2]{Wa25} essentially shows that \(a\)
    must lie in the image of the mH-base of some Levi monoid, which produces
    the subcovering \(\tilde{X}_{a,I}'\to X'\) for some \(I\subsetneq\SimRts\).
\end{proof}

\begin{corollary}
    \label[corollary]{cor:simple_locus_no_reduction}
    Let \(m\in\cM_{X'}(\bar{k})\) be a \(\bG\)-elliptic point, and
    \(K=\bar{k}(X')\) the geometric function field
    of \(X'\). Let \(V_K\) be the vector space induced by \(m\) and let
    \(x_K\) be the mHiggs field. Then there does not exist a proper subspace
    that is stabilized by the image of \(x_K\) in \(\End_K(V_K)\). The same is
    true if we replace \(V_K\) by \(V_K^\vee\).
\end{corollary}
\begin{proof}
    A proper subspace in \(V_K\) corresponds to a parabolic subgroup
    \(P_K\subset\GL(V_K)\), which extends to a parabolic reduction of the
    \(\bG\)-bundle part of \(m\) over
    \(\breve{X}'\) by properness of the moduli space of partial flags (and that
    \(X'\) is a smooth curve). Since generically the parabolic reduction is
    stable under \(x_K\), it must be stable under the mHiggs field over the
    whole \(X'\) by taking the closure. This means that \(m\) admits a parabolic
    reduction, hence cannot be \(\bG\)-elliptic. The case for \(V_K^\vee\) is
    the same because a partial flag in \(V_K\) corresponds bijectively to one
    in \(V_K^\vee\).
\end{proof}

\begin{remark}
    The name ``simple locus'' corresponds to that the vector bundle
    attached to the mHiggs bundle is a ``simple module'' under the ``action'' of the
    mHiggs field.
\end{remark}

\subsubsection{}
Let \(a\in\cA_X^\heartsuit(\bar{k})\), then we have a Cartesian diagram
\begin{equation}
    \begin{tikzcd}
        \tilde{X}_a' \ar[r]\ar[d] & \tilde{X}_a \ar[d]\\
        X' \ar[r] & X
    \end{tikzcd},
\end{equation}
where top horizontal map is \(\bW=W\)-equivariant.
If \(a\not\in\cA_X^\SIM\), then we can find \(I\subsetneq \SimRts\) and a
\(\bW_I\)-subcover \(\tilde{X}_{a,I}'\to X'\). Moreover, since \(a\) comes from
a point in \(\cA_X\), we see that \(I\) must be
stable under the action of \(\sigma\) through \(\Out(\bG)\).

On the other hand, \(I\) defines a twisted Levi subgroup \(G_I\)
(resp.~\(G_I'\)) of \(G\) (resp.~\(G'\)). Although \cite{Wa25} only
explained the construction of monoids associated with Levi subgroups or
endoscopic groups, the same process can be used to produce a monoidal symmetric
space (resp.~monoid) for \(G_I\) (resp.~\(G_I'\)) as well. Here we need the fact
that \(I\) is \(\sigma\)-stable.
Alternatively (and perhaps more economically), let \(\bM_I\) be the Levi monoid
associated with \(\bM\) and \(I\), then we may define the Levi monoidal
symmetric space to be
\begin{align}
    \FRM_I\defeq\OGT_*\bM_I^{\iota\sigma},
\end{align}
as well as the Levi monoid
\begin{align}
    \FRM_I'\defeq\OGT_*\bM_I^{\tau\sigma}.
\end{align}
Both are very flat in the sense that their abelianization maps are flat with
integral fibers, so the mH-fibrations associated with them fall into the
framework we already studied in this section.

Let \(\cA_{X,I}\) (resp.~\(\cA_{X,I}'\)) be the mH-base of such mH-fibration.
Similar to the case of endoscopic groups in \cite{Wa25}, there is in
general no map from \(\cA_{X,I}\) to \(\cA_X\), because the center
\(Z_{\FRM_I}\) of \(\FRM_I^\x\) does not map into \(Z_\FRM\). Instead, we
replace \(Z_{\FRM_I}\) by the preimage of \(Z_\FRM\) in \(Z_{\FRM_I}\), and the
resulting mH-fibration is simply a pullback of the usual one. Denote this new
mH-base by \(\cA_X^I\), then we have a canonical map
\begin{align}
    \nu_I\colon\cA_X^I\longto\cA_X,
\end{align}
which is finite and unramified by the same argument as in
\cite[Proposition~9.2.4]{Wa25}.
Similarly, we have another finite unramified map for the unitary case
\begin{align}
    \nu_I'\colon\cA_X^{I\prime}\longto\cA_X.
\end{align}

The image of \(\tilde{X}_{a,I}'\)
in \(\tilde{X}_a\) is a \(W_I\)-subcover of \(\tilde{X}_a\) over \(X\), through
which we may deduce that \(a\) comes from both a point in \(\cA_X^I\) and a
point in \(\cA_X^{I\prime}\). In particular, we have
\begin{align}
    \nu_I\bigl(\cA_X^I\bigr)\cap\cA_X^\heartsuit
    =\nu_I'\bigl(\cA_X^{I\prime}\bigr)\cap\cA_X^\heartsuit.
\end{align}

\subsubsection{}
The codimension of \(\nu_I\bigl(\cA_X^I\bigr)\cap\cA_X^\heartsuit\) in
\(\cA_X^\heartsuit\) can be estimated using the same argument as in
\cite[\S~6.11]{Wa25} for endoscopic groups, because the argument does not use
properties specific to endoscopic groups (just as it already works for the Levi
case in \textit{loc.~cit.}). A concrete formula can be given by defining and
using the \notion{resultant divisor}, but for brevity we will only state a
coarse corollary below (the proof is the same as in
\cite[Proposition~9.2.7]{Wa25}):

\begin{proposition}
    Let \(A\subset\cA_X^\heartsuit\) be an irreducible component that is very
    \((G,N)\)-ample (where \(N\in\bbN\)). Then the codimension of \(A-A^\SIM\)
    in \(A\) tends to \(\infty\) as \(N\to\infty\).
\end{proposition}

\subsubsection{}
\label{ssub:final_subtlety}
Lastly, we have product formulae for both \(\cM_a\) and \(\cM_a'\) for any
\(a\in\cA_X^\SIM\). The proof is exactly the same as
\cite[Proposition~6.9.1]{Wa25}. A subtlety here is that if one wishes to make
the product formula hold over \(k\) rather than \(\bar{k}\) for \(\cM_a'\), one has to use the
modified MASF introduced in \Cref{sub:affine_jacquet_rallis_fibers} (denoted by
\(\cM_v^{\prime\prime\Hit}(a)\) therein) instead, and similarly for the
symmetric side, due to potential non-existence of the ``base point''. In any case, it has no
influence on purely geometric arguments so we omit it here.
We leave the details to the reader and only record the following
equidimensional result:
\begin{proposition}
    \label[proposition]{prop:usual_mH_fiber_equidimensional}
    The fibers \(\cM_a\) and \(\cM_a'\) for any \(a\in\cA_X^\SIM\) are
    equidimensional.
\end{proposition}
\begin{proof}
    It follows from the product forumlae and
    \cite[Theorem~1.2.1]{Ch19} (see also \cite[Theorem~4.2.1]{Wa25}).
    The subtlety mentioned in \Cref{ssub:final_subtlety} plays no role since we
    only need the product formula over \(\bar{k}\).
\end{proof}



\addcontentsline{toc}{section}{References}
\bibliographystyle{halpha}
\bibliography{main}

\end{document}